\documentclass[onecolumn,secnumarabic,nobibnotes,superscriptaddress,aps]{revtex4-2}

\usepackage{times}
\usepackage{color, xcolor, colortbl}
\usepackage{graphicx}
\usepackage{epstopdf}
\ifpdf
\DeclareGraphicsExtensions{.eps,.pdf,.png,.jpg}
\else
\DeclareGraphicsExtensions{.eps}
\fi

\usepackage{geometry}
\usepackage{amsmath,amssymb, amsthm}
\usepackage{algorithm}
\usepackage{algorithmic}
\usepackage{bm}
\usepackage[caption=false]{subfig}
\usepackage[title]{appendix}
\usepackage{multirow}
\usepackage{braket}
\usepackage[english]{babel}
\usepackage{hyperref}
\usepackage[capitalize]{cleveref}
\usepackage{natbib}
\setcitestyle{square,numbers}
\usepackage{adjustbox}
\usepackage{xspace}
\usepackage{comment}
\usepackage{booktabs}
\usepackage{makecell}
\usepackage{threeparttable}

\newcommand{\bvec}[1]{\mathbf{#1}}

\newcommand{\valpha}{ {\bm{\alpha}} }
\newcommand{\vbeta}{ {\bm{\beta}} }

\newcommand{\vk}{\bvec{k}}

\newcommand{\vn}{\bvec{n}}

\newcommand{\vq}{\bvec{q}}
\newcommand{\vrr}{\bvec{r}}     

\newcommand{\vv}{\bvec{v}}

\newcommand{\vx}{\bvec{x}}
\newcommand{\vy}{\bvec{y}}
\newcommand{\vz}{\bvec{z}}

\newcommand{\vG}{\bvec{G}}

\newcommand{\vR}{\bvec{R}}

\newcommand{\diag}{\operatorname{diag}}

\newcommand{\I}{\mathrm{i}}

\newcommand{\wt}[1]{\widetilde{#1}}

\newcommand{\abs}[1]{\left\lvert#1\right\rvert}

\newcommand{\ud}{\,\mathrm{d}}
\newcommand{\Or}{\mathcal{O}}

\newcommand{\RR}{\mathbb{R}}

\newcommand{\ZZ}{\mathbb{Z}}

\newtheorem{thm}{\protect\theoremname}
\theoremstyle{plain}
\newtheorem{lem}[thm]{\protect\lemmaname}
\theoremstyle{plain}
\newtheorem{rem}[thm]{\protect\remarkname}
\theoremstyle{plain}

\theoremstyle{plain}
\newtheorem{cor}[thm]{\protect\corollaryname}

\newtheorem{defn}[thm]{\protect\definitionname}

\providecommand{\definitionname}{Definition}
\providecommand{\assumptionname}{Assumption}
\providecommand{\corollaryname}{Corollary}
\providecommand{\lemmaname}{Lemma}
\providecommand{\propositionname}{Proposition}
\providecommand{\remarkname}{Remark}
\providecommand{\theoremname}{Theorem}
\providecommand{\problemname}{Problem}
\crefname{lem}{Lemma}{Lemmas} 

\numberwithin{figure}{section}
\numberwithin{table}{section}

\newcommand{\xsum}{\mathop{\sum\nolimits'}}

\crefname{equation}{Eq.}{Eqs.}
\crefformat{equation}{Eq.~#2(#1)#3}

\newcommand{\XX}[1]{\textcolor{blue}{[XX: #1]}}

\newcommand{\DeptMath}{Department of Mathematics, University of California, Berkeley, CA 94720, USA}
\newcommand{\LBLMath}{Applied Mathematics and Computational Research Division, Lawrence Berkeley National Laboratory, Berkeley, CA 94720, USA}

\begin{document}

\title{Finite-size effects in periodic coupled cluster calculations}
\author{Xin Xing}
\email{xxing@berkeley.edu}
\affiliation{\DeptMath}
\author{Lin Lin}
\email{linlin@math.berkeley.edu}
\affiliation{\DeptMath}
\affiliation{\LBLMath}

\begin{abstract}
We provide the first rigorous study of the finite-size error in the simplest and representative coupled cluster theory, namely the coupled cluster doubles (CCD) 
theory, for gapped periodic systems. 
Assuming that the CCD equations are solved using exact Hartree-Fock orbitals and orbital energies, we prove that the convergence rate of finite-size error scales as $\mathcal{O}(N_\mathbf{k}^{-\frac13})$, where $N_{\mathbf{k}}$ is the number of discretization point in the Brillouin zone and characterizes the system size.
Our analysis shows that the dominant error lies in the coupled cluster amplitude calculation, and the convergence of the finite-size error in energy calculations can be boosted to $\mathcal{O}(N_\mathbf{k}^{-1})$ with accurate amplitudes. 
This also provides the first proof of the scaling of the finite-size error in the third order M{\o}ller-Plesset perturbation theory (MP3) for periodic systems.
\end{abstract}

\maketitle

\tableofcontents

\maketitle 
\newpage

\section{Introduction}

In 1990, Kenneth Wilson wrote ``\textit{ab initio} quantum chemistry is an emerging computational area that is fifty years ahead of lattice gauge theory, a principal competitor for supercomputer time, and a rich source of new ideas and new approaches to the computation of many fermion systems.''~\cite{Wilson1990}
Coupled cluster (CC) theory is one of the most advanced \textit{ab initio} quantum chemistry methods, and the coupled cluster singles and doubles with perturbative triples (CCSD(T)) is often referred to as the ``gold standard'' of molecular quantum chemistry.
Compared to the success in molecular systems, applications of the CC theory to periodic systems (i.e., bulk solids and other extended systems)~\cite{HirataPodeszwaTobitaEtAl2004,BartlettMusial2007,ZhangGruneis2019} have been much more limited. Nonetheless, thanks to the combined improvements of computational power and numerical algorithms in the past few years,  periodic CC calculations have been increasingly performed routinely for ground state and excited state properties of condensed matter systems~\cite{ZhangGruneis2019, GruberLiaoTsatsoulisEtAl2018,McClainSunChanEtAl2017}. 

Unlike CC calculations for molecular systems, properties of periodic systems need to be evaluated in the thermodynamic limit (TDL). The TDL can be approached by employing a large supercell containing $N_\vk$ unit cells, but this approach does not take advantage of the translational symmetry and is thus increasingly inefficient as the system size grows. A more efficient approach is to discretize the Brillouin zone (BZ) of one unit cell into $N_\vk$ grid points, and the most widely used discretization method is a uniform grid called the Monkhorst-Pack grid~\cite{MonkhorstPack1976}.  The result in the TDL is given by the limit $N_{\vk}\to \infty$.  Due to the singularity caused by the long range Coulomb interaction, the convergence of the energy and other physical properties towards the TDL is often slow and follows a low-order power law. 
It is therefore important to understand the precise scaling of the finite-size effects in periodic CC calculations. Finite-size scaling analysis is the foundation for the power law extrapolation to calculate properties in the TDL, as well as  for more advanced finite-size error correction schemes ~\cite{ChiesaCeperleyMartinEtAl2006,LiaoGrueneis2016,GruberLiaoTsatsoulisEtAl2018,MihmMcIsaacShepherd2019}.

To our knowledge, this work is the first mathematical study of finite-size errors of periodic CC theories. We interpret the finite-size error as numerical quadrature error of a trapezoidal rule applied to certain singular integrals. Thus the main body of this work is (1) to analyze the singularity structure of quantities in CC theories, and (2) to bound the trapezoidal quadrature errors associated with these singular integrands. At first glace, the task (2) seems to be a classical problem in numerical analysis. We therefore emphasize that standard quadrature error analysis for smooth integrands (see e.g., ~\cite{Trefethen2019Approximation}) may produce an overly pessimistic upper bound of the finite-size error that does not decrease at all as $N_{\vk}$ increases. 
Our  quadrature error analysis adapts the Poisson summation formula in a new setting and is related to a recently developed trapezoidal rule-based quadrature analysis for certain singular integrals \cite{Izzo2022}. 
This provides tighter estimates and is more generally applicable than a previous quadrature analysis based on the partial Euler-Maclaurin formula for finite-size error studies \cite{XingLiLin2021}.

To simplify the presentation and the analysis, we focus on ground-state energy calculations in three-dimensional (3D) insulating systems using the simplest and representative CC theory, i.e., the coupled cluster doubles (CCD) theory. The Brillouin zone is discretized using a Monkhorst-Pack grid of size $N_{\vk}^{\frac13}\times N_{\vk}^{\frac13} \times N_{\vk}^{\frac13}$. The core of CC theories is the amplitude equation, which is a nonlinear system often solved by iterative methods. In particular, when the amplitude equation is solved using a fixed point iteration with a zero initial guess, it can systematically generate a set of perturbative terms as in the M{\o}ller-Plesset perturbation theory~\cite{ShavittBartlett2009} that can be represented using Feynman diagrams.
In practice, the number of iterations needs to be truncated at some number $n$, and we refer to the resulting scheme as CCD($n$).
We assume that the amplitude equations are solved with exact (or in practice, sufficiently accurate) Hartree-Fock orbitals and orbital energies in the TDL (see~\cref{appendix:mp3ccd}).
The main result of this paper is that under such assumptions, the convergence rate of CCD$(n)$ (for any fixed $n$) is $\Or\left(N_{\vk}^{-\frac13}\right)$ (\cref{thm:error_ccd}). 

It is worth noting that previous numerical studies~\cite{LiaoGrueneis2016,GruberLiaoTsatsoulisEtAl2018} have suggested that under different assumptions, the finite-size error of the CCD energy calculation can scale as  $\Or(N_\vk^{-1})$. 
A possible origin of this discrepancy will be discussed at the end of the manuscript in \cref{sec:discuss}. 
The restriction of discussion to CCD$(n)$ is a technical one. 
Under additional assumptions, the same convergence rate can be established for the converged solution of CCD with $n\rightarrow\infty$ (\cref{thm:error_ccd_converge}). 
Many finite-size error correction methods work under the assumption that the error in the CC double amplitude is small, and the finite-size error mainly comes from the evaluation of the total energy using the CC double amplitude~\cite{GruberLiaoTsatsoulisEtAl2018}. 
Our analysis reveals that the opposite is true in general: most of the finite-size error is in fact in the CC double amplitude, which is responsible for the $\Or\left(N_{\vk}^{-\frac13}\right)$ convergence rate. On the other hand, with accurate CC double amplitudes, the convergence rate of energy calculations could be improved to $\Or\left(N_{\vk}^{-1}\right)$ without any further finite-size corrections. 

The finite-size error analysis of many quantum chemistry methods can be reduced to certain quadrature error analysis. This perspective has recently provided the first unified finite-size error analysis for the periodic Hartree-Fock and the second order M{\o}ller-Plesset perturbation theory (MP2)~\cite{XingLiLin2021}, and similar analysis can be carried out for more complex theories such as the random phase approximation (RPA) and the second order screened exchange (SOSEX)~\cite{XingLin2022}. 
The commonality of these theories (beyond the Hartree-Fock level) is that they only include certain perturbative terms (referred to as the ``particle-hole'' Feynman diagrams, see the main text for the explanation), and the associated integrand singularities are relatively weak. As a result, for ground-state energy calculations in 3D insulating systems,  
the finite-size errors of MP2, RPA, SOSEX all scale as $\Or\left(N_{\vk}^{-1}\right)$. 
Starting from the third order M{\o}ller-Plesset perturbation theory (MP3), other perturbative terms (referred to as ``particle-particle'' and ``hole-hole'' diagrams) must be taken into account, and the singularities in these terms are much stronger. Our new techniques can be used to analyze the quadrature error associated with these terms.
Since the MP3 diagrams form a subset of the CCD diagrams, our result also gives the first proof that the finite-size error of the MP3 energy scales as $\Or\left(N_{\vk}^{-\frac13}\right)$ (\cref{thm:error_mp3}). 

\section{Background}
Denote a unit cell as $\Omega$ and its Bravais lattice as $\mathbb{L}$.  
Denote the corresponding reciprocal Brillouin zone and lattice as $\Omega^*$ and $\mathbb{L}^*$. 
Consider a mean-field (Hartree-Fock) calculation with a Monkhorst-Pack mesh $\mathcal{K}$ which is a uniform mesh of size $N_\vk$ discretizing $\Omega^*$. 
Each molecular orbital from the calculation can be represented as
\begin{equation*}
        \psi_{n\vk}(\vrr) = \dfrac{1}{\sqrt{N_\vk}} e^{\I \vk\cdot\vrr} u_{n\vk}(\vrr)=
        \frac{1}{\abs{\Omega}\sqrt{N_{\vk}}} \sum_{\mathbf{G}\in\mathbb{L}^*} \hat{u}_{n\vk}(\mathbf{G}) e^{\I (\vk+\mathbf{G}) \cdot \mathbf{r}},
\end{equation*}
where $n$ is a band index and $u_{n\vk}$ is periodic with respect to the unit cell. 
As is common in the chemistry literature, $n\in\{i,j\}$ refers to an occupied (or ``hole'') orbital and $n\in\{a,b\}$ refers to an unoccupied 
(or ``particle'') orbital. 
Throughout this paper, we use the following \textit{normalized} electron repulsion integral (ERI):
\begin{equation}\label{eqn:eri}
        \braket{n_1\vk_1,n_2\vk_2|n_3\vk_3,n_4\vk_4}
        = 
        \frac{4\pi}{\abs{\Omega}} \xsum_{\vG\in\mathbb{L}^*}
        \frac{1}{\abs{\vq+\vG}^2}  
        \hat{\varrho}_{n_1\vk_1,n_3(\vk_1 + \vq)}(\mathbf{G}) \hat{\varrho}_{n_2\vk_2,n_4(\vk_2 - \vq)}(-\mathbf{G}),
\end{equation}
where $\vq = \vk_3 - \vk_1$ is the momentum transfer vector, the crystal momentum conservation $\vk_1 + \vk_2 -\vk_3 - \vk_4 \in \mathbb{L}^*$ is assumed implicitly, 
and $\hat{\varrho}_{n'\vk', n\vk}(\vG) = \braket{\psi_{n'\vk'} | e^{\I(\vk' - \vk - \vG)\cdot \vrr} | \psi_{n\vk}}$ is the pair product.
This normalized ERI (and the normalized amplitude below) is used mainly for better illustration of the connection between various calculations in 
the finite and the TDL cases, and will introduce extra $\frac{1}{N_\vk}$ factors in the energy and amplitude formulations compared to the standard ones in the literature. 

In this paper, we consider insulating systems with a direct gap between occupied and virtual bands.
In addition, we assume that the orbitals and orbital energies can be exactly evaluated at any $\vk$ point and the virtual bands are truncated (i.e., only a finite 
number of virtual bands are included in the calculations) which corresponds to calculations using a fixed basis set. 
Lastly, we assume that with a proper gauge, both $\psi_{n\vk}(\vrr)$ and $\varepsilon_{n\vk}$ are smooth and periodic with respect to $\vk\in\Omega^*$. For systems free of topological obstructions~\cite{BrouderPanatiCalandraEtAl2007,MonacoPanatiPisanteEtAl2018}, these conditions may be replaced by weaker conditions using techniques based on Green's functions. However, such a treatment can introduce a considerable amount of overhead to the presentation, and therefore we adopt the stronger but simpler assumptions on the orbitals and orbital energies as stated above. 

All the finite order energy diagrams in CCD share the common  form
\begin{align}
        E_\#^{N_\vk} 
        & = \frac{1}{N_{\vk}^3}
        \sum_{\vk_i\vk_j \vk_a\in\mathcal{K}}\sum_{i j a b}
        \left(
        2\braket{i\vk_i,j \vk_j |a\vk_a, b\vk_b}
        -
        \braket{i\vk_i,j \vk_j |b\vk_b, a\vk_a}
        \right)
        T_{ijab}^{\#,N_\vk}(\vk_i, \vk_j, \vk_a)
        \nonumber \\
        & = \frac{1}{N_{\vk}^3}
        \sum_{\vk_i\vk_j \vk_a\in\mathcal{K}}\sum_{i j a b}
        W_{ijab}(\vk_i, \vk_j, \vk_a)
        T_{ijab}^{\#,N_\vk}(\vk_i, \vk_j, \vk_a),
        \label{eqn:energy_finite}
\end{align}
where $W_{ijab}(\vk_i, \vk_j, \vk_a)$ is the antisymmetrized ERI and the \textit{normalized} double amplitude $T_{ijab}^{\#,N_\vk}(\vk_i, \vk_j, \vk_a)$ is different in each term (annotated by $\#$).
For example, the double amplitudes in MP2 and MP3-4h2p (reads ``4 hole 2 particle'', because $i,j,k,l$ are hole indices and $a,b$ are particle indices) energies are defined as 
\begin{align}
        T_{ijab}^{\text{MP2},N_\vk}(\vk_i, \vk_j, \vk_a) 
        & = \dfrac{1}{\varepsilon_{i\vk_i, j\vk_j}^{a\vk_a, b\vk_b}}\braket{a\vk_a, b\vk_b | i\vk_i, j\vk_j},
        \label{eqn:amplitude_mp2}
        \\
        T_{ijab}^{\text{MP3-4h2p},N_\vk}(\vk_i, \vk_j, \vk_a)
        &       =       
        \dfrac{1}{N_\vk}\sum_{\vk_k\in\mathcal{K}}
        \sum_{kl}
        \dfrac{1}{\varepsilon_{i\vk_i,j\vk_j}^{a\vk_a,b\vk_b}}
        \braket{k\vk_k, l\vk_l| i\vk_i,j \vk_j}
        \dfrac{\braket{a\vk_a, b\vk_b|k\vk_k,l \vk_l}}{\varepsilon_{k\vk_k,l\vk_l}^{a\vk_a,b\vk_b}}. 
        \label{eqn:amplitude_mp3_4h2p}
\end{align}
Unlike the MP2 energy which only involves interactions between particle-hole pairs $(i\vk_i,a\vk_a)$ and $(j\vk_j,b\vk_b)$, the MP3-4h2p energy 
involves interactions between hole-hole pairs $(i\vk_i,k\vk_k)$ and $(j\vk_j,l\vk_l)$. 
It is such interaction terms involving particle-particle or hole-hole pairs in MP3 and CCD that lead to considerable difficulties 
in the finite-size error analysis, compared to the existing analysis for MP2.

The CCD correlation energy is also defined as \cref{eqn:energy_finite} with the double amplitude being an infinite sum of the double amplitudes 
from a subset of finite order perturbation energies. 
Specifically, the double amplitude in CCD calculation is defined as the solution of a nonlinear amplitude equation which consists of constant, 
linear, and quadratic terms. The exact definition of the amplitude equation is provided in \cref{eqn:amplitude_ccd}.

A common practice in CCD calculation is to solve the amplitude equation approximately using fixed point iterations with a zero initial guess which
is equivalent to a quasi-Newton method \cite{Schneider2009}.
After $n$ iterations, we refer to the approximate amplitude as the CCD$(n)$ amplitude and the resulting approximate energy as the CCD$(n)$
energy.  
At the $(n+1)$th iteration, plugging the CCD$(n)$ amplitude from the previous iteration into the right hand side of the amplitude equation in 
\cref{eqn:amplitude_ccd} gives the CCD$(n+1)$ amplitude. 
If the mean-field calculation gives a good reference wavefunction and the direct gap between occupied and virtual bands is sufficiently large, 
this iteration converges and CCD$(n)$ converges to CCD \cite{Schneider2009}. 
The CCD($n$) scheme is directly related to the M{\o}ller-Plesset perturbation theory~\cite{ShavittBartlett2009}. 
For example, MP2 can be identified with CCD$(1)$, and CCD$(2)$ contains all terms in MP2 and MP3, as well as a subset of terms in MP4. 

In the TDL with $\mathcal{K}$ converging to $\Omega^*$, the correlation energy in \cref{eqn:energy_finite} converges to an integral as 
\begin{equation}\label{eqn:energy_tdl}
        E_\#^\text{TDL} =       \frac{1}{|\Omega^*|^3}
        \int_{(\Omega^*)^{\times 3}}\ud\vk_i \ud\vk_j \ud\vk_a
        \sum_{i j a b}
        W_{ijab}(\vk_i, \vk_j, \vk_a)
        T_{ijab}^{\#,\text{TDL}}(\vk_i, \vk_j, \vk_a),
\end{equation}
where we note that the double amplitude is converged as well (indicated by its superscript ``TDL''). 
For each finite order perturbation energy in CCD except MP2, its double amplitude also converges to an integral in the TDL. 
For example, the double amplitude \cref{eqn:amplitude_mp3_4h2p} in MP3-4h2p term converges to 
\begin{align}
        \label{eqn:amplitude_mp3_4h2p_tdl}
        T_{ijab}^{\text{MP3-4h2p},\text{TDL}}(\vk_i, \vk_j, \vk_a)
        &       =       
        \dfrac{1}{|\Omega^*|}\int_{\Omega^*}\ud\vk_k \sum_{kl}
        \dfrac{1}{\varepsilon_{i\vk_i,j\vk_j}^{a\vk_a,b\vk_b}}
        \braket{k\vk_k, l\vk_l| i\vk_i,j \vk_j}
        \dfrac{\braket{a\vk_a, b\vk_b|k\vk_k,l \vk_l}}{\varepsilon_{k\vk_k,l\vk_l}^{a\vk_a,b\vk_b}}. 
\end{align}
Since CCD$(n)$ consists of a finite number of perturbation terms, its double amplitude converges to a sum of many integral terms in 
the TDL. For more background information, we refer readers to \cref{appendix:mp3ccd}.  

\section{Statement of main results}\label{sec:main_result}
Comparing $E_\#^{N_\vk}$ in \cref{eqn:energy_finite} and $E_\#^\text{TDL}$ in \cref{eqn:energy_tdl}, we could split the finite-size error of any term in CCD 
calculation as 
\begin{align}
        &E_\#^\text{TDL} - E_\#^{N_\vk} 
         = \left(\frac{1}{|\Omega^*|^3}\int_{(\Omega^*)^{\times 3}}\ud\vk_i\ud\vk_j\ud\vk_a -       \frac{1}{N_{\vk}^3}
        \sum_{\vk_i\vk_j \vk_a\in\mathcal{K}}\right) \sum_{ijab} 
        \left(
        W_{ijab}T_{ijab}^{\#,\text{TDL}}\right)(\vk_i,\vk_j,\vk_a)
        \nonumber \\
        &\hspace{5em} +         \frac{1}{N_{\vk}^3}
        \sum_{\vk_i\vk_j \vk_a\in\mathcal{K}} \sum_{ijab}W_{ijab}(\vk_i,\vk_j,\vk_a)
        \left(
        T^{\#,\text{TDL}}_{ijab}(\vk_i, \vk_j, \vk_a) - T^{\#, N_\vk}_{ijab}(\vk_i, \vk_j, \vk_a)
        \right),
        \label{eqn:error_splitting}
\end{align}
where the two parts can be interpreted respectively as the finite-size errors in the \textbf{energy calculation using exact amplitudes} and in the \textbf{amplitude calculations}.
Analyzing these two parts separately, we provide a rigorous estimate of the finite-size error in CCD$(n)$ calculations.

\begin{thm}[Error of CCD($n$)]\label{thm:error_ccd}
        In CCD$(n)$ calculation with any $n>0$, the finite-size errors in energy calculation using exact amplitudes and in amplitude calculations
        can be estimated as 
        \begin{align*}
                \left|
                \left(\frac{1}{|\Omega^*|^3}\int_{(\Omega^*)^{\times 3}}\ud\vk_i\ud\vk_j\ud\vk_a -       \frac{1}{N_{\vk}^3}
                \sum_{\vk_i\vk_j \vk_a\in\mathcal{K}}\right) \sum_{ijab} 
                \left(
                W_{ijab}T_{ijab}^{\text{CCD}(n),\text{TDL}}\right)(\vk_i,\vk_j,\vk_a)
                \right|
                & = \Or(N_\vk^{-1}),
                \\
                \max_{ijab, \vk_i\vk_j\vk_a\in\mathcal{K}}
                \left|
                T^{{\text{CCD}(n)}, \text{TDL}}_{ijab}(\vk_i, \vk_j, \vk_a) - T^{{\text{CCD}(n)}, N_\vk}_{ijab}(\vk_i, \vk_j, \vk_a)
                \right|
                & = \Or(N_\vk^{-\frac13}).
        \end{align*}
        Combining these two estimates with \cref{eqn:error_splitting}, the overall finite-size error in CCD($n$) energy calculation is 
        \[
        \left|E_{\text{CCD}(n)}^\text{TDL} -    E_{\text{CCD}(n)}^{N_\vk}\right| = \Or(N_\vk^{-\frac13}).
        \]
\end{thm}

We remark that in MP2, there is no finite-size error in its amplitude calculation, i.e., 
$T^{\text{MP2},\text{TDL}}_{ijab} = T^{\text{MP2},N_\vk}_{ijab}$. 
As a result, $\left|E_\text{MP2}^\text{TDL} - E_{\text{MP2}}^{N_\vk}\right| = \Or(N_\vk^{-1})$ which recovers the result in~\cite{XingLiLin2021}.
Since all terms in MP3 energy are a subset of CCD($2$), the above results on CCD$(n)$ also provide a finite-size error analysis for MP3 energy 
calculation. 
\begin{cor}\label{thm:error_mp3}
        The finite-size error in MP2 calculations is $\Or\left(N_{\vk}^{-1}\right)$, and the finite-size error in MP3 calculations is $\Or\left(N_{\vk}^{-\frac13}\right)$.
\end{cor}

\cref{thm:error_ccd} provides the finite-size error estimates for CCD$(n)$ calculations that consist 
of finite number of perturbative terms (Feynman diagrams) in CCD, but not for the converged CCD calculation.
These estimates holds for any fixed number of iterations $n$ even when the iteration does not converge 
as $n\rightarrow \infty$, and the prefactors in these estimates depend on $n$. 
Under additional assumptions that can control the regularities of the iterates and guarantee the convergence of the fixed point iterations in the finite and the TDL cases, 
we show that the convergence rate of the finite-size error for the converged CCD energy calculation matches that of the CCD$(n)$ energy calculations.

\begin{cor}[Error of CCD]\label{thm:error_ccd_converge}
        Under additional assumptions, the finite-size error in the CCD energy calculation is $\Or(N_\vk^{-\frac13})$.
\end{cor}

\section{Proof for Theorem \ref{thm:error_ccd}}\label{sec:proof}
\subsection{Setup}
In CCD$(n)$ and all its included perturbation terms (e.g., MP2 and MP3), the double amplitudes computed in the finite case 
can be viewed as tensors 
\[
\left\{
T_{ijab}(\vk_i, \vk_j, \vk_a)
\right\}_{ijab, \vk_i\vk_j\vk_a\in\mathcal{K}} 
\in 
\mathbb{C}^{n_\text{occ}\times n_\text{occ} \times n_\text{vir} \times n_\text{vir}\times N_\vk\times N_\vk \times N_\vk},
\]
where we assume $n_\text{occ}$ occupied and $n_\text{vir}$ virtual bands. 
Meanwhile, the exact double amplitudes in the TDL can be viewed as a set of functions of $\vk_i,\vk_j,\vk_a \in \Omega^*$ indexed by band 
indices $i,j,a,b$. 
As shown later in \cref{lem:thm1_nonsmooth}, all these functions are in a function space $\mathbb{T}(\Omega^*)$ with special smoothness 
properties and the exact double amplitude can be described as
\[
\left\{
T_{ijab}(\vk_i, \vk_j, \vk_a)
\right\}_{ijab} 
\in 
\mathbb{T}(\Omega^*)^{n_\text{occ}\times n_\text{occ} \times n_\text{vir} \times n_\text{vir}}. 
\]
In CCD$(n)$ calculations with a finite $N_{\vk}$, the computed amplitude approximates the exact amplitude 
at $\mathcal{K}\times\mathcal{K}\times\mathcal{K}$.  
We define a map that evaluates the exact amplitude at this discrete mesh as
\begin{align*}
\mathcal{M}_{\mathcal{K}}:&\  
\mathbb{T}(\Omega^*)^{n_\text{occ}\times n_\text{occ} \times n_\text{vir} \times n_\text{vir}}
&& \longrightarrow
& \mathbb{C}^{n_\text{occ}\times n_\text{occ} \times n_\text{vir} \times n_\text{vir}\times N_\vk\times N_\vk \times N_\vk}
\\
& 
\left\{T_{ijab}(\vk_i, \vk_j, \vk_a)\right\}_{ijab}
&& \longrightarrow
& 
\left\{T_{ijab}(\vk_i, \vk_j, \vk_a)\right\}_{ijab, \vk_i\vk_j\vk_a\in\mathcal{K}}.
\end{align*}
In the following discussions, we use $T$ to refer to amplitude tensors in the finite case and $t$ to refer to amplitude functions 
in the TDL case. 
We focus on estimating the error in the amplitude calculation between $T$ and $t$ using the (entrywise) max norm: 
\begin{equation}
        \|T - \mathcal{M}_{\mathcal{K}}t\|_\infty = \max_{ijab, \vk_i\vk_j\vk_a\in\mathcal{K}} |T_{ijab}(\vk_i, \vk_j, \vk_a) - t_{ijab}(\vk_i, \vk_j, \vk_a)|.        
\end{equation}
Similarly, define $\|t\|_\infty = \max_{ijab} \|t_{ijab}\|_{L^{\infty}(\Omega^*\times\Omega^*\times\Omega^*)}$ and $\|\mathcal{M}_{\mathcal{K}}t\|_\infty \leqslant \|t\|_\infty$.

Define the two linear functionals that compute the correlation energy with a given double amplitude in the finite and the TDL cases, respectively, 
as 
\begin{align*}
        \mathcal{G}_{N_\vk}(T)
        & = 
        \frac{1}{N_{\vk}^3}
        \sum_{\vk_i\vk_j \vk_a\in\mathcal{K}}\sum_{i j a b}
        W_{ijab}(\vk_i, \vk_j, \vk_a)
        T_{ijab}(\vk_i, \vk_j, \vk_a),
        \\
        \mathcal{G}_\text{TDL}(t)
        & = 
        \frac{1}{|\Omega^*|^3}
        \int_{(\Omega^*)^{\times 3}}\ud\vk_i \ud\vk_j \ud\vk_a
        \sum_{i j a b}
        W_{ijab}(\vk_i, \vk_j, \vk_a)
        t_{ijab}(\vk_i, \vk_j, \vk_a).
\end{align*}
Furthermore, we denote the two mappings that define the fixed point iterations over the amplitude equations in the finite and 
the TDL cases, respectively, as 
\begin{align*}
        \mathcal{F}_{N_\vk}(T)
        &:
        \mathbb{C}^{n_\text{occ}\times n_\text{occ} \times n_\text{vir} \times n_\text{vir}\times N_\vk\times N_\vk \times N_\vk}
        \rightarrow
        \mathbb{C}^{n_\text{occ}\times n_\text{occ} \times n_\text{vir} \times n_\text{vir}\times N_\vk\times N_\vk \times N_\vk},
        \\
        \mathcal{F}_\text{TDL}(t)
        &: 
        \mathbb{T}(\Omega^*)^{n_\text{occ}\times n_\text{occ} \times n_\text{vir} \times n_\text{vir}}
        \rightarrow
        \mathbb{T}(\Omega^*)^{n_\text{occ}\times n_\text{occ} \times n_\text{vir} \times n_\text{vir}},
\end{align*}
which correspond to the right hand sides of \cref{eqn:amplitude_ccd} and \cref{eqn:amplitude_ccd_tdl} with all the concerned
$i,j,a,b,\vk_i,\vk_j,\vk_a$. 
One main technical result of this paper is to prove that the image of $\mathcal{F}_\text{TDL}$ is in $\mathbb{T}(\Omega^*)^{n_\text{occ}\times n_\text{occ} \times n_\text{vir} \times n_\text{vir}}$ (see \cref{lem:thm1_nonsmooth}).

Using these notations, the CCD$(n)$ energy calculation in the finite case can be formulated as 
\begin{equation}\label{eqn:fixed_point_finite}
        E_{\text{CCD}(n)}^{N_\vk} = \mathcal{G}_{N_\vk}(T_n)
        \quad 
        \text{with} \quad  
        \begin{array}{ll}
                T_m & = \mathcal{F}_{N_\vk}(T_{m-1}), m =1, 2, \ldots
                \\
                T_0 & = \bm{0} \in 
                \mathbb{C}^{n_\text{occ}\times n_\text{occ} \times n_\text{vir} \times n_\text{vir}\times N_\vk\times N_\vk \times N_\vk}
        \end{array},
\end{equation}
and in the TDL case can be formulated as 
\begin{equation}\label{eqn:fixed_point_tdl}
        E_{\text{CCD}(n)}^\text{TDL} = \mathcal{G}_\text{TDL}(t_n)
        \quad 
        \text{with}
        \quad
        \begin{array}{ll}
                t_m & = \mathcal{F}_\text{TDL}(t_{m-1}), m = 1, 2, \ldots
                \\
                t_0 & = \bm{0} \in 
                \mathbb{T}(\Omega^*)^{n_\text{occ}\times n_\text{occ} \times n_\text{vir} \times n_\text{vir}}
        \end{array}.
\end{equation}
To connect to the previous notations in \cref{sec:main_result}, we have 
\[
T_n = \{T_{ijab}^{\text{CCD}(n), N_\vk}(\vk_i,\vk_j,\vk_a)\}_{ijab,\vk_i,\vk_j,\vk_a\in\mathcal{K}}
\quad \text{and} \quad t_n = \{T_{ijab}^{\text{CCD}(n), \text{TDL}}(\vk_i,\vk_j,\vk_a)\}_{ijab}. 
\]
When the two fixed point iterations converge with respect to $n$, the corresponding CCD$(n)$ energies converge to the CCD energies in the finite and the TDL cases, respectively.

Lastly, we introduce the notations for the trapezoidal quadrature rule that will be used in the proof. 
Given an $n$-dimensional cubic domain $V$, we construct a uniform mesh $\mathcal{X}$ in $V$ by first partitioning $V$ into subdomains uniformly and 
then sampling one point in each subdomain with the same offset. 
The (generalized) trapezoidal rule for an integrand $f(\vx)$ in $V$ using $\mathcal{X}$ is defined as
\[
\mathcal{Q}_V(f, \mathcal{X}) = \dfrac{|V|}{|\mathcal{X}|}\sum_{\vx_i \in \mathcal{X}} f(\vx_i).
\]
Further denote the targeted exact integral and the corresponding quadrature error as 
\begin{equation*}
        \mathcal{I}_V(f) = \int_V f(\vx)\ud\vx, \qquad \mathcal{E}_{V}(f, \mathcal{X}) = \mathcal{I}_V(f) - \mathcal{Q}_V(f, \mathcal{X}). 
\end{equation*}
In the following analysis, we abuse the notation ``$C$'' to denote a generic constant that is independent of any concerned 
terms in the context unless otherwise specified. In other words, $f\leqslant C g$ is equivalent to $|f| = \Or(|g|)$. 

\subsection{Proof Outline}
In CCD$(n)$ calculation, the splitting of the finite-size error shown in \cref{eqn:error_splitting} can be written as 
\begin{align}
        \left|
                E_{\text{CCD}(n)}^\text{TDL} - E_{\text{CCD}(n)}^{N_\vk}
        \right|
        & = 
        \left|
                \mathcal{G}_\text{TDL}(t_n) - \mathcal{G}_{N_\vk}(T_n)
        \right|
        \nonumber\\
        & \leqslant 
        \left|
                \mathcal{G}_\text{TDL}(t_n) - \mathcal{G}_{N_\vk}(\mathcal{M}_{\mathcal{K}}t_n)
        \right|
        +
        \left|
                \mathcal{G}_{N_\vk}(\mathcal{M}_{\mathcal{K}}t_n) - \mathcal{G}_{N_\vk}(T_n)
        \right|
        \nonumber\\
        & \leqslant 
        \left|
                \mathcal{G}_\text{TDL}(t_n) - \mathcal{G}_{N_\vk}(\mathcal{M}_{\mathcal{K}}t_n)
        \right|
        +
        C \left\|\mathcal{M}_{\mathcal{K}}t_n - T_n\right\|_\infty,
        \label{eqn:energy_splitting}
\end{align}
where the last two terms can be interpreted as the error in the energy calculation 
using the exact CCD$(n)$ amplitude $t_n$ and the error in the amplitude calculation, respectively.
The last inequality uses the boundedness of the linear operator $\mathcal{G}_{N_\vk}$, i.e., 
\begin{align*}
        \left| 
                \mathcal{G}_{N_\vk}(T)
        \right|
        & \leqslant
        \frac{1}{N_{\vk}^3}
        \sum_{\vk_i\vk_j \vk_a\in\mathcal{K}}\sum_{i j a b}
        \left| 
                W_{ijab}(\vk_i, \vk_j, \vk_a)
                T_{ijab}(\vk_i, \vk_j, \vk_a)
        \right|
        \\
        & \leqslant 
        \frac{1}{N_{\vk}^3}
        \sum_{\vk_i\vk_j \vk_a\in\mathcal{K}}\sum_{i j a b}
        C
        \max_{ijab,\vk_i\vk_j\vk_c\in\mathcal{K}}
        \left| 
                T_{ijab}(\vk_i, \vk_j, \vk_a)
        \right|
        \leqslant C\|T\|_\infty.
\end{align*}
This uses the fact that $W_{ijab}(\vk_i,\vk_j,\vk_a) = \Or(1)$.

The error in the energy calculation using exact amplitude can be interpreted as a quadrature error 
\begin{equation}\label{eqn:error_energy_exact_t}
        \left|
                \mathcal{G}_\text{TDL}(t_n) - \mathcal{G}_{N_\vk}(\mathcal{M}_{\mathcal{K}}t_n)
        \right|
        = 
        \left|
                \dfrac{1}{|\Omega^*|^3}
                \mathcal{E}_{\Omega^*\times\Omega^*\times\Omega^*}
                \left(
                \sum_{ijab}
                \left(
                W_{ijab}[t_n]_{ijab}
                \right)(\vk_i, \vk_j, \vk_a),
                \mathcal{K}\times\mathcal{K}\times\mathcal{K}
                \right)
        \right|.
\end{equation}
The defined integrand is periodic with respect to $\vk_i,\vk_j,\vk_a\in\Omega^*$. 
As a result, the dominant quadrature error is determined by the smoothness properties of the integrand. 
The antisymmetrized ERI, $W_{ijab}$, is made of two ERIs.
The ERI $\braket{i\vk_i, j\vk_j | a\vk_a, b\vk_b}$  is singular (or slightly more accurately, nonsmooth) 
only at zero momentum transfer $\vq = \vk_a - \vk_i = \bm{0}$ (and its periodic images) due to the fraction term in its definition \cref{eqn:eri}, 
\[
        \frac{4\pi}{\abs{\Omega}} 
        \frac{
        \hat{\varrho}_{i\vk_i, a(\vk_i + \vq)}(\bm{0}) \hat{\varrho}_{j\vk_j, b(\vk_j - \vq)}(\bm{0})        
        }{|\vq|^2}.
\]
In this term, the numerator is smooth with respect to $\vk_i,\vk_j,\vq$ and the singularity solely comes from the denominator that only depends on $\vq$.
The other ERI $\braket{i\vk_i, j\vk_j | b\vk_b, a\vk_a}$ in $W_{ijab}$ has similar singularity structure at its zero momentum transfer point.

We characterize the singularity structure of CC amplitudes and ERIs in terms of the algebraic singularity of certain orders (see \cref{sub:fractional}). Our first main technical result is that the singularity structure of all the exact CCD$(n)$ amplitude entries, $[t_n]_{ijab}(\vk_i,\vk_j,\vk_a)$, is the same as that of the ERI 
 $\braket{i\vk_i, j\vk_j | a\vk_a, b\vk_b}$ (or equivalently the MP2/CCD(1) amplitude entries).
\begin{lem}[Singularity structure of the amplitude]
\label{lem:thm1_nonsmooth}
In CCD$(n)$ calculation with $n > 0$, each entry of the exact double amplitude $t_n = \{T_{ijab}^{\text{CCD}(n), \text{TDL}}\}_{ijab}$ belongs to the following function space
\begin{align*}
        \mathbb{T}(\Omega^*) = \big\{
                f(\vk_i, \vk_j, \vk_a):\  
                & f \text{ is periodic with respect to } \vk_i, \vk_j, \vk_a \in \Omega^*,\\
                & f \text{ is smooth everywhere except at } \vk_a = \vk_i 
                \text{ with algebraic singularity of order } 0,\\
                & f \text{ is smooth with respect to } \vk_i, \vk_j \text{ at nonsmooth point }\vk_a = \vk_i
        \big\}.
\end{align*}
\end{lem}


Our second technical result is the estimate of the quadrature error in the energy calculation using an arbitrary amplitude $t$ whose each entry lies in 
$\mathbb{T}(\Omega^*)$ (which covers the exact CCD$(n)$ amplitude). 
\begin{lem}[Energy error with exact amplitude]
\label{lem:thm1_energy}
        For an arbitrary amplitude $t\in \mathbb{T}(\Omega^*)^{n_\text{occ}\times n_\text{occ}\times n_\text{vir}\times n_\text{vir}}$, 
        the finite-size error in energy calculation using $t$ can be estimated as  
        \[
        \left|
                \mathcal{G}_\text{TDL}(t) - \mathcal{G}_{N_\vk}(\mathcal{M}_\mathcal{K}t)
        \right|
        \leqslant C N_\vk^{-1}.
        \]
\end{lem}
The above two lemmas together prove that the finite-size error in the energy calculation using the exact CCD$(n)$ amplitude in \cref{eqn:error_energy_exact_t} scales as $\Or(N_\vk^{-1})$. As will be seen later, this can be much faster than the convergence rate using the numerically computed amplitudes. 

Similar to the finite-size error splitting in \cref{eqn:energy_splitting}, the error in amplitude calculation can be split into two terms using
the recursive definitions of the CCD$(n)$ amplitudes in the finite and the TDL cases
\begin{align}
        \left\|
                \mathcal{M}_{\mathcal{K}}t_{n} - T_{n} 
        \right\|_\infty
        & = 
        \left\|
                \mathcal{M}_{\mathcal{K}}\mathcal{F}_\text{TDL}(t_{n-1}) - \mathcal{F}_{N_\vk}(T_{n-1})
        \right\|_\infty
        \nonumber \\
        & \leqslant 
        \left\|
                \mathcal{M}_{\mathcal{K}}\mathcal{F}_\text{TDL}(t_{n-1}) - 
                \mathcal{F}_{N_\vk}(\mathcal{M}_{\mathcal{K}}t_{n-1})
        \right\|_\infty
        +
        \left\|
                \mathcal{F}_{N_\vk}( \mathcal{M}_{\mathcal{K}}t_{n-1})
                -
                \mathcal{F}_{N_\vk}(T_{n-1})
        \right\|_\infty.
        \label{eqn:amplitude_splitting}
\end{align}
The first term is the error between the exact CCD$(n)$ amplitude and the one computed using the exact CCD$(n-1)$ amplitude $\mathcal{M}_\mathcal{K}t_{n-1}$. 
The second term is the error between the amplitude computed using the exact CCD$(n-1)$ amplitude, $\mathcal{M}_\mathcal{K}t_{n-1}$, and the one using the 
amplitude CCD$(n-1)$ amplitude, $T_{n-1}$.
The latter can be interpreted as the error accumulation from the CCD$(n-1)$ amplitude calculation.

To estimate the first error term in the \cref{eqn:amplitude_splitting}, the error between each exact and approximate amplitude entry using $t_{n-1}$, i.e., 
\[
\left[
\mathcal{M}_{\mathcal{K}}\mathcal{F}_\text{TDL}(t_{n-1})
-
\mathcal{F}_{N_\vk}(\mathcal{M}_{\mathcal{K}}t_{n-1})
\right]_{ijab,\vk_i\vk_j\vk_a}, 
\quad 
\forall i,j,a,b,\ \forall \vk_i,\vk_j,\vk_a\in\mathcal{K},
\]
can be decomposed into the summation of a series of quadrature errors that are associated with the calculation of different linear 
and quadratic terms in the amplitude equation. 

One third  technical result is the estimate of these quadrature errors, and it suggests that this first error term is of scale $\Or(N_\vk^{-\frac13})$. 

\begin{lem}[Amplitude error in a single iteration]
\label{lem:thm1_amplitude_exact}
        For an arbitrary amplitude $t\in \mathbb{T}(\Omega^*)^{n_\text{occ}\times n_\text{occ}\times n_\text{vir}\times n_\text{vir}}$, 
        the finite-size error in the next iteration amplitude calculation when using $t$ can be estimated as  
        \begin{equation}\label{eqn:thm1_amplitude_exact}
                \left\|
                        \mathcal{M}_{\mathcal{K}}\mathcal{F}_\text{TDL}(t) - \mathcal{F}_{N_\vk}(\mathcal{M}_{\mathcal{K}}t)
                \right\|_\infty
                \leqslant 
                C N_\vk^{-\frac13}.
        \end{equation}
\end{lem}

To address the second error term in \cref{eqn:amplitude_splitting}, the goal is to show that the application of $\mathcal{F}_{N_\vk}$ propagates the error in the 
CCD$(n-1)$ amplitude calculation in a controlled way.  
Specifically, noting that $\mathcal{F}_{N_\vk}(T)$ consists of constant, linear and quadratic terms of $T$ (see \cref{eqn:amplitude_ccd} for details), we can use 
its explicit definition to show that the second error term can be bounded by the error in the CCD$(n-1)$ amplitude calculation.

\begin{lem}[Lipschitz continuity of the finite CCD iteration mapping]
\label{lem:thm1_amplitude_prev}
For two arbitrary amplitude tensors $T, S \in 
\mathbb{C}^{n_\text{occ}\times n_\text{occ} \times n_\text{vir} \times n_\text{vir}\times N_\vk\times N_\vk \times N_\vk}$, 
the iteration map $\mathcal{F}_{N_\vk}$ in the finite case satisfies 
\begin{equation}\label{eqn:thm1_amplitude_prev}
        \|\mathcal{F}_{N_\vk}(T) - \mathcal{F}_{N_\vk}(S)\|_\infty \leqslant C \left(1 + \|T\|_\infty + \|S\|_\infty \right) \| T - S\|_\infty. 
\end{equation}
\end{lem}

Substitute $t = t_{n-1}$ in \cref{eqn:thm1_amplitude_exact} and $T = T_n, S= \mathcal{M}_\mathcal{K}t_{n-1}$ in \cref{eqn:thm1_amplitude_prev}, 
the error splitting in \cref{eqn:amplitude_splitting} can then be further estimated as 
\begin{equation}\label{eqn:amplitude_splitting2}
        \left\|
                \mathcal{M}_{\mathcal{K}}t_{n} - T_n
        \right\|_\infty
        \leqslant 
        C 
        \left( 
                1 + \|T_{n-1}\|_\infty + \|t_{n-1}\|_\infty
        \right)
        \left\|
                \mathcal{M}_\mathcal{K}t_{n-1} - T_{n-1}
        \right\|_\infty
        +
        C N_\vk^{-\frac13},
\end{equation}
where the second constant $C$ depends on the exact CCD$(n-1)$ amplitude $t_{n-1}$. 
Combining this recursive relation and the fact that $\mathcal{M}_\mathcal{K}t_0 = T_0$, the finite-size error in the amplitude calculation can be estimated as 
\[
        \left\|
                \mathcal{M}_{\mathcal{K}}t_{n} - T_n
        \right\|_\infty       
        \leqslant 
        C N_\vk^{-\frac13},
        \qquad 
        \forall n > 0,
\]
where constant $C$ depends on $t_m$ with $m=0,1,2,\ldots, n-1$. 
Lastly, we finish our proof of \cref{thm:error_ccd} by combining this estimate, \cref{lem:thm1_energy} with $t = t_{n-1}$, and the error splitting in \cref{eqn:energy_splitting}.

\section{Main technical tools}
Our main idea is to interpret the finite size errors in the CCD$(n)$ energy and amplitude calculations as numerical quadrature errors of trapezoidal
rules over certain singular integrals. 
Specifically, all the averaged summations over $\mathcal{K}$ in $\mathcal{G}_{N_\vk}$ and $\mathcal{F}_{N_\vk}$
are trapezoidal rules that approximate corresponding integrations over $\Omega^*$ in $\mathcal{G}_\text{TDL}$ and $\mathcal{F}_\text{TDL}$.
The problem is thus reduced to estimating the quadrature errors of trapezoidal rules over the integrands 
defined in $\mathcal{G}_\text{TDL}$ and $\mathcal{F}_\text{TDL}$ which consist of ERIs and exact amplitudes. 

In general, the asymptotic error of a trapezoidal rule depends on boundary condition and smoothness property of the integrand.
In the ideal case when an integrand is periodic and smooth, the quadrature error decays super-algebraically, i.e., faster than $N_\vk^{-l}$ with any $l > 0$, according to the 
standard Euler-Maclaurin formula.
(See \cref{lem:quaderror0} with a simple Fourier analysis explanation.) 
All the integrands defined in $\mathcal{G}_\text{TDL}$ and $\mathcal{F}_\text{TDL}$ turn out be periodic but have point singularities.
Therefore it is important to characterize the singularity structure of ERIs and exact amplitudes that constitute all these concerned integrands.

The proof of \cref{thm:error_ccd} involves two main technical challenges: describing the singularity structure of exact amplitudes in \cref{lem:thm1_nonsmooth}, and the quadrature error estimates for integrands defined in energy and amplitude calculations using exact amplitudes in \cref{lem:thm1_energy} and \cref{lem:thm1_amplitude_exact}.
For the first challenge, we define a class of functions with algebraic singularity of certain orders in \cref{sub:fractional}. 
For the second challenge, we summarize five general integral forms from the energy and amplitude calculations in \cref{sub:quaderror} and provide tight quadrature error estimates based on 
Poisson summation formula.

\subsection{Algebraic singularity}\label{sub:fractional}


Consider an ERI $\braket{n_1\vk_1,n_2\vk_2|n_3\vk_3,n_4\vk_4}$ as a periodic function of $\vk_1, \vk_2, \vq = \vk_3 - \vk_1$ over $\Omega^*$ 
while $\vk_4 = \vk_2 - \vq$ using the crystal momentum conservation.
By its definition, this ERI can be split as
\begin{equation}\label{eqn:eri_expansion}
        \braket{n_1\vk_1,n_2\vk_2|n_3\vk_3,n_4\vk_4}  
        = 
        \frac{4\pi}{\abs{\Omega}} 
        \frac{
        \hat{\varrho}_{n_1\vk_1,n_3(\vk_1 + \vq)}(\bm{0}) \hat{\varrho}_{n_2\vk_2,n_4(\vk_2 - \vq)}(\bm{0})
        }{
                |\vq|^2
        }
        + 
        \frac{4\pi}{\abs{\Omega}}
        \sum_{\vG\in\mathbb{L}^*\setminus \{\bm{0}\}}
        \dfrac{\cdots}{|\vq+\vG|^2},
\end{equation}
where all the terms with $\vG\neq \bm{0}$ are smooth with respect to $\vk_1, \vk_2, \vq \in \Omega^*$ and the singularity of the ERI only comes from the first fraction term at $\vq = \bm{0}$. 
The numerator of this fraction is smooth with respect to $\vk_1, \vk_2, \vq \in \Omega^*$ (note the assumption that $\psi_{n\vk}(\vrr)$ is smooth with respect to $\vk$)
and is of scale $\Or(|\vq|^a)$ with $a \in \{0, 1, 2\}$ near $\vq = \bm{0}$ using orbital orthogonality.  
The exact value of $a$ depends on the relation between band indices $(n_1, n_2)$ and $(n_3, n_4)$. 
As can be verified by direct calculation, this fraction and its derivatives over $\vq$ with any fixed $\vk_1, \vk_2$ satisfy the following 
characterization. 
\begin{defn}[Algebraic singularity for univariate functions]\label{def:fractional}
        A function $f(\vx)$ has \textbf{algebraic singularity of order} $\gamma\in\RR$ at $\vx_0 \in \mathbb{R}^d$  if
        there exists $\delta > 0$ such that 
        \[
        \left|
        \dfrac{\partial^\valpha}{\partial \vx^{\valpha}} f(\vx)  
        \right|
        \leqslant 
        C_\valpha |\vx - \vx_0|^{\gamma - |\valpha|},
        \qquad \forall 0 < |\vx - \vx_0| < \delta, \ \forall \valpha \geqslant 0,
        \]
        where constant $C_\valpha$ depends on $\delta$ and the non-negative $d$-dimensional derivative multi-index $\valpha$.  
        For brevity, $f$ is also said to be singular (or nonsmooth) at $\vx_0$ with order $\gamma$.
\end{defn}
\begin{rem}
        In integral equations, algebraic singularity is commonly used to describe the asymptotic behavior of a kernel function near a singular point.
        The above definition slightly generalizes this concept to additionally include the asymptotic behaviors of all the derivatives. 
        Note that when $\gamma > 0$, $f(\vx)$ is continuous but nonsmooth at $\vx_0$ since its derivatives can be singular at this point.
        In this case, we slightly abuse the name and still refer to $\vx_0$ as a point of algebraic singularity.
\end{rem}

A simple example of such nonsmooth functions is $p(\vx)/|\vx|^2$ where $p(\vx)$ is smooth and of scale $\Or(|\vx|^{\gamma+2})$ 
near $\vx = \bm{0}$.
Using this concept, the leading fraction term in \cref{eqn:eri_expansion} is nonsmooth at $\vq = \bm{0}$ with order $\gamma\in\{-2,-1,0\}$. 
Since the smooth terms in \cref{eqn:eri_expansion} do not change the inequalities in \cref{def:fractional} qualitatively, the ERI example is also nonsmooth at 
$\vq = \bm{0}$ with order $\gamma$. 
In addition, to connect the algebraic singularities of the ERI example with varying $\vk_1, \vk_2 \in \Omega^*$, we further introduce the algebraic singularity with respect 
to one variable for a multivariate function. 
\begin{defn}[Algebraic singularity for multivariate functions]\label{def:fractional2}
        A function $f(\vx, \vy)$ is smooth with respect to $\vy \in V_Y\subset \mathbb{R}^{d_y}$ for any fixed $\vx$ and has algebraic singularity of order $\gamma$ with respect to $\vx$ at $\vx_0 \in \mathbb{R}^{d_x}$ 
        if there exists $\delta > 0$ such that 
        \[
                \left|
                \dfrac{\partial^\valpha}{\partial \vx^{\valpha}} \left( \dfrac{\partial^\vbeta}{\partial\vy^\vbeta}f(\vx, \vy)  \right)
                \right|
                \leqslant 
                C_{\valpha,\vbeta} |\vx - \vx_0|^{\gamma - |\valpha|},  \quad \forall 0<|\vx - \vx_0|< \delta, \forall \vy \in V_Y, \forall \valpha,\vbeta\geqslant 0,
        \]
        where constant $C_{\valpha,\vbeta}$ depends on $\delta$, $\valpha$ and $\vbeta$.
        Compared to the univariate case in \cref{def:fractional}, the key additions are the shared algebraic singularity of partial derivatives over $\vy$ at $\vx = \vx_0$ 
        with order $\gamma$ and the independence of $C_{\valpha, \vbeta}$ on $\vy\in V_Y$. 
\end{defn}

A simple example of such nonsmooth functions is $p(\vx,\vy)/|\vx|^2$ where $p(\vx, \vy)$ is smooth and of scale $\Or(|\vx|^{\gamma+2})$ near 
$\vx = \bm{0}$.  
The first fraction term in \cref{eqn:eri_expansion} is of this form with $\vx = \vq$, $\vy = (\vk_i, \vk_j)$, and $\gamma \in \{-2,-1,0\}$. 
Therefore the ERI example is smooth everywhere with respect to $\vk_i, \vk_j, \vq\in\Omega^*$ except at $\vq = \bm{0}$ with order $\gamma\in\{-2,-1,0\}$. 
If treating the ERI as a function of $\vk_1,\vk_2,\vk_3$, then the ERI is smooth everywhere except at $\vk_3 = \vk_1$ with order $\gamma$. 
This defines the algebraic singularity used in $\mathbb{T}(\Omega^*)$ in \cref{lem:thm1_nonsmooth}. 

\cref{lem:thm1_nonsmooth} states that all entries of the exact CCD$(n)$ amplitude, $[t_n]_{ijab}(\vk_i, \vk_j, \vk_a)$, are 
smooth everywhere in $\Omega^*\times\Omega^*\times\Omega^*$ except at $\vq = \vk_a - \vk_i = \bm{0}$ with order $0$ (as specified in $\mathbb{T}(\Omega^*)$). 
Due to similar singularity structures between CC amplitudes and ERIs, we also refer to $\vq$ as the \textit{momentum transfer of the amplitude}. 
Recall that the exact CCD$(1)$ amplitude $t_1 =  \left\{\braket{a\vk_a, b\vk_b | i\vk_i, j\vk_j}/\varepsilon_{i\vk_i,j\vk_j}^{a\vk_a, b\vk_b}\right\}_{ijab}$ 
has entries in $\mathbb{T}(\Omega^*)$ and $t_n$ is defined by recursively applying $\mathcal{F}_\text{TDL}$ to $t_1$. 
It is thus sufficient to prove that
\[
        \mathcal{F}_\text{TDL}(t)\in \mathbb{T}(\Omega^*)^{n_\text{occ}\times n_\text{occ}\times n_\text{vir}\times n_\text{vir}}, \quad 
        \forall t \in \mathbb{T}(\Omega^*)^{n_\text{occ}\times n_\text{occ}\times n_\text{vir}\times n_\text{vir}}.
\]

For each set of $(i,j,a,b)$, $[\mathcal{F}_\text{TDL}(t)]_{ijab}(\vk_i, \vk_j, \vk_a)$ consists of constant, linear, and quadratic terms of $t$ (see \cref{eqn:amplitude_ccd_tdl}). 
It turns out that each of these terms as a function of $\vk_i, \vk_j, \vq = \vk_a - \vk_i$ is smooth everywhere except at $\vq = \bm{0}$ with order $0$. 
The constant term is an MP2 amplitude entry $[t_1]_{ijab}$ and can be verified directly.
All the linear terms are of integral form with integrands being the product of an ERI and an amplitude entry, and can be categorized into three classes 
with representative examples (ignoring orbital energy fraction and constant prefactor) as
\begin{align}
        & \int_{\Omega^*}\ud\vk_k \sum_{kc}\braket{a\vk_a, k\vk_k | i\vk_i, c\vk_c} t_{kjcb}(\vk_k, \vk_j, \vk_c),
        \\
        & \int_{\Omega^*}\ud\vk_k \sum_{kc}\braket{a\vk_a, k\vk_k | i\vk_i, c\vk_c} t_{kjbc}(\vk_k, \vk_j, \vk_b),
        \\
        & \int_{\Omega^*}\ud\vk_k \sum_{kl}\braket{k\vk_k, l\vk_l | i\vk_i, j\vk_j} t_{klab}(\vk_k, \vk_l, \vk_a),
        \label{eqn:integral_4h2p_linear}
\end{align}
The difference among the three classes is the singularity structure of the integrand  with respect to the integration variable, e.g., $\vk_k$ in the above examples. 
For any fixed $\vk_i,\vk_j,\vq\neq\bm{0}$, the singular points of integrands in the three examples above with respect to $\vk_k$ are: (1) none, (2) $\vk_k = \vk_b$, 
(3) $\vk_k = \vk_i$ and $\vk_k = \vk_a$. 
Hence the three integrands are nonsmooth with respect to $\vk_k$ at zero, one, and two points, respectively.
The goal is to show that derivatives of these integrals with respect to $\vk_i, \vk_j, \vq \in \Omega^*$
exist except at $\vq = \bm{0}$ and satisfy the algebraic singularity condition in \cref{def:fractional2} with order $\gamma = 0$. 
The first two classes can be addressed using the Leibniz integral rule for the differentiation under integral sign 
and direct derivative estimates after the interchange of differentiation and integration operations. 
The third class requires more involved analysis based on an additional technical lemma in \cref{app:nonsmooth_integral}.
\Cref{fig:diagram} plots diagrams of all the linear amplitude terms where the first class contains (a), the second contains (b) and (c), and the third contains (d), (e) and (f). 

\begin{figure}
        \subfloat[${\sum_{KC}}\braket{AK|IC}t_{KJ}^{CB}$]{
                \includegraphics[height=0.15\textheight]{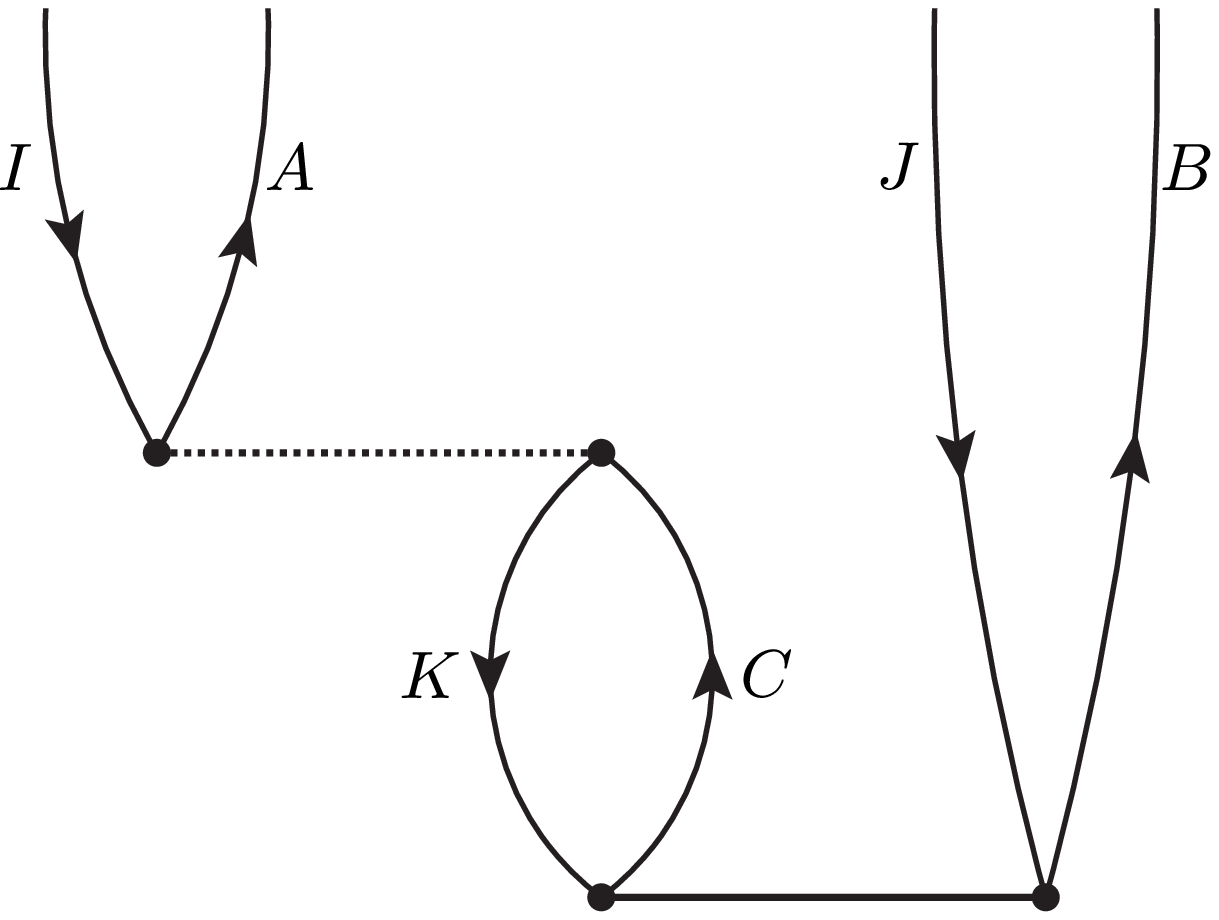}
        }
        \hspace{3em}
        \subfloat[${\sum_{KC}}\braket{AK|CI}t_{KJ}^{CB}$]{
                \includegraphics[height=0.15\textheight]{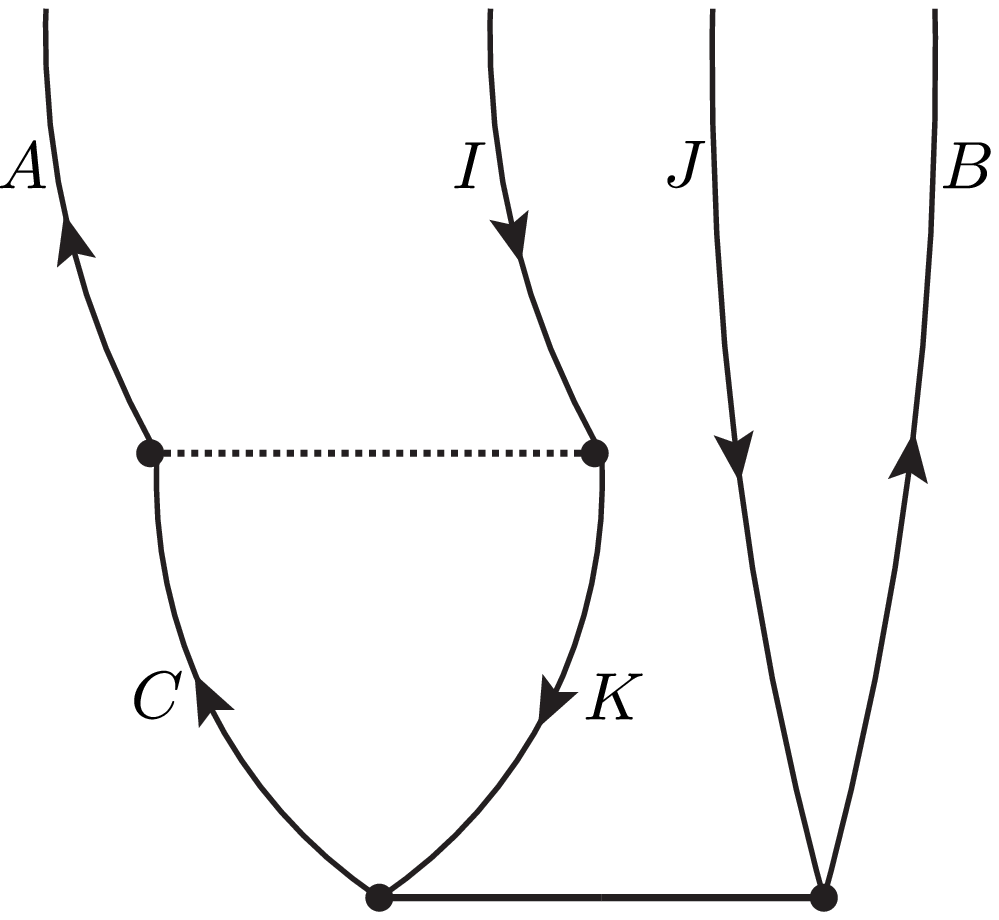}
        }
        \hspace{3em}
        \subfloat[${ \sum_{KC}}\braket{AK|IC}t_{KJ}^{BC}$]{
                \includegraphics[height=0.15\textheight]{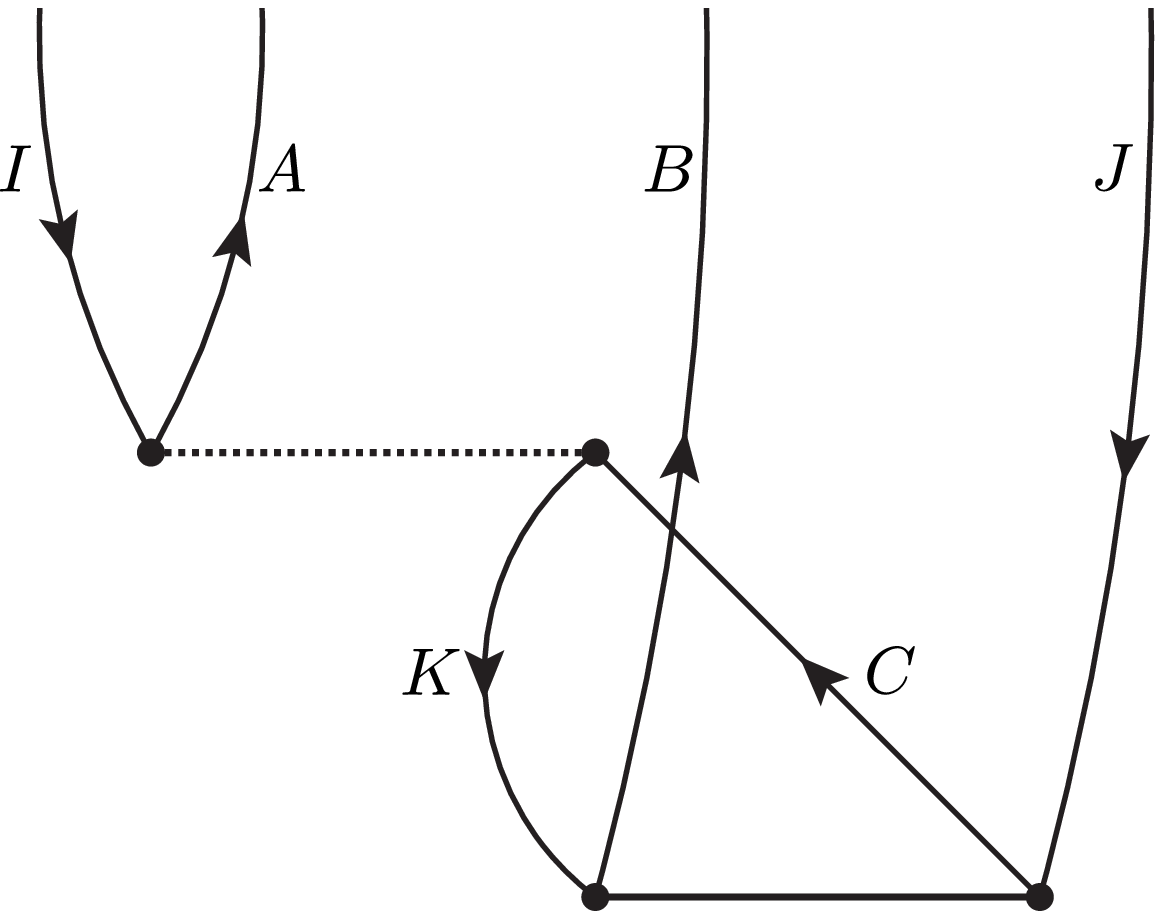}
        }
        \\
        \subfloat[${ \sum_{KC}}\braket{AK|CJ}t_{IK}^{CB}$]{
                \hspace{0.5em}
                \includegraphics[height=0.15\textheight]{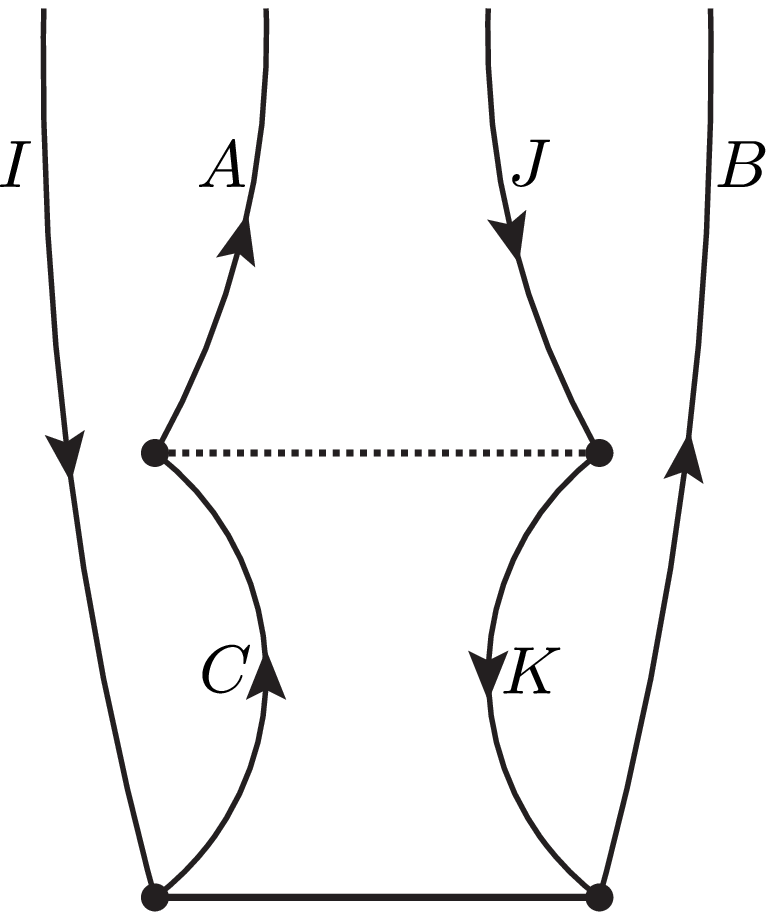}
                \hspace{0.5em}
        }
        \hspace{4em}
        \subfloat[${ \sum_{KL}}\braket{KL|IJ}t_{KL}^{AB}$]{
                \hspace{0.5em}
                \includegraphics[height=0.15\textheight]{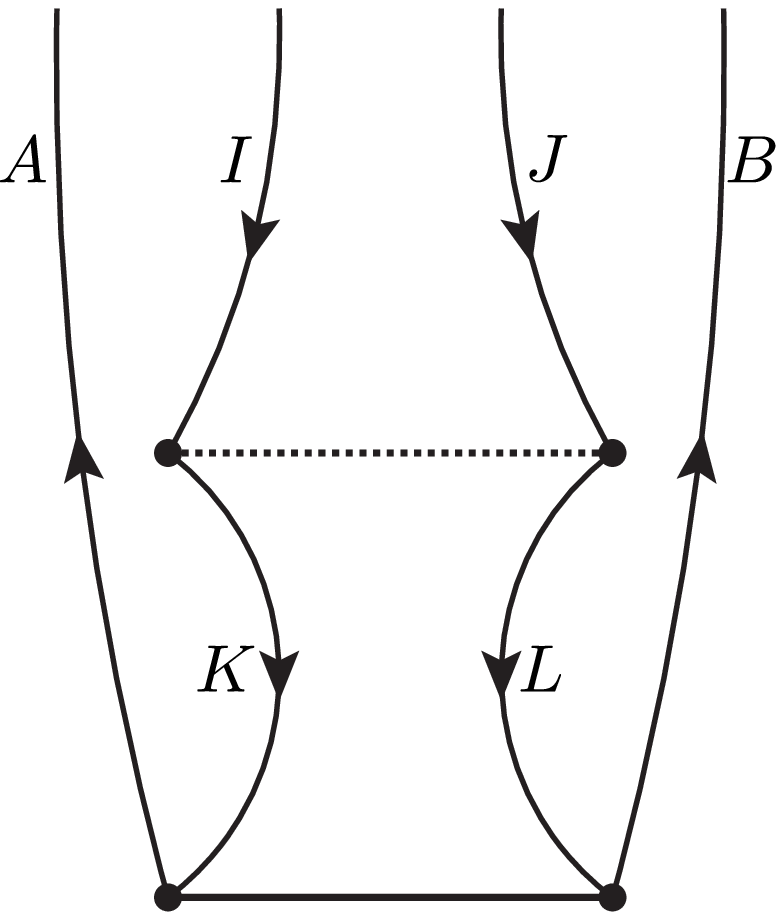}
                \hspace{0.5em}
        }
        \hspace{4em}
        \subfloat[${ \sum_{CD}}\braket{AB|CD}t_{IJ}^{CD}$]{
                \hspace{0.5em}
                \includegraphics[height=0.15\textheight]{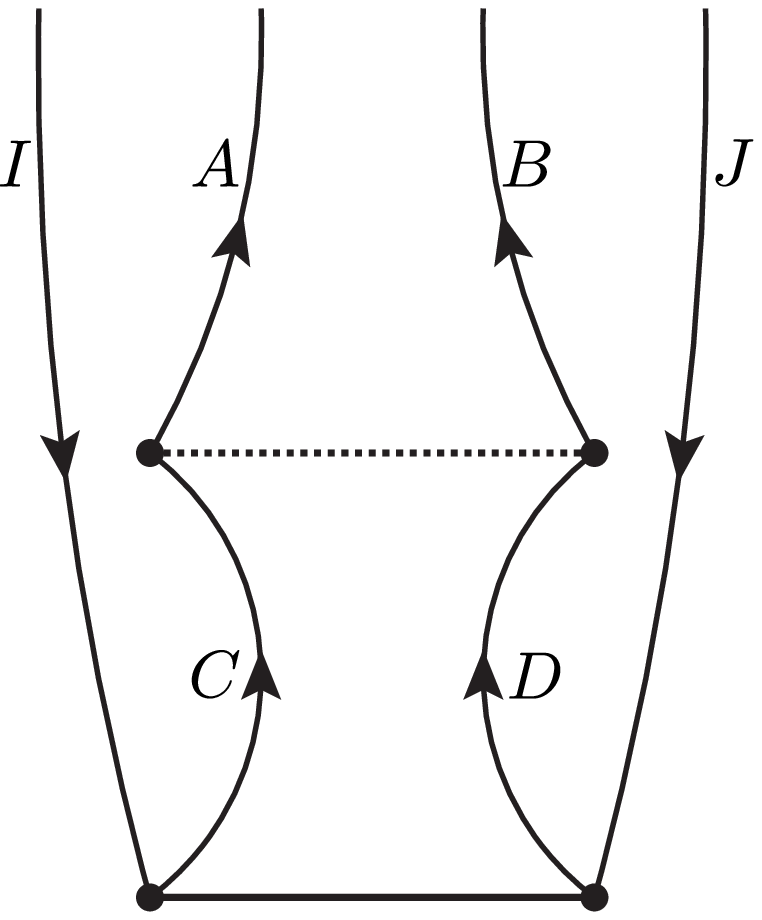}
                \hspace{0.5em}
        }

        \caption{Diagrams of all the linear terms in the amplitude calculation. The dashed horizontal line denotes the ERI, and the solid horizontal line denotes the $t$ amplitude.
        The permuted amplitudes from (a), (b), (c), and (d) are not plotted.
        A capital letter $P$ refers to an orbital index $(p, \vk_p)$.
        The amplitude calculation in (f) can be formulated as an integral over momentum vector $\vk_c$, and in all other subplots as integrals over momentum vector $\vk_k$.
        }
        \label{fig:diagram}
\end{figure}

All the quadratic terms are also of integral form with integrands being the product of an ERI and two amplitude entries, and 
the analysis of their smoothness property can be decomposed into two subproblems similar to those for the linear terms above.  
For example, the 4h2p quadratic term is defined as 
\begin{equation}
        \dfrac{1}{|\Omega^*|^2}\int_{\Omega^*}\ud\vk_k\
        \left(
                \int_{\Omega^*}\ud\vk_c \sum_{klcd} \braket{k\vk_k, l\vk_l | c\vk_c, d\vk_d}
                t_{ijcd}(\vk_i, \vk_j, \vk_c)
        \right) t_{klab}(\vk_k, \vk_l, \vk_a).
\end{equation}
The  term in the parenthesis is a linear term as a function of $\vk_i,\vk_j, \vq, \vk_k$ and is singular at $\vk_k = \vk_i$ with order $0$. 
In other words, its singularity structure is the same as the ERI $\braket{k\vk_k, l\vk_l | i\vk_i, j\vk_j}$ in the 4h2p linear term in \cref{eqn:integral_4h2p_linear}. 
The overall 4h2p quadratic term thus resembles the 4h2p linear term, simply with the ERI replaced by the term in the parenthesis above, and the analysis for linear terms still can be
applied to prove its algebraic singularity of order $0$ at $\vq = \bm{0}$.

\subsection{Quadrature error for periodic functions with algebraic singularity}\label{sub:quaderror}

Two quadrature error estimate problems need to be addressed for \cref{thm:error_ccd}: 
the error in the energy calculation using exact amplitude in \cref{lem:thm1_energy}
and the error in the single iteration amplitude calculation using exact amplitude in \cref{lem:thm1_amplitude_exact}.
Integrands in these two problems are all products of orbital energy fractions, ERIs, and exact amplitudes, and are 
periodic with respect to involved momentum vectors over $\Omega^*$. 
Therefore, their smoothness properties determine the dominant quadrature errors with a trapezoidal rule. 
Specifically, the integrand nonsmoothness only comes from the algebraic singularity of ERIs and amplitudes at their zero momentum transfer points 
(and their periodic images). 
The dominant quadrature errors thus come from the numerical quadrature over integration variables dependent on 
the momentum transfer vectors of the included ERIs and amplitudes. 

By proper change of variables and splitting of error terms, all the quadrature error estimates in \cref{lem:thm1_energy} and \cref{lem:thm1_amplitude_exact} 
can be summarized as those for the five types of integrals described in \cref{tab:class_integral}.
The quadrature errors for these five integrals are studied separately in \cref{lem:quaderror0,lem:quaderror1,lem:quaderror2,lem:quaderror3,lem:quaderror4} in \cref{app:quadrature_error}.
In our finite-size error analysis, we have $V = \Omega^*$, $d=3$, and $m = N_\vk^{-\frac13}$. 

\begin{table}[htbp]
        \caption{Five types of integrals in the quadrature error estimate problems that appear in CCD energy and amplitude calculations.
        All functions are assumed to be periodic with respect to $V\subset \mathbb{R}^d$ and are smooth everywhere except at the singular points.
        Parameter $d$ denotes the domain dimension and $m$ denotes the number of points along each dimension in the uniform mesh used 
        for trapezoidal quadrature rule (thus the mesh has $m^d$ points). 
        The ``Estimate'' column gives the quadrature error estimates for the trapezoidal rule using an $m^d$-sized mesh, which are
        obtained in the lemma specified in the ``Lemma'' column.
        }

        \label{tab:class_integral}
        \def\arraystretch{2.5}
        \begin{tabular}{p{0.45\textwidth}@{\hskip 1em}p{0.24\textwidth}@{\hskip 1em}cc}
                \toprule
                Description & Singular points and order & Estimate & Lemma \\
                \midrule
                $\int_{V}\ud\vx f(\vx)$ & None
                & Super-Algebraic& \cref{lem:quaderror0}\\
                $\int_{V}\ud\vx f(\vx)$ & $\vx = \bm{0}$ of order $\gamma$
                & $m^{-(d+\gamma)}$ & \cref{lem:quaderror1} \\
                $\int_{V}\ud\vx f_1(\vx)f_2(\vx)$ & \makecell[l]{$f_1(\vx)$: $\vx = \bm{0}$ of order $\gamma$;\\ $f_2(\vx)$: $\vx = \vz$ of order $0$} & $m^{-(d+\gamma)}$ & \cref{lem:quaderror2}\\
                $\int_{V\times V}\ud\vx_1\ud\vx_2 f_1(\vx_1, \vx_2)f_2(\vx_1, \vx_2)$ & $f_i(\vx_1, \vx_2)$:  $\vx_i = \bm{0}$ of order $\gamma_i$, $i = 1,2$ 
                & $m^{-(d+\min_i \gamma_i)}$ & \cref{lem:quaderror3}\\
                $\int_{V\times V}\ud\vx_1\ud\vx_2 f_1(\vx_1, \vx_2)f_2(\vx_1, \vx_2) f_3(\vx_1, \vx_2 \pm \vx_1)$ & \makecell[l]{$f_i(\vx_1, \vx_2)$:  $\vx_i = \bm{0}$ of order \\ $\gamma_i$, $i = 1,2$; \\ $f_3(\vx_1, \vz)$: $\vz = \bm{0}$ of order $0$}
                & $m^{-(d+\min_i \gamma_i)}$ & \cref{lem:quaderror4}\\
                \bottomrule
        \end{tabular}
\end{table}

To demonstrate the classification in \cref{tab:class_integral}, we provide three examples.
\paragraph*{Example 1:} The 4h2p linear terms in the amplitude calculation with any fixed $i,j,a,b$ and $\vk_i,\vk_j,\vk_a$
using an exact amplitude $t \in \mathbb{T}(\Omega^*)^{n_\text{occ}\times n_\text{occ} \times n_\text{vir} \times n_\text{vir}}$ is of the form  
\begin{equation}\label{eqn:amplitude_4h2p_2}
        \int_{\Omega^*}\ud\vk_k \sum_{kl}\braket{k\vk_k, l\vk_l | i\vk_i, j\vk_j} t_{klab}(\vk_k, \vk_l, \vk_a).
\end{equation}
The ERI and the amplitude have momentum transfers $\vq_1 = \vk_i - \vk_k$ and $\vq_2 = \vk_a - \vk_k$, respectively.
By the change of variable $\vk_k \rightarrow \vk_i - \vq_1$, the integral with each $k,l$ over $\vq_1\in \Omega^*$ is a product of two functions 
that have algebraic singularity at $\vq_1= \bm{0}$ and $\vq_2 = \vq_1 + (\vk_a - \vk_i) = \bm{0}$ and thus belongs to type 3 in \cref{tab:class_integral}.
Since the ERI with $k=i, l=j$ is nonsmooth at $\vq_1 = \bm{0}$ with order $-2$, the overall quadrature error in the calculation of 
this term using $\mathcal{K}$ is of scale $\Or(N_\vk^{-\frac13})$ according to the estimate. 

\paragraph*{Example 2:} 
The correlation energy exchange term using an exact amplitude $t$ is of form 
\[
        \int_{\Omega^*\times\Omega^*\times\Omega^*}\ud\vk_i\ud\vk_j\ud\vk_a
        \sum_{ijab}
        \braket{i\vk_i, j\vk_j | b\vk_b, a\vk_a} t_{ijab}(\vk_i, \vk_j, \vk_a). 
\]
The ERI and the amplitude have momentum transfers $\vq_1 = \vk_b - \vk_i$ and $\vq_2 = \vk_a - \vk_i$, respectively.
By the change of variables $\vk_j \rightarrow \vk_a + \vq_1$ and $\vk_i \rightarrow \vk_a - \vq_2$, the integrand is defined over $\vk_a, \vq_1, \vq_2\in\Omega^*$ 
and is smooth with respect to $\vk_a$.
As a result, the partial integral over $\vk_a$ belongs to type 1 and for each fixed $\vk_a$ the partial integral over $\vq_1, \vq_2$ belongs to type 4. 
Since the included ERIs and amplitudes all have algebraic singularities of order $0$, the quadrature error in calculating this term 
scales as $\Or(N_\vk^{-1})$. 

\paragraph*{Example 3:} The 4h2p quadratic term in the amplitude calculation is of form 
\begin{equation*}
        \int_{\Omega^*\times \Omega^*}\ud\vk_k\ud\vk_c
                \sum_{klcd} \braket{k\vk_k, l\vk_l | c\vk_c, d\vk_d}
                t_{ijcd}(\vk_i, \vk_j, \vk_c)
                t_{klab}(\vk_k, \vk_l, \vk_a),
\end{equation*}
where $\vk_l = \vk_i + \vk_j - \vk_k$ and $\vk_d = \vk_i + \vk_j - \vk_c$ by crystal momentum conservation. 
The three momentum transfers are $\vq_1 = \vk_c - \vk_k$, $\vq_2 = \vk_c - \vk_i$, and $\vq_3 = \vk_a - \vk_k$. 
By the change of variables $\vk_c \rightarrow \vk_i + \vq_2$ and $\vk_k \rightarrow \vk_i + \vq_2 - \vq_1$, we have $\vq_3 = \vq_1 - \vq_2 + (\vk_a - \vk_i)$  
and the integral over $\vq_1, \vq_2$ belongs to type 5. 
Since the included ERIs and amplitudes all have algebraic singularities of order $0$, the quadrature error in calculating this term 
scales as $\Or(N_\vk^{-1})$. 

\cref{tab:amplitude_error} summarizes the quadrature error estimates for all the CCD amplitude and energy calculations
in the proofs of \cref{lem:thm1_energy} and \cref{lem:thm1_amplitude_exact}.
Note that all entries of the exact CCD$(n)$ amplitude $t_n$ are smooth everywhere except at $\vk_a = \vk_i$ with order $0$. 
The order of singularity of an ERI $\braket{n_1\vk_1,n_2\vk_2|n_3\vk_3,n_4\vk_4}$ at $\vk_3 - \vk_1=\bm{0}$ equals to $-2$, $-1$, and $0$ respectively when  
its band indices match (i.e., $n_1=n_3, n_2 = n_4$), partially match (i.e., $n_1=n_3,n_2\neq n_4$ or $n_1\neq n_3,n_2 = n_4$), and do not match (i.e., $n_1\neq n_3, n_2\neq n_4$).
Excluding the special terms with super-algebraically decaying errors, the rule of thumb for these quadrature error estimates is that one calculation 
can have $\Or(N_\vk^{-\frac13})$, $\Or(N_\vk^{-\frac23})$, and $\Or(N_\vk^{-1})$ errors respectively when it contains ERIs 
with matching, partially matching, and no-matching band indices. 
For example, the 4h2p linear term in \cref{eqn:amplitude_4h2p_2} consists of ERI-amplitude products with varying band indices $k, l$. 
The products with (1) $k=i,l=j$, (2) $k=i,l\neq j$ or $k\neq i, l=j$, and (3) $k\neq i, l\neq j$ have respectively $\Or(N_\vk^{-\frac13})$, $\Or(N_\vk^{-\frac23})$, 
and $\Or(N_\vk^{-1})$ quadrature errors due to the ERI $\braket{k\vk_k, l\vk_l | i \vk_i, j\vk_j}$. 
The overall quadrature error in the 4h2p linear term is thus $\Or(N_\vk^{-\frac13})$, dominated by the product term with $k=i,l=j$.
In CCD amplitude calculations, ERIs with matching or partially matching band indices only appear in terms whose diagrams have interaction vertices with ladder structure,
see \cref{fig:diagram} for the four linear terms in \cref{tab:amplitude_error} with $\Or(N_\vk^{-\frac13})$ error.


\begin{table}[htbp]
        \caption{Error estimate for each individual term that appear in CCD  energy (using exact amplitudes) as well as amplitude calculations.
        All amplitude terms assume fixed $I,J,A,B$ and the integral over momentum vector and the summation over band indices are over intermediate orbitals, e.g., $K, L, C, D$. 
        Each of the 3h3p linear terms permuted by $\mathcal{P}$ in \cref{eqn:amplitude_ccd_tdl} has the same error 
        estimates as the unpermuted one and is not listed here.
        }
        \label{tab:amplitude_error}
        \begin{tabular}{lllc}
                \toprule
                & Type            & Terms                                                                                                              & Error Estimate                          \\
                \midrule
                Energy &  & $\sum_{IJAB}\braket{IJ|AB}t_{IJ}^{AB}$,  $\sum_{IJAB}\braket{IJ|BA}t_{IJ}^{AB}$ & $N_\vk^{-1}$ \\
                \midrule
        \multirow{5}{*}{Amplitude} & constant                & $\braket{AB|IJ}$                                                                                                   & 0                                       \\
        & \multirow{3}{*}{linear} & $\braket{KL|IJ}t_{KL}^{AB}$, $\braket{AB|CD}t_{IJ}^{CD}$, $\braket{AK|CI}t_{KJ}^{CB}$, $\braket{AK|CJ}t_{KI}^{BC}$ & $N_\vk^{-\frac13}$                      \\
                                & & $\braket{AK|IC}t_{KJ}^{BC}$                                                                                        & $N_\vk^{-1}$                            \\
                                & & $\braket{AK|IC}t_{KJ}^{CB}$                                                                                        & \text{Super-Algebraic} \\
        & \multirow{2}{*}{quadratic}&  $\braket{LK|DC}t_{IL}^{AD}t_{KJ}^{CB}$ &  \text{Super-Algebraic}\\
        & & all other terms& $N_\vk^{-1}$                            \\
                \bottomrule
        \end{tabular}
\end{table}

\section{Numerical Examples}
To demonstrate the finite-size errors in the energy and amplitude calculations with large $\vk$-point meshes, 
we consider a model system with an effective potential field. Using a fixed effective potential field, we can obtain orbitals and orbital energies in the TDL at any $\vk$ point, which satisfies the assumptions in \cref{thm:error_ccd}.
This simplified setup also enables us to perform large calculations using up to $N_{\vk}=16^3=4096$ $\vk$ points, which can be interpreted as a supercell with $4096$ atoms in total.

Let the unit cell be $\Omega = [0,1]^3$, we use $16$ planewave basis functions along each direction to discretize operators and functions in this unit cell (i.e., the number of $\vG$ points is $16^3=4096$ and this is independent of $N_{\vk}$). 
At each momentum vector $\vk\in \Omega^*$, we solve the effective Kohn-Sham equation 
to obtain $n_\text{occ} = 1$ occupied orbital and $n_\text{vir} = 1$ virtual orbital where the Gaussian effective potential is defined as 
\[
        V(\vrr) = \sum_{\vR\in\mathbb{L}} C\exp
        \left(
                -\frac12 (\vrr + \vR - \vrr_0)^\top \Sigma^{-1}(\vrr + \vR - \vrr_0)
        \right),
\]
with $\vrr_0 = (0.5, 0.5, 0.5)$, $\Sigma = \diag(0.1^2, 0.2^2, 0.3^2)$, and $C = -200$. 
This model problem has a direct gap of size around 30.4 between the occupied and virtual bands.

\Cref{fig:example} plots six calculations that are representative of the error estimates summarized in \cref{tab:amplitude_error} using 
the exact CCD(1) amplitude. 
To identify the asymptotic error scaling without reference values, in each plot, we use the three data points at small $N_\vk$ to construct the power-law 
extrapolations (i.e., the curve fitting in the form $C_0 + C_1 N_\vk^{-s}$) and the discrepancy between the extrapolation and the actual values at larger $N_\vk$ 
can then be used to measure the fitting quality. 
As can be observed, the convergence rates of the tested energy and amplitude calculations are consistent with the theoretical estimates in 
\cref{tab:amplitude_error}. 
These results thus justify the finite-size error estimates in \cref{lem:thm1_energy} and \cref{lem:thm1_amplitude_exact}, 
and the series of general quadrature error estimates for periodic functions with algebraic singularity obtained 
in \cref{lem:quaderror0,lem:quaderror1,lem:quaderror2,lem:quaderror3,lem:quaderror4}.
These numerical evidences demonstrate that the estimate of the finite-size  error in \cref{thm:error_ccd} is sharp.

\begin{figure}[htbp]
        \subfloat[$\sum_{IJAB}\braket{IJ|AB}t_{IJ}^{AB}$\label{fig:example1}]{
                \includegraphics[width=0.32\textwidth]{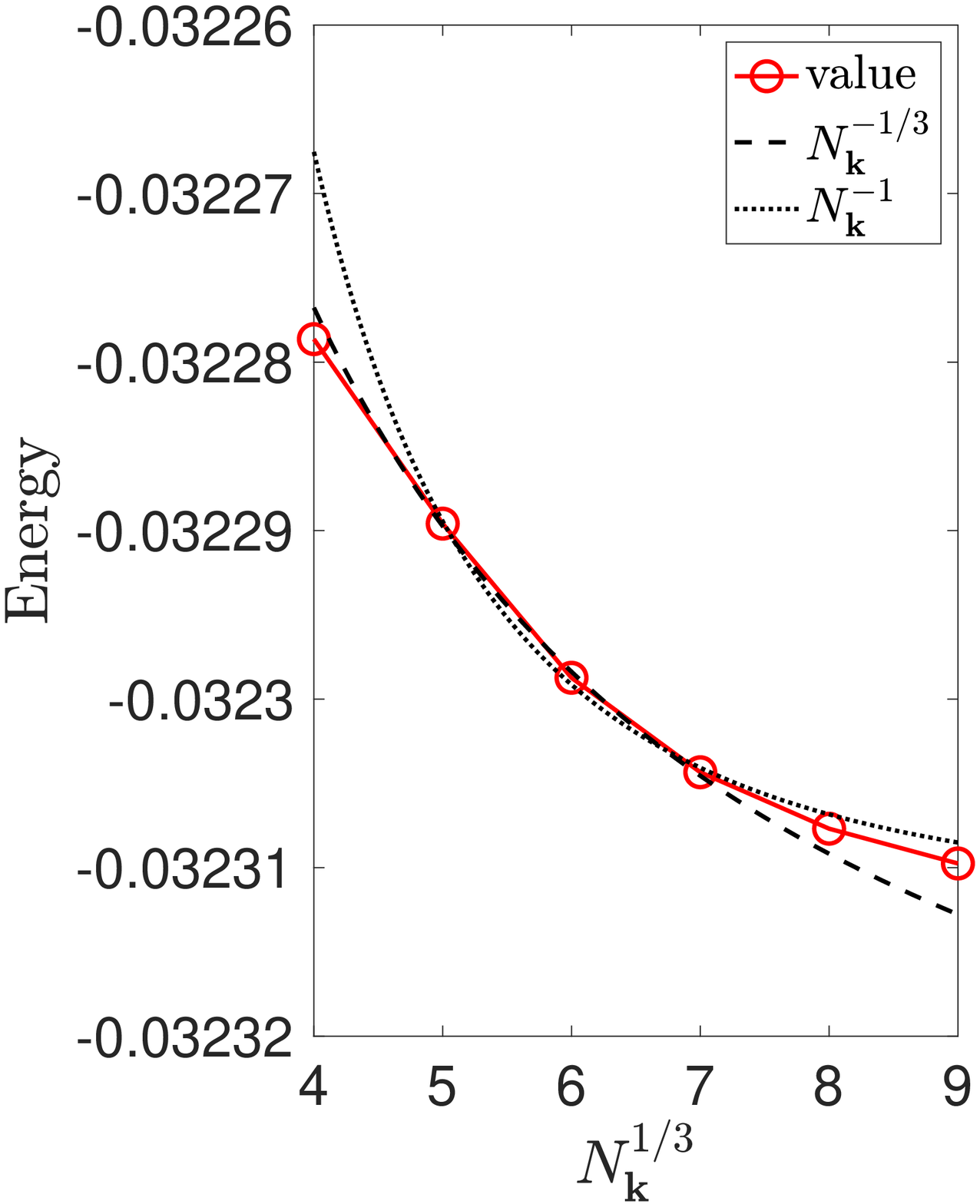}
        }
        \subfloat[$\sum_{KL}\braket{KL|IJ}t_{KL}^{AB}$\label{fig:example2}]{
                \includegraphics[width=0.32\textwidth]{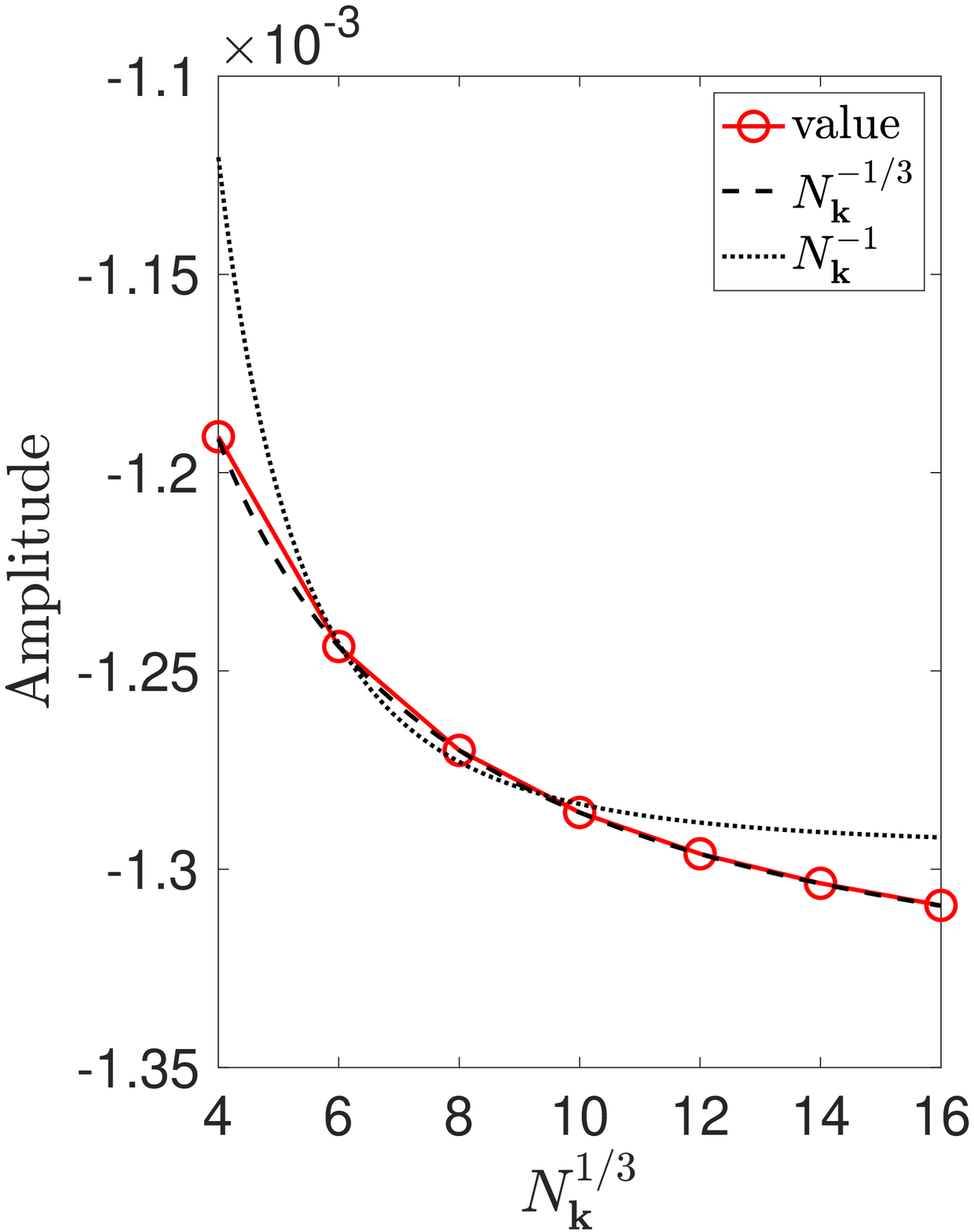}
        }
        \subfloat[$\sum_{KC}\braket{AK|IC}t_{KJ}^{BC}$\label{fig:example3}]{
                \includegraphics[width=0.32\textwidth]{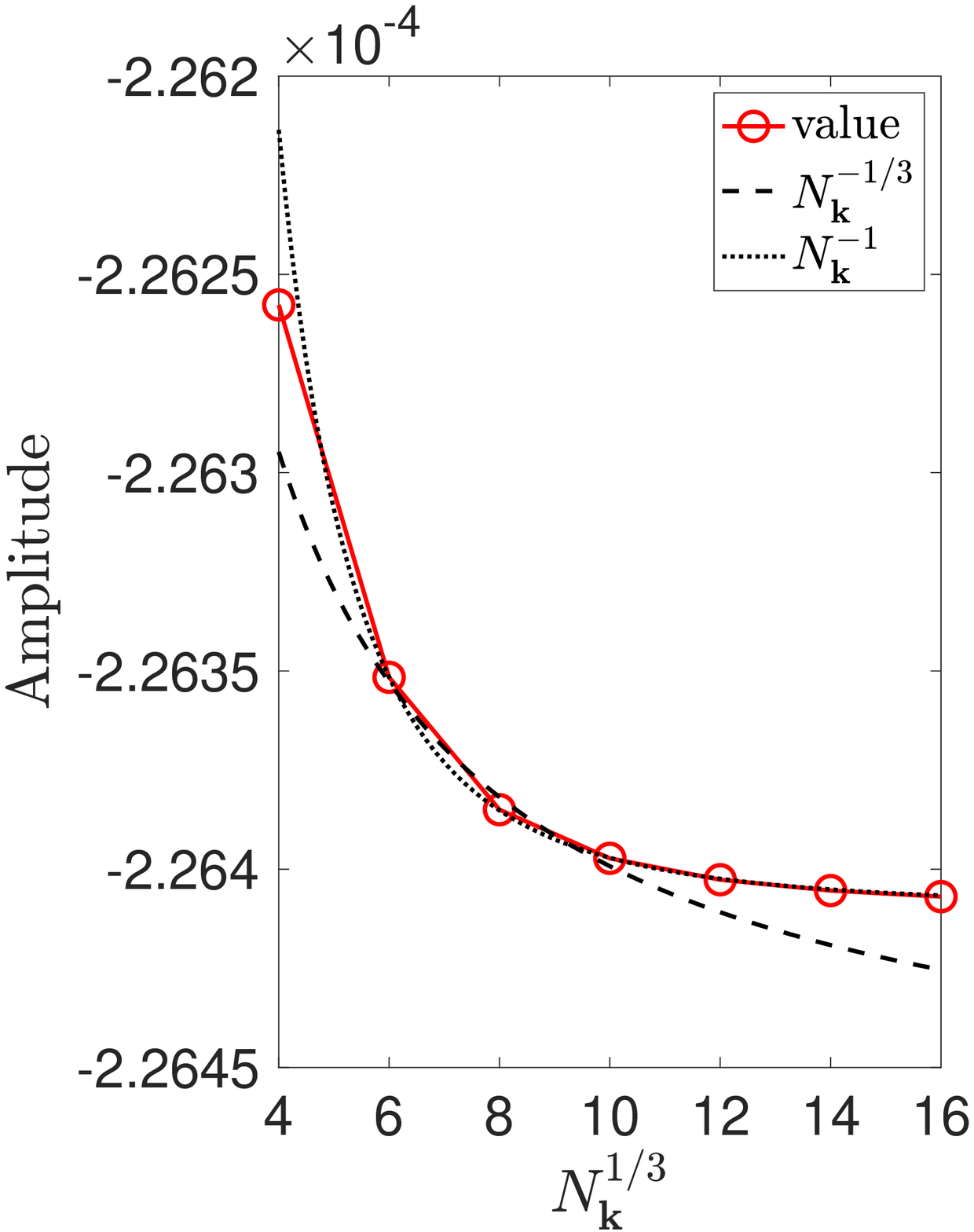}
        }

        \subfloat[$\sum_{KC}\braket{AK|IC}t_{KJ}^{CB}$\label{fig:example4}]{
                \includegraphics[width=0.32\textwidth]{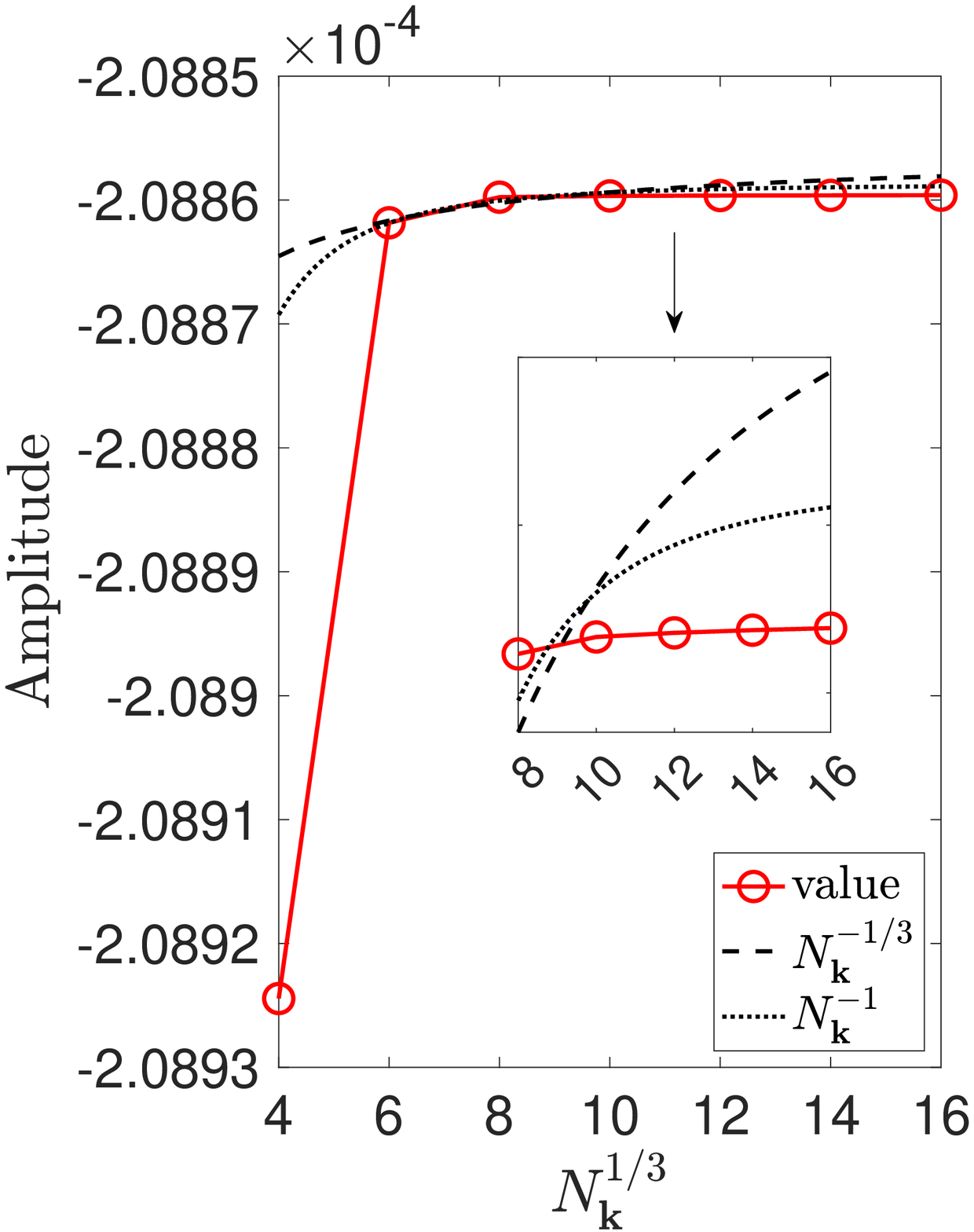}
        }
        \subfloat[$\sum_{KLCD}\braket{LK|DC}t_{IL}^{AD}t_{KJ}^{CB}$\label{fig:example5}]{
                \includegraphics[width=0.32\textwidth]{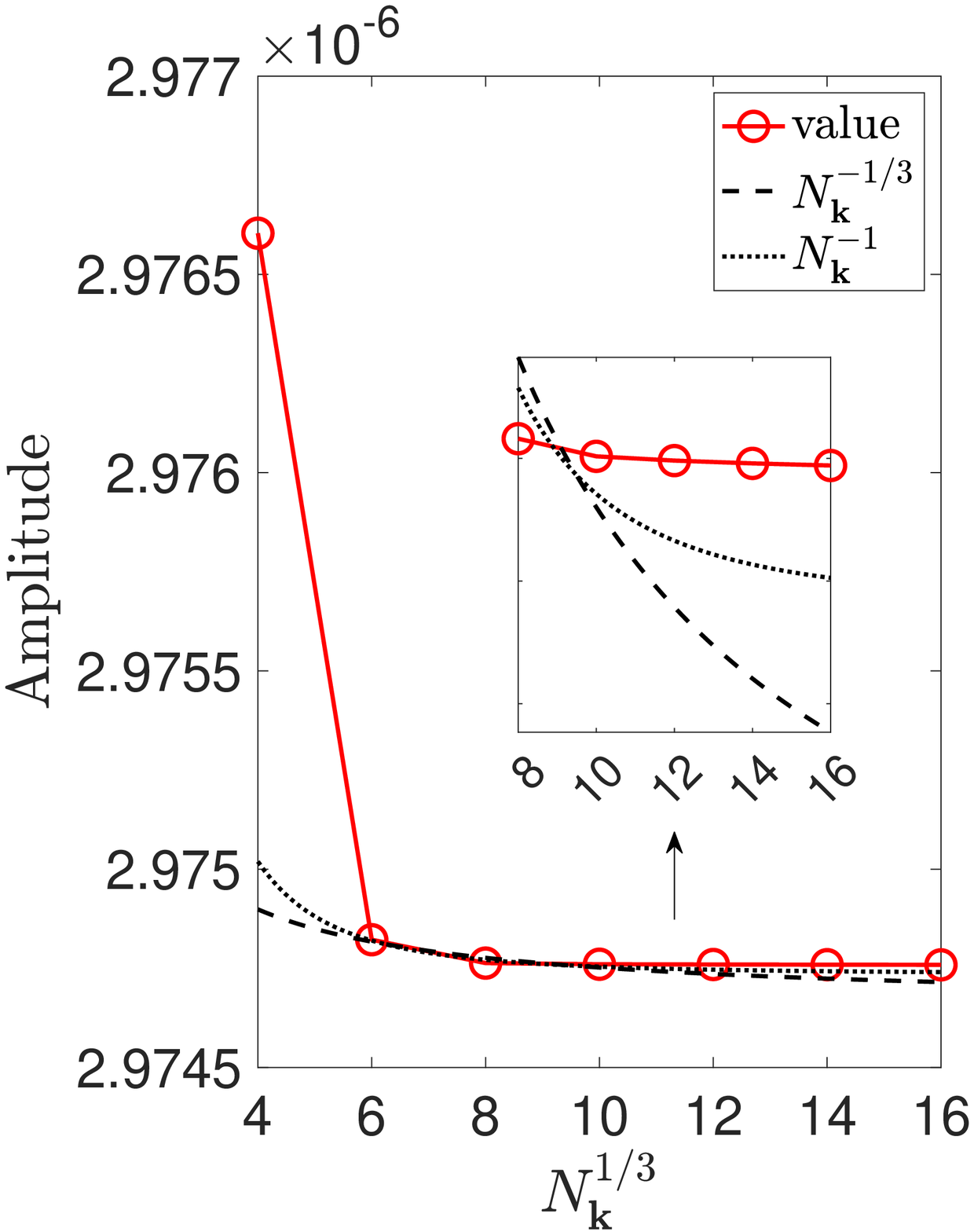}
        }
        \subfloat[$\sum_{KLCD}\braket{KL|CD}t_{IJ}^{CD}t_{KL}^{AB}$\label{fig:example6}]{
                \includegraphics[width=0.32\textwidth]{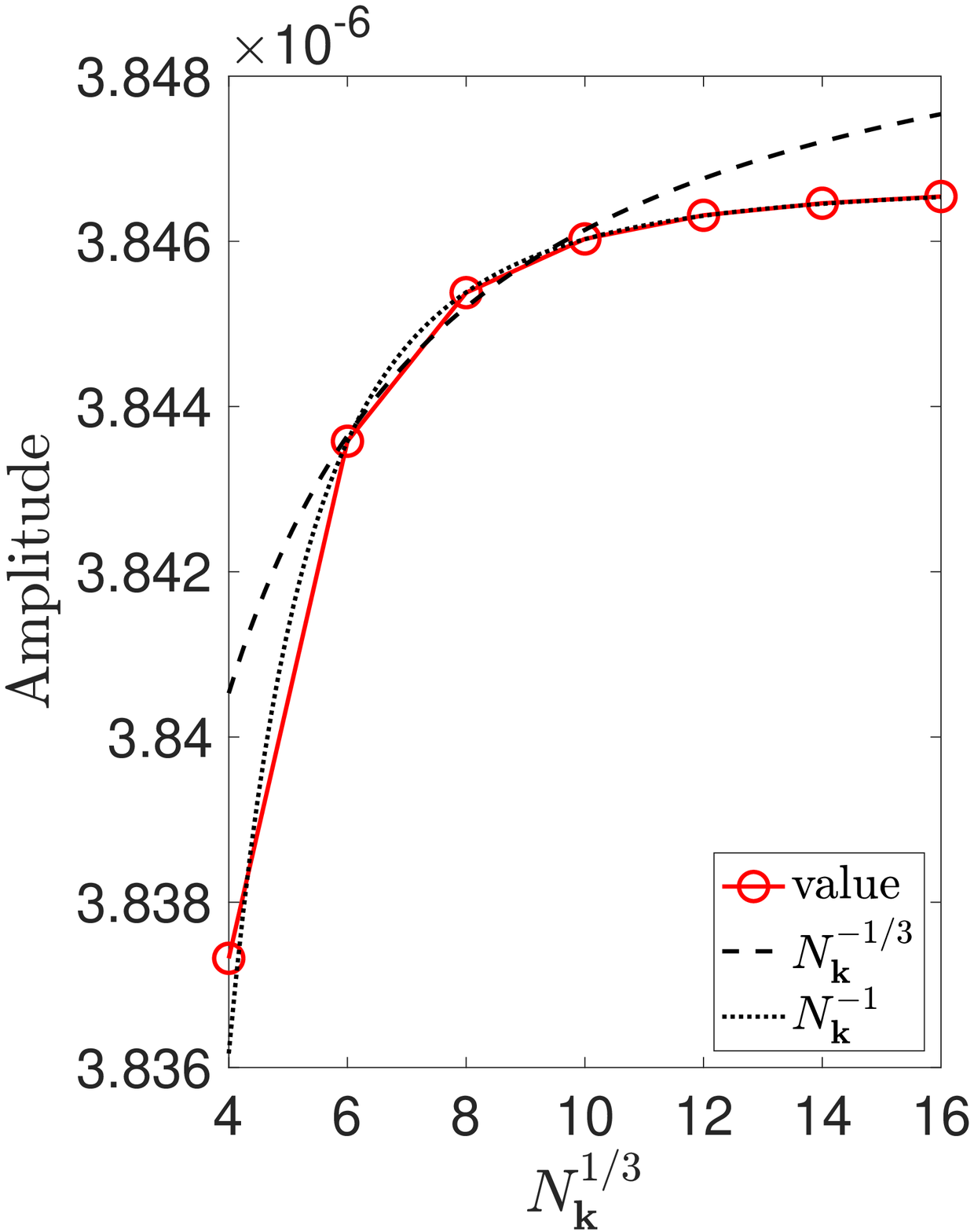}
        }
        \caption{Energy and amplitude calculations using exact CCD$(1)$ amplitude.
        All the amplitudes are evaluated at $\vk_i=\vk_j= (0,0,0),\vk_a=(0,0,\pi)$ and $(i,j,a,b) = (1,1,2,2)$. 
        The power-law extrapolations use the three data points at $N_\vk = 5^3,6^3,7^3$ in (a) and $N_\vk = 6^3, 8^3, 10^3$ in the remaining subplots.
        These calculations are estimated theoretically in \cref{tab:amplitude_error} to have quadrature errors decay asymptotically in the rate of $N_\vk^{-1}$, $N_\vk^{-\frac13}$, $N_\vk^{-1}$, super-algebraically, 
        super-algebraically, and $N_\vk^{-1}$, respectively. 
        In subplots (d) and (e), the actual data curve converges faster than the power-law extrapolations using $N_\vk^{-1}$ 
        and $N_\vk^{-\frac13}$. This can be seen as evidence that the quadrature errors decay super-algebraically. 
        \label{fig:example}
        }
\end{figure}

\section{Discussion}\label{sec:discuss}

We have investigated the convergence rate of the periodic coupled cluster theory calculations towards the thermodynamic limit. 
The analysis in this paper focuses on the simplest and representative CC theory, i.e., the coupled cluster doubles (CCD) theory. 
Since CCD consists of many finite order perturbation energy terms from M{\o}ller-Plesset perturbation theory, this  
also provides the first finite-size error analysis of these included perturbation energy terms, e.g., MP3. 
We interpret the finite-size error as numerical quadrature error.
The key steps include: (1) analyze the singularity structure, i.e., the \textit{algebraic singularity} of the integrand; 
(2) bound the quadrature error of the (univariate or multivariate) trapezoidal rule of singular integrands with certain algebraic singularity. 
Our quadrature analysis based on the Poisson summation formula for certain ``punctured'' trapezoidal rules and may be of independent interest in other contexts. 

Our main result in \cref{thm:error_ccd} studied the finite-size error in CCD$(n)$ calculation with any $n > 0$ number of fixed point iterations 
over the amplitude equation. 
However, for gapless and small-gap systems, it has been observed in practice that this fixed point iteration might not converge or the amplitude equation 
might have multiple solutions. 
In the case of divergence, the perturbative interpretation of CCD is not valid any more and our analysis over CCD$(n)$ can not be exploited to study the 
finite size error in the CCD energy calculation. 
In the case of multiple solutions, CCD itself is not well defined and nor is the problem of its finite-size error analysis. 

Our finite-size error analysis is not applicable to theories that do not directly rely on Hartree-Fock orbitals and orbital energies, such as tensor network methods or quantum Monte Carlo methods. The precise analysis of these methods may require a more detailed understanding of the behavior of structure factors~\cite{FraserFoulkesRajagopalEtAl1996,HolzmannClayMoralesEtAl2016}. For gapless systems (e.g., metals), additional singularities are introduced by the orbital energy differences of the form $(\varepsilon_{i\vk_i,j\vk_j}^{a\vk_a,b\vk_b})^{-1}$ and the occupation number near the Fermi surface. Quadrature error analysis in this case needs to take into account of these additional singularity structures, and the finite-size scaling in metals can also be qualitatively different from that in gapped systems~\cite{HolzmannBernuCeperley2011,GruberLiaoTsatsoulisEtAl2018,MihmYangShepherd2021}.

Our statement that the finite-size error in CCD energy calculation scales as $\Or(N_{\vk}^{-\frac13})$ may seem pessimistic compared to numerical results in 
the literature~\cite{LiaoGrueneis2016,GruberLiaoTsatsoulisEtAl2018}, which find the finite-size error of certain CC calculations can scale as $\Or(N_{\vk}^{-1})$ when the nonlinear amplitude equation is solved self-consistently. 
Our analysis is sharp for any constant number of iteration steps in the CCD$(n)$ scheme, under the assumption of the Hartree-Fock orbitals and orbital energies can be evaluated exactly at any given $\vk$ point. 
The orbital energies are needed in setting up the CC iterations (\cref{eqn:amplitude_ccd}), and the correction of finite-size errors in the occupied orbital energies are important for the accurate evaluation of the Fock exchange energy. 
However, since the abstract form of the CC amplitude equation (\cref{eqn:CC_amplitude_abstract}) can be statement without explicitly referring to  orbital energies, there may be a fortuitous error cancellation when the iteration scheme reaches self-consistency. 
Specifically, if CCD calculation uses inexact Hartree-Fock orbital energies without any correction, the special structure of CCD amplitude equations implies 
that this simple scheme can be equivalent to a more complex one, which simultaneously applies the Madelung constant correction to the Hartree-Fock orbital energies \cite{xing2021unified}, 
and the shifted Ewald kernel~\cite{FraserFoulkesRajagopalEtAl1996} correction to the ERIs. 
In other words, the finite-size error correction to the orbital energy alone may be detrimental in CCD theories.
The analysis of this more complex method is beyond the scope of this work. 
A viable path may be combining the quadrature error analysis with the singularity subtraction~method \cite{GygiBaldereschi1986,WenzienCappelliniBechstedt1995,xing2021unified} 
for simultaneous correction of the orbital energies and ERIs. 

While we have focused on the finite-size error of the ground state
energy, we think our quadrature based analysis includes some of the essential
ingredients in analyzing the finite-size errors for a wide range of
diagrammatic methods in quantum physics and quantum chemistry, such as
$n$-th order M{\o}ller-Plesset perturbation theory (MPn), GW, CCSD,
CCSD(T), and equation of motion coupled cluster (EOM-CC) theories.

\vspace{1em}
\noindent\textbf{Acknowledgement:}

This material is based upon work supported by the U.S. Department of Energy,
Office of Science, Office of Advanced Scientific Computing Research and Office
of Basic Energy Sciences, Scientific Discovery through Advanced Computing
(SciDAC) program (X.X.).  This work is also partially supported by the Air
Force Office of Scientific Research under award number FA9550-18-1-0095 (L.L.).
L.L. is a Simons Investigator.  We thank Timothy Berkelbach, Garnet Chan, and Alexander Sokolov for
insightful discussions.

%
%
%
%
%
%
%
%
%
%

\newpage
\bibliographystyle{abbrv}
\bibliography{mp2}

\newpage
\begin{appendices}
\crefalias{section}{appendix}

\section{Brief introduction of MP3 and CCD}\label{appendix:mp3ccd}

The double amplitude $T_{ijab}^{\#, N_\vk}(\vk_i, \vk_j, \vk_a)$ is commonly denoted in the literature as $t_{i\vk_i,j\vk_j}^{a\vk_a, b\vk_b}$
which assumes implicitly the crystal momentum conservation, 
\[
\vk_i + \vk_j - \vk_a - \vk_b \in \mathbb{L}^*,
\]
and a set of $\vk_i, \vk_j, \vk_a$ can determine a unique $\vk_b\in\Omega^*$ accordingly. 
This explains our notation of the double amplitude as a function of $\vk_i, \vk_j, \vk_a$ with discrete indices $i,j,a,b$. 
For brevity, we use a capital letter $P$ to denote an orbital index $(p, \vk_p)$, and use $P\in \{I, J, K, L\}$ to refer to occupied orbitals 
(a.k.a., holes) and $P\in\{A,B,C,D\}$ to refer to unoccupied orbitals (a.k.a., particles). 
Any summation $\sum_P$ refers to summing over all occupied or virtual band indices $p$ and all momentum vectors $\vk_p\in \mathcal{K}$ while 
the crystal momentum conservation is enforced according to the summand. 


\subsection{Amplitude in the finite case}
With a finite $\vk$-point mesh $\mathcal{K}$ of size $N_\vk$, the normalized MP3 amplitude $T_{ijab}^{\text{MP3}, N_\vk}(\vk_i, \vk_j, \vk_a)  = t_{IJ}^{AB}$ 
with $\vk_i, \vk_j, \vk_a \in \mathcal{K}$  is defined as
\begin{align}
        t_{IJ}^{AB}
        & 
        = \frac{1}{\varepsilon_{IJ}^{AB}} \braket{AB|IJ}+ \dfrac{1}{\varepsilon_{IJ}^{AB}}
        \Bigg[\dfrac{1}{N_\vk}\sum_{KL}\braket{KL|IJ} s_{KL}^{AB} + \dfrac{1}{N_\vk}\sum_{CD}\braket{AB|CD} s_{IJ}^{CD} 
        \nonumber
        \\
        & 
        + \mathcal{P} 
        \left(
        \dfrac{1}{N_\vk}\sum_{KC} 
        (2\braket{AK|IC} - \braket{AK|CI}) s_{KJ}^{CB} - \braket{AK|IC}s_{KJ}^{BC}- \braket{AK|CJ}s_{KI}^{BC}
        \right)
        \Bigg],
        \label{eqn:amplitude_mp3}
\end{align}
where $\varepsilon_{IJ}^{AB} = \varepsilon_I + \varepsilon_J - \varepsilon_A - \varepsilon_B$, $
s_{IJ}^{AB} = \braket{AB|IJ} / \varepsilon_{IJ}^{AB}$ is the normalized MP2 double amplitude,
and $\mathcal{P}$ is a permutation operator defined as
$
\mathcal{P}(\cdots)_{IJ}^{AB} = (\cdots)_{IJ}^{AB} + (\cdots)_{JI}^{BA}.
$
The three included summations are referred to as the 4-hole-2-particle (4h2p), 2-hole-4-particle (2h4p), and 3-hole-3-particle (3h3p) terms in MP3 
according to the number of dummy occupied and virtual orbitals involved. 
We note that the summation over each $P$ implicitly enforces the   crystal momentum conservation.
For example, the MP3-4h2p amplitude is explicitly written as
\[
\dfrac{1}{\varepsilon_{IJ}^{AB}}
\dfrac{1}{N_\vk}
\sum_{KL}\braket{KL|IJ}s_{KL}^{AB}
= 
\dfrac{1}{\varepsilon_{i\vk_i,j\vk_j}^{a\vk_a,b\vk_b}}
\dfrac{1}{N_\vk}\sum_{\vk_k\in\mathcal{K}}\sum_{kl}  \braket{k\vk_k, l\vk_l |  i\vk_i, j\vk_j} s_{k\vk_k, l\vk_l}^{a\vk_a, b\vk_b},
\]
where $\vk_l \in \mathcal{K}$ is uniquely determined by $\vk_i, \vk_j, \vk_k$.  

In CCD theory with a finite mesh $\mathcal{K}$, the wavefunction is represented in an exponential ansatz as 
\[
\ket{\Psi} = e^{\mathcal{T}}  \ket{\Phi} := \exp\left(\dfrac{1}{N_\vk}\sum_{IJAB}t_{IJ}^{AB} a_{A}^\dagger a_{B}^\dagger a_J a_I\right) \ket{\Phi},
\] 
where $a_P^\dagger$ and $a_P$ are creation and annihilation operators, $\ket{\Phi}$ is the reference Hartree-Fock determinant, 
and $t_{IJ}^{AB} = T_{ijab}^{\text{CCD}, N_\vk}(\vk_i, \vk_j, \vk_a)$ is the \textit{normalized} CCD double amplitude.
The double amplitude satisfies the amplitude equation (which is derived from the Galerkin projection) as
\begin{equation}
\braket{\Phi_{IJ}^{AB}, e^{-\mathcal{T}} \mathcal{H}_\mathcal{K} e^{\mathcal{T}} \Phi}  = 0, \qquad \forall I, J, A, B,       
\label{eqn:CC_amplitude_abstract}
\end{equation}
where $\ket{\Phi_{IJ}^{AB}} = a_{A}^\dagger a_{B}^\dagger a_J a_I \ket{\Phi}$ is an excited single Slater determinant and $\mathcal{H}_\mathcal{K}$ is the model 
Hamiltonian with $\vk$-point mesh $\mathcal{K}$. 
In practice, this nonlinear amplitude equation can be solved using a quasi-Newton method \cite{Schneider2009}, which can be equivalently written in the form of a fixed point iteration as
\begin{align}
        t_{IJ}^{AB}
        & = 
        \frac{1}{\varepsilon_{IJ}^{AB}} \braket{AB|IJ} + 
        \frac{1}{\varepsilon_{IJ}^{AB}} \mathcal{P}
        \left(
        \sum_C\kappa_C^A t_{IJ}^{CB}
        -\sum_K\kappa_I^K t_{KJ}^{AB} 
        \right)
        + 
        \frac{1}{\varepsilon_{IJ}^{AB}} 
        \Bigg[
        \dfrac{1}{N_\vk}\sum_{KL}\chi_{IJ}^{KL} t_{KL}^{AB} 
        \nonumber       
        \\
        & 
        +\dfrac{1}{N_\vk}\sum_{CD}\chi_{CD}^{AB} t_{IJ}^{CD} 
        + \mathcal{P} 
        \left(\dfrac{1}{N_\vk}\sum_{KC} (2\chi_{IC}^{AK} - \chi_{CI}^{AK}) t_{KJ}^{CB} - \chi_{IC}^{AK}t_{KJ}^{BC} - \chi_{CJ}^{AK}t_{KI}^{BC}
        \right)
        \Bigg].
        \label{eqn:amplitude_ccd}
\end{align}
This reformulation of CCD amplitude equation can also be derived from the CCSD amplitude equation in \cite{HirataPodeszwaTobitaEtAl2004} by removing all the terms related to single amplitudes and normalizing  
the involved ERIs and amplitudes (which gives the extra $1/N_\vk$ factor in the equation and the intermediate blocks). 
The intermediate blocks in the equation are defined as
\begin{align*}
        \kappa_C^A & = -\dfrac{1}{N_\vk^2}\sum_{KLD} \left( 2\braket{KL|CD} - \braket{KL|DC} \right) t_{KL}^{AD},
        \\
        \kappa_I^K & = \dfrac{1}{N_\vk^2}\sum_{LCD} \left( 2\braket{KL|CD} - \braket{KL|DC} \right) t_{IL}^{CD},
        \\
        \chi_{IJ}^{KL} & = \braket{KL|IJ} + \dfrac{1}{N_\vk}\sum_{CD} \braket{KL|CD} t_{IJ}^{CD},
        \\
        \chi_{CD}^{AB} & = \braket{AB|CD},
        \\
        \chi_{IC}^{AK} & = \braket{AK|IC} + \dfrac{1}{2N_\vk}\sum_{LD} \left(2\braket{LK|DC} - \braket{LK|CD}\right) t_{IL}^{AD} - \braket{LK|DC} t_{IL}^{DA},
        \\
        \chi_{CI}^{AK} &= \braket{AK|CI} - \dfrac{1}{2N_\vk}\sum_{LD} \braket{LK|CD} t_{IL}^{DA},
\end{align*}
and their momentum vector indices also assume the crystal momentum conservation
\begin{align*}
        \kappa_{P}^{Q} & \rightarrow\vk_p - \vk_q \in \mathbb{L}^*,\\
        \chi_{PQ}^{RS} & \rightarrow \vk_p + \vk_q - \vk_r - \vk_s \in \mathbb{L}^*. 
\end{align*}

\subsection{Amplitude in the TDL}
In the TDL with $\mathcal{K}$ converging to $\Omega^*$, all the averaged summation $N_\vk^{-1}\sum_{\vk\in\mathcal{K}}$ converge to 
integration $|\Omega^*|^{-1}\int_{\Omega*}\ud\vk$ in MP3 and CCD. 
It is worth noting that the double amplitude is computed approximately on $\mathcal{K}\times \mathcal{K}\times\mathcal{K}$ as a tensor in the finite 
case and it converges to a function of $\vk_i, \vk_j, \vk_a$ defined in $\Omega^*\times\Omega^*\times \Omega^*$ in the TDL. 
In MP3, the exact amplitude $T^{\text{MP3}, \text{TDL}}_{ijab}(\vk_i, \vk_j, \vk_a) = t_{IJ}^{AB}$ with any $\vk_i, \vk_j, \vk_a \in \Omega^*$ 
can be formulated according to \cref{eqn:amplitude_mp3} as
\begin{align}
        & 
        t_{IJ}^{AB}
        = \frac{1}{\varepsilon_{IJ}^{AB}} \braket{AB|IJ}+ \dfrac{1}{\varepsilon_{IJ}^{AB}}
        \Bigg[\dfrac{1}{|\Omega^*|}\int_{\Omega^*}\ud\vk_k \sum_{kl}\braket{KL|IJ} s_{KL}^{AB} + \dfrac{1}{|\Omega^*|}\int_{\Omega^*}\ud\vk_c 
                \sum_{cd}\braket{AB|CD} s_{IJ}^{CD}
        \nonumber 
        \\
        &  
        \quad 
        + \mathcal{P} 
        \left(
        \dfrac{1}{|\Omega^*|}\int_{\Omega^*}\ud\vk_k \sum_{kc}
        (2\braket{AK|IC} - \braket{AK|CI}) s_{KJ}^{CB} - \braket{AK|IC}s_{KJ}^{BC}- \braket{AK|CJ}s_{KI}^{BC}
        \right)
        \Bigg].
        \label{eqn:amplitude_mp3_tdl}
\end{align}

Similarly in CCD, the amplitude equation in the TDL for the exact amplitudes as functions of $\vk_i, \vk_j, \vk_a \in \Omega^*$ 
can be formulated according to \cref{eqn:amplitude_ccd} as 
\begin{align}
        & t_{IJ}^{AB}
         = 
        \frac{1}{\varepsilon_{IJ}^{AB}} \braket{AB|IJ} + 
        \frac{1}{\varepsilon_{IJ}^{AB}} \mathcal{P}
        \left(
        \sum_C\kappa_C^A t_{IJ}^{CB}
        -\sum_K\kappa_I^K t_{KJ}^{AB} 
        \right)
        + 
        \frac{1}{\varepsilon_{IJ}^{AB}} 
        \Bigg[
        \dfrac{1}{|\Omega^*|}\int_{\Omega^*}\ud\vk_k \sum_{kl}\chi_{IJ}^{KL} t_{KL}^{AB}
        \nonumber       \\
        & 
        \quad 
        + 
        \dfrac{1}{|\Omega^*|}\int_{\Omega^*}\ud\vk_c \sum_{cd}\chi_{CD}^{AB} t_{IJ}^{CD} 
        + \mathcal{P} 
        \left(
                        \dfrac{1}{|\Omega^*|}\int_{\Omega^*}\ud\vk_k \sum_{kc} (2\chi_{IC}^{AK} - \chi_{CI}^{AK}) t_{KJ}^{CB} - \chi_{IC}^{AK}t_{KJ}^{BC} - 
                        \chi_{CJ}^{AK}t_{KI}^{BC}
        \right)
        \Bigg],
        \label{eqn:amplitude_ccd_tdl}
\end{align}
where the intermediate blocks in the TDL are defined as 
\begin{align*}
        \kappa_C^A & = -\dfrac{1}{|\Omega^*|^2}\int_{\Omega^*\times \Omega^*}\ud\vk_k\ud\vk_l \sum_{kld} \left( 2\braket{KL|CD} - \braket{KL|DC} \right) t_{KL}^{AD},
        \\
        \kappa_I^K & = \dfrac{1}{|\Omega^*|^2}\int_{\Omega^*\times \Omega^*}\ud\vk_c\ud\vk_d  \sum_{lcd} \left( 2\braket{KL|CD} - \braket{KL|DC} \right) t_{IL}^{CD},
        \\
        \chi_{IJ}^{KL} & = \braket{KL|IJ} + \dfrac{1}{|\Omega^*|}\int_{\Omega^*}\ud\vk_c \sum_{cd} \braket{KL|CD} t_{IJ}^{CD},
        \\
        \chi_{CD}^{AB} & = \braket{AB|CD},
        \\
        \chi_{IC}^{AK} & = \braket{AK|IC} + \dfrac{1}{2|\Omega^*|}\int_{\Omega^*}\ud\vk_l \sum_{ld}  \left(2\braket{LK|DC} - \braket{LK|CD}\right) t_{IL}^{AD} - \braket{LK|DC} t_{IL}^{DA} ,
        \\
        \chi_{CI}^{AK} &= \braket{AK|CI} - \dfrac{1}{2|\Omega^*|}\int_{\Omega^*}\ud\vk_l \sum_{ld} \braket{LK|CD} t_{IL}^{DA}. 
\end{align*}

\subsection{Amplitude in CCD($n$)}

In this paper, we use CCD$(n)$ to refer to solving the CCD amplitude approximately by applying $n$ fixed point iterations with a zero initial guess to  
the amplitude equation and then using the obtained amplitude to compute an approximate CCD energy. 
In the finite case, the initial amplitude for CCD$(0)$ is set as zero and when $n=1$ we have
\begin{equation*}
T^{\text{CCD}(1), N_\vk}_{ijab}(\vk_i, \vk_j, \vk_a)
 = \frac{1}{\varepsilon_{IJ}^{AB}} \braket{AB|IJ}.
\end{equation*}
Therefore, CCD$(1)$ can be identified with MP2.
In CCD($2$) calculation, the CCD$(1)$ amplitude is plugged into the right hand side of \cref{eqn:amplitude_ccd} where the constant term plus all the linear 
terms exactly gives the MP3 amplitude in \cref{eqn:amplitude_mp3} and all the quadratic terms belong to the MP4 amplitude. 
Therefore, CCD$(2)$ contains all terms in MP2 and MP3, as well as a subset of MP4.

At the $n$th iteration, we plug the CCD$(n-1)$ amplitude from the $(n-1)$th iteration into the right hand side of the amplitude equation in
\cref{eqn:amplitude_ccd} and the left hand side gives the CCD$(n)$ amplitude, i.e., 
\begin{align*}
&T^{\text{CCD}(n), N_\vk}_{ijab}(\vk_i, \vk_j, \vk_a)
 = 
\frac{1}{\varepsilon_{IJ}^{AB}} \braket{AB|IJ} + 
\frac{1}{\varepsilon_{IJ}^{AB}} \mathcal{P}
\left(
\sum_C\kappa_C^A T^{\text{CCD}(n-1), N_\vk}_{ijcb}(\vk_i, \vk_j, \vk_c)
\right.
\\
&\hspace{4em}\left.
 -\sum_K\kappa_I^KT^{\text{CCD}(n-1), N_\vk}_{kjab}(\vk_k, \vk_j, \vk_a)
\right)
+ 
\frac{1}{\varepsilon_{IJ}^{AB}} 
\dfrac{1}{N_\vk}\sum_{KL}\chi_{IJ}^{KL} T^{\text{CCD}(n-1), N_\vk}_{klab}(\vk_k, \vk_l, \vk_a)+ 
\cdots
\end{align*}
where all the involved intermediate blocks are now computed using CCD$(n-1)$ amplitude, e.g., 
\begin{align*}
\kappa_C^A &=\dfrac{1}{N_\vk^2} \sum_{KLD} \left( 2\braket{KL|CD} - \braket{KL|DC} \right)T^{\text{CCD}(n-1), N_\vk}_{klad}(\vk_k, \vk_l, \vk_c), \\
\chi_{IJ}^{KL} &= \braket{KL|IJ} + \frac{1}{N_\vk}\sum_{CD} \braket{KL|CD} T^{\text{CCD}(n-1), N_\vk}_{ijcd}(\vk_i, \vk_j, \vk_c).
\end{align*}

In the finite case, by unfolding the fixed point iteration, the CCD$(n)$ amplitude consists of many averaged summations of products of ERIs 
and orbital energy fractions which correspond to the double amplitudes of certain perturbation terms in M{\o}ller-Plesset perturbation 
theory \cite{ShavittBartlett2009} with order up to $2^n$.
Each averaged summation in the unfolded CCD$(n)$ amplitude converges in the TDL to an integral over involved intermediate momentum vectors in $\Omega^*$. 
As a result, the CCD$(n)$ amplitude in the TDL could be explicitly formulated as a summation of many integrals which are respectively approximated by 
trapezoidal rules in the finite case. 

On the other hand, the CCD$(n)$ amplitude in the TDL can also be defined recursively by applying $n$ fixed point iterations to the amplitude equation 
in the TDL in \cref{eqn:amplitude_ccd_tdl}, i.e., 
\begin{align}
        &T^{\text{CCD}(n), \text{TDL}}_{ijab}(\vk_i, \vk_j, \vk_a)
        = 
        \frac{1}{\varepsilon_{IJ}^{AB}} \braket{AB|IJ} + 
        \frac{1}{\varepsilon_{IJ}^{AB}} \mathcal{P}
        \left(
        \sum_c\kappa_{c\vk_a}^{a\vk_a} T^{\text{CCD}(n-1), \text{TDL}}_{ijcb}(\vk_i, \vk_j, \vk_a)
        \right.
        \label{eqn:amplitude_tdl_iterative}
        \\
        &\hspace{1em}\left.
        -\sum_k\kappa_{i\vk_i}^{k\vk_i}T^{\text{CCD}(n-1), \text{TDL}}_{kjab}(\vk_i, \vk_j, \vk_a)
        \right)
        + 
        \frac{1}{\varepsilon_{IJ}^{AB}} 
        \dfrac{1}{|\Omega^*|}\int_{\Omega^*}\ud\vk_k \sum_{kl}\chi_{IJ}^{KL} T^{\text{CCD}(n-1), \text{TDL}}_{klab}(\vk_k, \vk_l, \vk_a)+ 
        \cdots
        \nonumber
\end{align}
where all the involved intermediate blocks are computed using exact CCD$(n-1)$ amplitude.
Unfolding this fixed point iteration, it can be verified that this recursive definition of the amplitude in the TDL is consistent with the above definition 
obtained  by taking the thermodynamic limit of each individual averaged summation term in the double amplitude in the finite case.

\section{Proof of Theorem~\ref{thm:error_ccd}}\label{app:theorem0}

\subsection{Proof of \cref{lem:thm1_nonsmooth}: Singularity structure of the exact CCD$(n)$ amplitude} \label{lemproof:thm1_nonsmooth}

Based on the smoothness property of ERIs in \cref{sub:fractional}, it can be verified that each CCD$(1)$ amplitude entry, i.e., the MP2 amplitude with any $(i,j,a,b)$, 
lies in $\mathbb{T}(\Omega^*)$ and thus satisfies the statement in the lemma. 
Since the exact CCD$(n)$ amplitude with $n > 1$ is defined by recursively applying $\mathcal{F}_\text{TDL}$ in \cref{eqn:fixed_point_tdl} to the CCD$(1)$ amplitude, it is sufficient to prove that 
$
        \mathcal{F}_\text{TDL}(t)\in \mathbb{T}(\Omega^*)^{n_\text{occ}\times n_\text{occ}\times n_\text{vir}\times n_\text{vir}},
        \forall t \in \mathbb{T}(\Omega^*)^{n_\text{occ}\times n_\text{occ}\times n_\text{vir}\times n_\text{vir}}.
$

Consider an arbitrary $t\in \mathbb{T}(\Omega^*)^{n_\text{occ}\times n_\text{occ}\times n_\text{vir}\times n_\text{vir}}$. 
Fixing a set of $(i,j,a,b)$, we focus on analyzing the constant, linear, and quadratic terms included in the entry $\left[\mathcal{F}_\text{TDL}(t)\right]_{ijab}$
\begin{align}
        & \left[\mathcal{F}_\text{TDL}(t)\right]_{ijab}(\vk_i, \vk_j, \vk_a)
        = 
        \dfrac{1}{\varepsilon_{i\vk_i, j\vk_j}^{a\vk_a, b\vk_b}}\braket{a\vk_a, b\vk_b | i\vk_i, j\vk_j}
        \nonumber\\
        & 
        \qquad + 
        \dfrac{1}{\varepsilon_{i\vk_i, j\vk_j}^{a\vk_a, b\vk_b}}
        \dfrac{1}{|\Omega^*|}\int_{\Omega^*}\ud\vk_k \sum_{kl} \braket{k\vk_k, l\vk_l | i\vk_i, j\vk_j} t_{klab}(\vk_k, \vk_l, \vk_a)
        \nonumber\\
        & 
        \qquad + 
        \dfrac{1}{\varepsilon_{i\vk_i, j\vk_j}^{a\vk_a, b\vk_b}}
        \dfrac{1}{|\Omega^*|^2}\int_{\Omega^*}\ud\vk_k\int_{\Omega^*}\ud\vk_c \sum_{klcd} \braket{k\vk_k, l\vk_l | c\vk_c, d\vk_d} 
        t_{ijcd}(\vk_i, \vk_j, \vk_c)t_{klab}(\vk_k, \vk_l, \vk_a)
        \nonumber \\
        & \qquad + \cdots, 
        \label{eqn:amplitude_map_expansion}
\end{align}
where the listed linear and quadratic terms come from the 4h2p term $\chi_{IJ}^{KL}t_{KL}^{AB}$ in \cref{eqn:amplitude_ccd_tdl}
and the neglected terms include all the remaining linear and quadratic terms. 
Our goal is to prove that each of these terms as a function of $\vk_i,\vk_j,\vk_a$ is in $\mathbb{T}(\Omega^*)$. 
It can be verified directly that these terms satisfy the periodicity condition described in $\mathbb{T}(\Omega^*)$.
Therefore we focus on showing that these terms are smooth everywhere except at $\vq = \vk_a - \vk_i = \bm{0}$ with order $0$ 
and is smooth with respect to $\vk_i, \vk_j \in \Omega^*$ when $\vq = \bm{0}$. 
We recall the algebraic singularity for multivariate functions in \cref{def:fractional2} that a periodic function 
$f(\vk_i, \vk_j, \vk_a)$ is smooth everywhere in except at $\vk_a = \vk_i$ with order $\gamma$ if with the change of variable $\vk_a \rightarrow \vk_i + \vq$, 
there exists constants $\{C_{\valpha,\vbeta}\}$ satisfying
\begin{equation}\label{eqn:amplitude_fractional}
        \left|
                \dfrac{\partial^\valpha}{\partial \vq^{\valpha}} \left( \dfrac{\partial^\vbeta}{\partial(\vk_i,\vk_j)^\vbeta}
                f(\vk_i, \vk_j, \vq)  \right)
        \right|
        \leqslant 
        C_{\valpha,\vbeta} |\vq|^{\gamma - |\valpha|},  \quad \forall \vq \in \Omega^*\setminus\{\bm{0}\},\vk_i,\vk_j\in \Omega^*, \forall \valpha,\vbeta\geqslant 0,
\end{equation}
where the inequality is extended to all $\vq \in \Omega^*\setminus \{\bm{0}\}$ by using the function smoothness.

The constant term is exactly an MP2/CCD(1) amplitude entry and lies in $\mathbb{T}(\Omega^*)$. 

All the linear terms takes the form of an integral over an intermediate momentum vector in $\Omega^*$, and 
the integrand is products of one ERI and one amplitude entry (see \cref{eqn:amplitude_map_expansion} for an example).
These linear terms can be categorized into three classes according to the number of singular points of the integrand with respect to the intermediate momentum vector, 
see \cref{tab:class_linear_term}.
The analysis of smoothness properties with respect to $\vk_i, \vk_j, \vk_a$ is similar for terms in the same class. 
Below we illustrate the analysis for one example in each class. 

\begin{table}[htbp]
        \caption{Classification of linear amplitude terms by the number of nonsmooth points with respect to the intermediate momentum vector.
        The 3h3p terms permuted by $\mathcal{P}$ are of the same class as the unpermuted ones and thus not listed. 
        The summations over intermediate orbitals, e.g., $K,L,C,D$, and the prefactor $1/\varepsilon_{IJ}^{AB}$ are omitted for brevity.
        }\label{tab:class_linear_term}
        \begin{tabular}{cc}
                \toprule
                \makecell{
                Number of singular
points
                } & Linear terms  \\
                \midrule
                0 & $\braket{AK|IC}t_{KJ}^{CB}$\\
                1 & $\braket{AK|CI}t_{KJ}^{CB}$, $\braket{AK|IC}t_{KJ}^{BC}$\\
                2 & $\braket{KL|IJ}t_{KL}^{AB}$, $\braket{AB|CD}t_{IJ}^{CD}$, $\braket{AK|CJ}t_{IK}^{CB}$\\
                \bottomrule
        \end{tabular}
\end{table}

For linear terms with no singular point, we consider the 3h3p term $\braket{AK|IC}t_{KJ}^{CB}$ detailed as (ignoring the prefactor and orbital energy fraction), 
\[
        \int_{\Omega^*}\ud\vk_k \sum_{kc} \braket{a\vk_a, k\vk_k | i\vk_i, c\vk_c} t_{kjcb}(\vk_k, \vk_j, \vk_c),
\]
where $\vk_c = \vk_a + \vk_k - \vk_i = \vk_k + \vq$.
The ERIs and the amplitudes have momentum transfers $\vk_i - \vk_a = -\vq$ and $\vk_c - \vk_k = \vq$, respectively, which are both independent of $\vk_k$.
As a result, for any $(\vk_i,\vk_j,\vq) \in (\Omega^*)^{\times 3}$ with $\vq \neq \bm{0}$, there exists an open domain containing this point where the above integrand 
is smooth with respect to $\vk_i, \vk_j, \vq$ and $\vk_k\in\Omega^*$. 
This meets the condition of the Leibniz integral rule which can then be used to prove that this integral is smooth at all points $(\vk_i, \vk_j, \vq)\in(\Omega^*)^{\times 3}$ 
with $\vq \neq \bm{0}$ and any derivative of this integral equals to the integral of the corresponding integrand derivatives. 
The algebraic singularity condition in \cref{eqn:amplitude_fractional} for this term at any $\vk_i,\vk_j, \vq\in\Omega^*$ with $\vq \neq \bm{0}$ can then 
be verified as 
\begin{align}
        \left|
        \partial^\valpha_{\vq} \partial^\vbeta_{\vk_i\vk_j}
        \int_{\Omega^*}\ud\vk_k \cdots 
        \right|
        & = 
        \left|
        \int_{\Omega^*}\ud\vk_k  
        \partial^\valpha_{\vq} \partial^\vbeta_{\vk_i\vk_j}
        \cdots
        \right|
        \nonumber \\
        & \leqslant 
        C 
        \int_{\Omega^*}\ud\vk_k  
        \sum_{kc}
        \left|
        \sum_{\substack{
                \valpha_0+\valpha_1 =\valpha \\
                \vbeta_0+\vbeta_1=\vbeta}
        }
        \partial^{\valpha_0}_{\vq} \partial^{\vbeta_0}_{\vk_i\vk_j}
        \braket{a\vk_a, k\vk_k | i\vk_i, c\vk_c} 
        \partial^{\valpha_1}_{\vq} \partial^{\vbeta_1}_{\vk_i\vk_j}
        t_{kjcb}(\vk_k, \vk_j, \vk_c)
        \right|
        \nonumber \\
        & \leqslant 
        C 
        \int_{\Omega^*}\ud\vk_k  
        \sum_{kc}
        \left|
        \sum_{\substack{
                \valpha_0+\valpha_1 =\valpha \\
                \vbeta_0+\vbeta_1=\vbeta}
        }
        C_{\valpha, \vbeta}
        |\vq|^{0-|\valpha_0|}
        |\vq|^{0-|\valpha_1|}
        \right|
        \nonumber\\
        & \leqslant 
        C_{\valpha, \vbeta} |\vq|^{-|\valpha|},
        \label{eqn:frac_estimate1}
\end{align}
where $C_{\valpha, \vbeta}$ denotes a generic constant depending on $\valpha, \vbeta$ and the third inequality uses the 
algebraic singularity of the ERIs and the amplitudes at $\vq = \bm{0}$ with order $0$. 
Lastly, the ERI terms at $\vq=\bm{0}$ are smooth with respect to $\vk_i, \vk_j$ (see \cref{eqn:eri}) and so are the amplitudes by the assumption 
$t_{kjcb} \in \mathbb{T}(\Omega^*)$. 
We thus can use the Leibniz integral rule to prove that the integral at $\vq = \bm{0}$ is smooth with respect to $\vk_i,\vk_j$.
The above discussion then shows this integral to be in $\mathbb{T}(\Omega^*)$.

For linear terms with one nonsmooth point, we consider the 3h3p term $\braket{AK|CI}t_{KJ}^{CB}$ detailed as
\[
        \int_{\Omega^*}\ud\vk_k \sum_{kc} \braket{a\vk_a, k\vk_k | c\vk_c, i\vk_i} t_{kjcb}(\vk_k, \vk_j, \vk_c),
\]
where the ERIs and the amplitudes have momentum transfers $\vk_k - \vk_i$ and $\vk_c - \vk_k = \vq$, respectively. 
By the change of variable $\vk_k \rightarrow \vk_i + \vq_1$ and the integrand periodicity, this term can be reformulated as 
\[
        \int_{\Omega^*}\ud\vq_1 \sum_{kc} \braket{a\vk_a, k(\vk_i+\vq_1) | c(\vk_a +\vq_1), i\vk_i} t_{kjcb}(\vk_i+\vq_1, \vk_j, \vk_a+\vq_1),
\]
where the integrand is nonsmooth at $\vq_1 = \bm{0}$ due to the ERIs and is asymptotically of scale $\Or(1/|\vq_1|^2)$ near $\vq_1 =\bm{0}$. 
As a result, for any $(\vk_i,\vk_j,\vq) \in (\Omega^*)^{\times 3}$ with $\vq \neq \bm{0}$, there exists an open domain containing this point where 
the concerned integrand is smooth with respect to $\vk_i, \vk_j, \vq$ and $\vq_1\in\Omega^*\setminus\{\bm{0}\}$
and its absolute value is bounded by $C/|\vq_1|^2$ from above which is integrable in $\Omega^*$.
This still meets the condition of the Leibniz integral rule which can then be used to prove that this integral is smooth at all points 
$(\vk_i,\vk_j,\vq) \in (\Omega^*)^{\times 3}$ with $\vq \neq \bm{0}$. 
The algebraic singularity of the integral at $\vq = \bm{0}$ can be similarly proved as in \cref{eqn:frac_estimate1}, 
except now that the ERI derivatives are estimated as 
\[
        \left| 
                \partial^{\valpha_0}_{\vq} \partial^{\vbeta_0}_{\vk_i\vk_j}
                \braket{a\vk_a, k(\vk_i+\vq_1) | c\vk_c, i\vk_i} 
        \right|
        \leqslant C_{\valpha, \vbeta} /|\vq_1|^2, 
\]
by noting that the ERI here has momentum transfer $\vq_1$ and is smooth with respect to $\vk_i,\vk_j,\vq$. 

For linear terms with two nonsmooth points, we consider the 4h2p linear term $\braket{KL|IJ}t_{KL}^{AB}$ as detailed in \cref{eqn:amplitude_map_expansion}. 
We first denote the ERI and the amplitude with band indices $(k,l)$ as  
\begin{align*}
        F^{kl}_1(\vq_1, \vk_i, \vk_j, \vk_a) 
        & =     
        \braket{k(\vk_i - \vq_1), l(\vk_j + \vq_1)| i\vk_i,j \vk_j},
        \\
        F^{kl}_2(\vq_2, \vk_i, \vk_j, \vk_a)
        & = 
        t_{klab}(\vk_a - \vq_2, \vk_i+\vk_j - \vk_a + \vq_2, \vk_a),
\end{align*}
where $\vq_1$ and $\vq_2$ are the respective momentum transfers of the two terms.
Note that $F_1^{kl}$ does not depend on $\vk_a$ which is included as a variable for general cases. 
Both the ERI and the amplitude depend on $\vk_k$ and $\vk_l$ and these dependencies will be converted to that of $\vq_1, \vq_2$ 
using change of variables and the crystal momentum conservation. 
For example, we have 
$t_{klab}(\vk_k, \vk_i+\vk_j - \vk_k, \vk_a)=t_{klab}(\vk_k, \vk_l, \vk_a)=t_{klab}(\vk_a - \vq_2, \vk_i+\vk_j - \vk_a + \vq_2, \vk_a)$.

Note that $F_1^{kl}$ can be verified to be periodic and smooth everywhere with respect to $\vk_i,\vk_j,\vk_a,\vq_1 \in\Omega^*$ except at 
$\vq_1 = \bm{0}$ with order $0, -1,$ or $-2$ depending on the relation between $(k,l)$ and $(i,j)$. 
Similarly, $F_2^{kl}$ is periodic and smooth everywhere with respect to $\vk_i, \vk_j, \vk_a, \vq_2 \in \Omega^*$ except at $\vq_2 = \bm{0}$ 
with order $0$ by the assumption $t_{klab}\in\mathbb{T}(\Omega^*)$. 
Using this notation, the 4h2p linear term can be reformulated as 
\begin{align*}
        & \dfrac{1}{|\Omega^*|}\int_{\Omega^*}\ud\vk_k \sum_{kl} F^{kl}_1(\vk_i - \vk_k, \vk_i, \vk_j, \vk_a) 
        F^{kl}_2(\vk_a - \vk_k, \vk_i, \vk_j, \vk_a)
        \\
        & \qquad = 
        \dfrac{1}{|\Omega^*|}\int_{\Omega^*}\ud\vq_1 \sum_{kl} F^{kl}_1(\vq_1, \vk_i, \vk_j, \vk_a) 
        F^{kl}_2( \vq_1 - (\vk_i - \vk_a), \vk_i, \vk_j, \vk_a),
\end{align*}
where the second equation applies $\vk_k \rightarrow \vk_i - \vq_1$ and uses the integrand periodicity with respect to $\vk_k\in\Omega^*$. 
Due to the integrand being nonsmooth at $\vq_1 = \vk_i - \vk_a$ and $\vq_1 = \bm{0}$, the smoothness property of the integral with respect to $\vk_i,\vk_j,\vk_a$ 
cannot be obtained using the Leibniz integral rule as in the previous two cases. 
We provide a technical lemma analyzing the singularity structure of such a function in integral form.
\begin{lem}\label{lem:nonsmooth_integral}
        Let $f(\vx_1, \vx_2, \vz)$ be defined in $V\times V\times V_Z$ with $V = [-\frac12, \frac12]^d$ and $V_Z\subset \mathbb{R}^{d_z}$ of arbitrary dimension $d_z$. 
        Assume $f(\vx_1,\vx_2,\vz)$ is periodic with respect to $\vx_1, \vx_2 \in V$ and smooth everywhere except at
        $\vx_1 = \bm{0}$ or $\vx_2 = \bm{0}$ where the nonsmooth behavior can be characterized as
        \begin{equation}\label{eqn:singular_behavior}
                \left|
                \dfrac{\partial^{\valpha_1}}{\partial \vx_1^{\valpha_1}} 
                \dfrac{\partial^{\valpha_2}}{\partial \vx_2^{\valpha_2}} 
                \left(
                        \dfrac{\partial^{\vbeta}}{\partial \vz^{\vbeta}} 
                        f(\vx_1, \vx_2, \vz)  
                \right)
                \right|
                \leqslant 
                C_{\valpha_1, \valpha_2, \vbeta} |\vx_1|^{\gamma_1 - |\valpha_1|} |\vx_2|^{\gamma_2 - |\valpha_2|}, 
                \forall \vx_1,\vx_2\in V\setminus\{\bm{0}\}, \vz\in V_Z,
        \end{equation}
        with any derivative orders $\valpha_1,\valpha_2,\vbeta \geqslant 0$. 

        Assuming $ \min_i \gamma_i\geqslant -d+1$, the partially integrated function, 
        \[
        F(\vy, \vz) = \int_{V} \ud\vx f(\vx, \vx-\vy, \vz),
        \]      
        is smooth everywhere in $V \times V_Z$ except at $\vy = \bm{0}$ with order $\max(\gamma_1, \gamma_2)$.
\end{lem}
\begin{proof}
        See \cref{app:nonsmooth_integral}. 
\end{proof}
To convert to the condition in \cref{lem:nonsmooth_integral}, we reformulate the integrand for the 4h2p linear term as
\[
        f(\vq_1, \vq_2; \vk_i, \vk_j, \vk_a) = \sum_{kl} F^{kl}_1(\vq_1, \vk_i, \vk_j, \vk_a) F^{kl}_2(\vq_2, \vk_i, \vk_j, \vk_a),
\]
which satisfies \cref{eqn:singular_behavior} with $\vx_1 = \vq_1$, $\vx_2 = \vq_2$, $\vz = (\vk_i,\vk_j,\vk_a)$, $\gamma_1 = -2,$ and $\gamma_2 = 0$.
\cref{lem:nonsmooth_integral} shows that the partially integrated function,
\[
        \dfrac{1}{|\Omega^*|}\int_{\Omega^*}\ud\vq_1f(\vq_1, \vq_1 - \vy; \vk_i, \vk_j, \vk_a),
\]
is periodic and smooth everywhere with respect to $\vy, \vk_i, \vk_j, \vk_a\in\Omega^*$ except at $\vy = \bm{0}$ with order $0$.
Since the 4h2p linear term equals this function with $\vy = \vk_i - \vk_a$, we can follow \cref{def:fractional2} to show that 
it is periodic and smooth everywhere except at $\vk_a = \vk_i$ with order $0$. 
This proves the 4h2p linear term to be in $\mathbb{T}(\Omega^*)$.

All the quadratic terms are also of the same integral form where the integration is over two intermediate momentum vectors in $\Omega^*$ and 
the integrand is products of an ERI and two amplitude entries (see \cref{eqn:amplitude_map_expansion} for an example). 
Smoothness property analysis for these quadratic terms can be decomposed into two subproblems that can be addressed by the earlier analysis for linear terms.  
We use the 4h2p quadratic term in \cref{eqn:amplitude_map_expansion} to demonstrate the analysis. 

In the 4h2p quadratic term, the momentum vectors $\vk_l = \vk_i+\vk_j-\vk_k$ and $\vk_d = \vk_i + \vk_j - \vk_c$, and the integrand is a function of $\vk_k$ and $\vk_c$
for any fixed $\vk_i, \vk_j, \vk_a$. 
We first consider the partial integration over $\vk_c$ as
\[
H_{ijkl}(\vk_i, \vk_j, \vk_k) = \int_{\Omega^*}\ud\vk_c \sum_{cd} \braket{k\vk_k, l\vk_l | c\vk_c, d\vk_d} t_{ijcd}(\vk_i, \vk_j, \vk_c). 
\]
The integrand here as a function of $\vk_c$ is nonsmooth at $\vk_c = \vk_k$ and $\vk_c = \vk_i$ both with order $0$, due to 
the ERI and the amplitude. 
It is thus of the same form as the linear term with two nonsmooth points studied earlier, and we can use the same analysis to show that 
this intermediate function $H_{ijkl}(\vk_i, \vk_j, \vk_k)$ is smooth everywhere except at $\vk_i = \vk_k$ with order $0$.     
The overall quadratic term can then be written as 
\[
\int_{\Omega^*}\ud\vk_k \sum_{kl} H_{ijkl}(\vk_i, \vk_j, \vk_k)t_{klab}(\vk_k, \vk_l, \vk_a).
\]
The integrand here is of similar form to the 4h2p linear term, simply with the ERI term $\braket{k\vk_k, l\vk_l | i\vk_i, j\vk_j}$ replaced 
by $H_{ijkl}$ which has the same nonsmooth behavior at $\vk_k = \vk_i$ with order $0$. 
Using the same analysis for linear terms based on \cref{lem:nonsmooth_integral} can then show that this term as a function of 
$\vk_i, \vk_j, \vk_a$ is smooth everywhere except at $\vk_i = \vk_a$ with order $0$. 
This proves the 4h2p quadratic term to be in $\mathbb{T}(\Omega^*)$. 
All the other quadratic terms can be similarly analyzed using the analysis for the three types of linear terms above. 

With all the constant, linear, and quadratic terms in $\left[\mathcal{F}_\text{TDL}(t)\right]_{ijab}$ shown to be in $\mathbb{T}(\Omega^*)$, 
we finish the proof of \cref{lem:thm1_nonsmooth}.

\subsection{Proof of \cref{lem:thm1_energy}: Error in energy calculation using exact amplitude}
Consider an arbitrary amplitude $t \in \mathbb{T}(\Omega^*)^{n_\text{occ}\times n_\text{occ}\times n_\text{vir}\times n_\text{vir}}$. 
By expanding the antisymmetrized ERI, the finite-size error in the energy calculation using $t$ can be decomposed into the errors in 
the direct and the exchange term calculations as
\begin{align*}
        \mathcal{G}_\text{TDL}(t) - \mathcal{G}_{N_\vk}(\mathcal{M}_\mathcal{K} t)
        & = 
        \dfrac{2}{|\Omega^*|^3}
        \mathcal{E}_{\Omega^*\times\Omega^*\times\Omega^*}
        \left(
        \sum_{ijab}
        \braket{i\vk_i, j\vk_j | a\vk_a, b\vk_b} t_{ijab}(\vk_i, \vk_j, \vk_a), 
        \mathcal{K}\times\mathcal{K}\times\mathcal{K}
        \right)
        \\
        &\quad  
        - \dfrac{1}{|\Omega^*|^3}
        \mathcal{E}_{\Omega^*\times\Omega^*\times\Omega^*}
        \left(
        \sum_{ijab}
        \braket{i\vk_i, j\vk_j | b\vk_b, a\vk_a} t_{ijab}(\vk_i, \vk_j, \vk_a),
        \mathcal{K}\times\mathcal{K}\times\mathcal{K}
        \right).
\end{align*}

For each set of band indices $i,j,a,b$, we denote the integrand for the direct term calculation, i.e., the first term above, with the change of variable 
$\vk_a \rightarrow \vk_i + \vq$ as 
\[
F_\text{d}^{ijab}(\vk_i, \vk_j, \vq) = \braket{i\vk_i, j\vk_j | a(\vk_i + \vq), b(\vk_j - \vq)} t_{ijab}(\vk_i, \vk_j, \vk_i + \vq).
\]
The momentum transfers of the ERI and the amplitude entry are both equal to $\vq$. 
Expanding the ERI near $\vq = \bm{0}$ and using the assumption $t_{ijab}\in \mathbb{T}(\Omega^*)$, we can show that $F_\text{d}^{ijab}$ is periodic 
and smooth everywhere with respect to $\vk_i,\vk_j, \vq\in \Omega^*$ except at $\vq = \bm{0}$ with order $0$. 
Since $\vk_i$ and $\vk_a$ are sampled on the same mesh $\mathcal{K}$, the induced mesh for $\vq = \vk_a - \vk_i$ (map each $\vq$ to $\vq + \vG \in \Omega^*$ with some $\vG \in\mathbb{L}^*$ using the integrand 
periodicity with respect to $\Omega^*$) is of the same size as $\mathcal{K}$ and always contains $\vq = \bm{0}$. 
Denote this induced mesh as $\mathcal{K}_\vq$.
The quadrature error in the direct term calculation with each set of $i,j,a,b$ can then be formulated and split as
\begin{align}
    & \mathcal{E}_{\Omega^*\times\Omega^*\times\Omega^*}
    \left(
        F_\text{d}^{ijab}(\vk_i, \vk_j, \vq), \mathcal{K}\times\mathcal{K}\times\mathcal{K}_\vq
    \right)
    \nonumber \\
    = &  
    \mathcal{E}_{\Omega^*\times\Omega^*}
    \left(
        \int_{\Omega^*}\ud\vq
        F_\text{d}^{ijab}(\vk_i, \vk_j, \vq), \mathcal{K}\times\mathcal{K}
    \right)
    + 
    \dfrac{|\Omega^*|^2}{N_\vk^2}
    \sum_{\vk_i, \vk_j \in \mathcal{K}}
    \mathcal{E}_{\Omega^*}
    \left(
        F_\text{d}^{ijab}(\vk_i, \vk_j, \vq), \mathcal{K}_\vq
    \right).
    \label{eqn:error_mp2_split}
\end{align}
Fixing $\vk_i, \vk_j\in\mathcal{K}$, $F_\text{d}^{ijab}(\vk_i, \vk_j, \vq)$ as a function of $\vq$ is periodic and smooth everywhere in $\Omega^*$ except 
at $\vq = \bm{0}$ with order $0$ from the analysis above.
\cref{lem:quaderror1} provides quadrature error estimates for such periodic functions with a single point of algebraic singularity, 
and specifically in this case we have
\[
    \left|
    \mathcal{E}_{\Omega^*}
    \left(
        F_\text{d}^{ijab}(\vk_i, \vk_j, \vq), \mathcal{K}_\vq
    \right)
    \right|
    \leqslant C N_\vk^{-1},
    \quad
    \forall \vk_i,\vk_j \in \mathcal{K},
\]
where constant $C$ is independent of $\vk_i, \vk_j$ by using the algebraic singularity characterization of $F_\text{d}^{ijab}$ at $\vq=\bm{0}$  
and the prefactor estimate in \cref{lem:quaderror1} (see \cref{rem:quaderror1}).

Since $F_\text{d}^{ijab}$ is periodic and smooth with respect to $\vk_i, \vk_j\in\Omega^*$, $\int_{\Omega^*}\ud\vq F_\text{d}^{ijab}(\vk_i, \vk_j, \vq)$ is also periodic and 
smooth with respect to $\vk_i,\vk_j\in\Omega^*$ using the Leibniz integral rule.
According to \cref{lem:quaderror0}, the quadrature error for this partially integrated function (the first term in \cref{eqn:error_mp2_split})
decays super-algebraically as
\[
    \left|
    \mathcal{E}_{\Omega^*\times\Omega^*}
    \left(
        \int_{\Omega^*}\ud\vq
        F_\text{d}^{ijab}(\vk_i, \vk_j, \vq), \mathcal{K}\times\mathcal{K}
    \right)
    \right|
    \leqslant C_l N_{\vk}^{-l},\quad \forall l > 0.
\]
Plugging the above two estimates into \cref{eqn:error_mp2_split} proves that the quadrature error in the direct term calculation scales as 
\[
        \left|
                \mathcal{E}_{\Omega^*\times\Omega^*\times\Omega^*}
                \left(
                \sum_{ijab}
                \braket{i\vk_i, j\vk_j | a\vk_a, b\vk_b} t_{ijab}(\vk_i, \vk_j, \vk_a), 
                \mathcal{K}\times\mathcal{K}\times\mathcal{K}
                \right)
        \right|
    \leqslant C N_\vk^{-1}.
\]

Similar analysis can be applied to the exchange term where we formulate the integrand using two changes of variables $\vk_j \rightarrow \vk_a + \vq_1$ 
and $\vk_i \rightarrow \vk_a - \vq_2$ as 
\[
F_\text{x}^{ijab}(\vk_a, \vq_1, \vq_2) =         
        \braket{i(\vk_a - \vq_2), j(\vk_a + \vq_1) | b(\vk_a - \vq_2 + \vq_1), a\vk_a} t_{ijab}(\vk_a-\vq_2, \vk_a+\vq_1, \vk_a),
\]
where the momentum transfers of the ERI and the amplitude are $\vq_1$ and $\vq_2$, respectively. 
The ERI and the amplitude are smooth everywhere with respect to $\vk_a,\vq_1,\vq_2\in\Omega^*$ except at 
$\vq_1 = \bm{0}$ and $\vq_2 = \bm{0}$, respectively, both with order $0$. 
Similar to the discussion for the direct term, the exchange term calculation is equivalent to the trapezoidal rule over $F_\text{x}^{ijab}$ using  
uniform mesh $\mathcal{K}\times\mathcal{K}_\vq\times\mathcal{K}_\vq$. 
The associated quadrature error with each set of $(i,j,a,b)$ can be split as 
\begin{align}
        & \mathcal{E}_{\Omega^*\times\Omega^*\times\Omega^*}
          \left(
              F_\text{x}^{ijab}(\vk_a, \vq_1, \vq_2), \mathcal{K}\times\mathcal{K}_\vq\times\mathcal{K}_\vq
          \right)
          \nonumber \\
       = & 
        \mathcal{E}_{\Omega^*}
        \left(
            \int_{\Omega^*\times \Omega^*}\ud\vq_1\ud\vq_2
            F_\text{x}^{ijab}(\vk_a, \vq_1, \vq_2), \mathcal{K}
        \right)
        + 
        \dfrac{|\Omega^*|}{N_\vk}
        \sum_{\vk_a \in \mathcal{K}}
        \mathcal{E}_{\Omega^*\times \Omega^*}
        \left(
            F_\text{x}^{ijab}(\vk_a, \vq_1, \vq_2), \mathcal{K}_\vq\times \mathcal{K}_\vq
        \right).
        \label{eqn:error_mp2_exchange_split}
\end{align}
The first term decays super-algebraically since $\int_{\Omega^*\times \Omega^*}\ud\vq_1\ud\vq_2F_\text{x}^{ijab}(\vk_a, \vq_1, \vq_2)$ is smooth 
and periodic with respect to $\vk_a \in \Omega^*$ using the Leibniz integral rule. 
For the second term with each fixed $\vk_a \in \mathcal{K}$, $F_\text{x}^{ijab}$ can be viewed as a product of two periodic functions, 
$f_1(\vk_a, \vq_1, \vq_2)f_2(\vk_a, \vq_1, \vq_2)$, where $f_1$ is smooth everywhere except at $\vq_1 = \bm{0}$ with order $0$
and $f_2$ is smooth everywhere except at $\vq_2 = \bm{0}$ with order $0$. 
\cref{lem:quaderror3} provides quadrature error estimates for periodic functions in such a product form, and specifically in this case we have 
\[
        \left|
        \mathcal{E}_{\Omega^*\times \Omega^*}
        \left(
            F_\text{x}^{ijab}(\vk_a, \vq_1, \vq_2), \mathcal{K}_\vq\times \mathcal{K}_\vq
        \right)
        \right|
        \leqslant C N_\vk^{-1},
        \quad 
        \forall \vk_a \in \mathcal{K},
\]
where constant $C$ can be proved independent of $\vk_a$ using the algebraic singularity characterization of $f_1$ and $f_2$  
and the prefactor estimate in \cref{lem:quaderror3} (see \cref{rem:quaderror3}).

Plugging these two estimates into \cref{eqn:error_mp2_exchange_split} proves that the quadrature error in the exchange term calculation scales as 
\[
        \left|
        \mathcal{E}_{\Omega^*\times\Omega^*\times\Omega^*}
        \left(
        \sum_{ijab}
        \braket{i\vk_i, j\vk_j | b\vk_b, a\vk_a} t_{ijab}(\vk_i, \vk_j, \vk_a),
        \mathcal{K}\times\mathcal{K}\times\mathcal{K}
        \right)
        \right|
        \leqslant C N_\vk^{-1}.
\]

Combining the above estimates for the direct and exchange terms together, we have 
\[
\left|
        \mathcal{G}_\text{TDL}(t) - \mathcal{G}_{N_\vk}(\mathcal{M}_\mathcal{K} t)
\right|
= \Or(N_\vk^{-1}), 
\quad
\forall 
t \in \mathbb{T}(\Omega^*)^{n_\text{occ}\times n_\text{occ}\times n_\text{vir}\times n_\text{vir}},
\]
which covers the case of the exact CCD$(n)$ amplitudes with any $n > 0$.

\subsection{Proof of \cref{lem:thm1_amplitude_exact}: Amplitude error in a single iteration}
Consider the error in the amplitude calculation using an arbitrary amplitude 
$t \in \mathbb{T}(\Omega^*)^{n_\text{occ}\times n_\text{occ}\times n_\text{vir}\times n_\text{vir}}$.
Fixing a set of $(i,j,a,b)$ and $\vk_i,\vk_j,\vk_a\in\mathcal{K}$, the corresponding error entry can be detailed using the amplitude mapping definitions 
in \cref{eqn:amplitude_ccd} and \cref{eqn:amplitude_ccd_tdl} as 
\begin{align}
        & 
        \left[
                \mathcal{M}_{\mathcal{K}}\mathcal{F}_\text{TDL}(t) - \mathcal{F}_{N_\vk}(\mathcal{M}_{\mathcal{K}}t) 
        \right]_{ijab, \vk_i\vk_j\vk_a}
        \nonumber \\
        = & 
        \dfrac{1}{\varepsilon_{i\vk_i, j\vk_j}^{a\vk_a, b\vk_b}}
        \dfrac{1}{|\Omega^*|}
        \mathcal{E}_{\Omega^*}
        \left(
                \sum_{kl} \braket{k\vk_k, l\vk_l | i\vk_i, j\vk_j} t_{klab}(\vk_k, \vk_l, \vk_a), 
                \mathcal{K}
        \right)
        \nonumber \\
        & + 
                \dfrac{1}{\varepsilon_{i\vk_i, j\vk_j}^{a\vk_a, b\vk_b}}
                \dfrac{1}{|\Omega^*|^2}\mathcal{E}_{\Omega^*\times\Omega^*}\left(\sum_{klcd}
                \braket{k\vk_k, l\vk_l | c\vk_c, d\vk_d} t_{ijcd}(\vk_i, \vk_j, \vk_c)t_{klab}(\vk_k, \vk_l, \vk_a), \mathcal{K}\times\mathcal{K}
                \right)
        \nonumber\\
        & + \cdots,
        \label{eqn:error_amp_expansion}
\end{align}
where the constant terms cancel each other and the listed two quadrature errors are the errors in the 4h2p linear and 4h2p quadratic term calculations. 
The neglected terms are the errors in remaining linear and quadratic terms calculations which can all be similarly formulated as quadrature errors 
of trapezoidal rules. 
The problem is thus reduced to the error estimate for trapezoidal rules applied to integrands defined by different amplitude terms 
in the amplitude equation.

As shown in \cref{lemproof:thm1_nonsmooth}, the linear terms can be categorized into three classes listed in \cref{tab:class_linear_term} where the 
integrands respectively have zero, one, and two nonsmooth points. 
For terms in each class, their quadrature errors can be estimated similarly and we below demonstrate the error estimate for one example in each class. 

For linear terms with zero nonsmooth point, we consider the 3h3p term $\braket{AK|IC}t_{KJ}^{CB}$
detailed as, 
\[
        \int_{\Omega^*}\ud\vk_k \sum_{kc} \braket{a\vk_a, k\vk_k | i\vk_i, c\vk_c} t_{kjcb}(\vk_k, \vk_j, \vk_c).
\]
The ERIs and the amplitudes have momentum transfers $\vk_i-\vk_a =-\vq$ and $\vk_c-\vk_k = \vq$, respectively, which are independent
of $\vk_k$.
Therefore, for any $\vk_i, \vk_j, \vk_a\in\Omega^*$, the integrand is smooth and periodic with respect to $\vk_k$
and thus has the quadatrure error decay super-algebraically according to \cref{lem:quaderror0}, i.e.,
\begin{equation}
        \left|
                \mathcal{E}_{\Omega^*}
                \left(
                        \sum_{kc} \braket{a\vk_a, k\vk_k | i\vk_i, c\vk_c} t_{kjcb}(\vk_k, \vk_j, \vk_c),
                        \mathcal{K}
                \right)
        \right|
        \leqslant C_l N_\vk^{-l/3}, \quad \forall l > 0.
\end{equation}
where constant $C_l$ can be shown independent of $\vk_i, \vk_j, \vk_a\in\Omega^*$ using the prefactor estimate 
in \cref{lem:quaderror0} and the uniform boundedness of integrand derivatives over $\vk_k$ for all $\vk_i,\vk_j,\vk_a$ (see \cref{rem:quaderror0}).

For linear terms with one nonsmooth point, we consider the 3h3p term $\braket{AK|CI}t_{KJ}^{CB}$ detailed as
\[
        \int_{\Omega^*}\ud\vk_k \sum_{kc} \braket{a\vk_a, k\vk_k | c\vk_c, i\vk_i} t_{kjcb}(\vk_k, \vk_j, \vk_c),
\]
where the ERIs and the amplitudes have momentum transfers $\vk_k - \vk_i$ and $\vk_c - \vk_k = \vq$, respectively. 
For any $\vk_i, \vk_j, \vk_a\in\Omega^*$, these ERIs are smooth everywhere in $\Omega^*$ except at $\vk_k = \vk_i$
with order $0$, $-1$, or $-2$ depending on the relation between $(k,c)$ and $(i,a)$. 
It can then be verified that the overall integrand is smooth everywhere except at $\vk_k = \vk_i$ with order $-2$ due to the product term  
with $k=i,c=a$. 
\cref{lem:quaderror1} provides the quadrature error estimate for such a periodic function that has algebraic singularity at one point, 
and specifically in this case we have
\begin{equation}
         \left|
                \mathcal{E}_{\Omega^*}
                \left(
                        \sum_{kc} \braket{a\vk_a, k\vk_k | c\vk_c, i\vk_i} t_{kjcb}(\vk_k, \vk_j, \vk_c),
                        \mathcal{K}
                \right)
        \right|
        \leqslant C N_\vk^{-1/3},
\end{equation}
where constant $C$ can be shown independent of $\vk_i, \vk_j, \vk_a$ using the prefactor estimate in \cref{lem:quaderror1} and the 
algebraic singularity characterization of the ERIs and the amplitudes (see \cref{rem:quaderror1}).

For linear terms with two nonsmooth points, we consider the 4h2p term $\braket{KL|IJ}t_{KL}^{AB}$ detailed in \cref{eqn:error_amp_expansion}. 
First denote the integrand with each set of $(k,l)$ as
\begin{equation}\label{eqn:integrand_amplitude}
        F^{kl}(\vk_k) = 
        \braket{k\vk_k, l\vk_l| i\vk_i,j \vk_j}
        t_{klab}(\vk_k, \vk_l, \vk_a).
\end{equation}
The ERI in $F^{kl}(\vk_k)$ is smooth everywhere except at $\vk_k = \vk_i$ with order $0$, $-1$, or $-2$ 
depending on the relation between $(k,l)$ and $(i,j)$. 
The amplitude in $F^{kl}(\vk_k)$ is smooth everywhere except at $\vk_k = \vk_a$ with order $0$. 
Applying the change of variable $\vk_k \rightarrow \vk_i - \vq_1$, $F^{kl}(\vq_1)$ can be formulated as the product of two periodic functions, $f_1(\vq_1)f_2(\vq_1)$, 
where $f_1(\vq_1)$ is nonsmooth at $\vq_1 = \bm{0}$ with order $0$, $-1$, or $-2$ and $f_2(\vq_1)$ is nonsmooth 
at $\vq_1 = \vk_i - \vk_a$ with order $0$. 
Quadrature error of periodic functions in such a product form is estimated by \cref{lem:quaderror2} when $\vk_i\neq\vk_a \in \mathcal{K}$ and 
by \cref{lem:quaderror1} when $\vk_i = \vk_a \in \mathcal{K}$ as 
\begin{equation*}
        \left|
        \mathcal{E}_{\Omega^*}\left(
        F^{kl}(\vq_1), \mathcal{K}_\vq
        \right) 
        \right|
        \leqslant 
        C
        \begin{cases}
                N_\vk^{-1} & k\neq i, l\neq j \\
                N_\vk^{-\frac23} & k=i, l\neq j \text{ or } k\neq i, l=j\\
                N_\vk^{-\frac13} & k=i, l=j 
        \end{cases}
\end{equation*}
where constant $C$ can be shown independent of $\vk_i,\vk_j,\vk_a$ using the prefactor estimates in the two lemmas 
and the algebraic singularity characterization of the ERIs and the amplitudes (see \cref{rem:quaderror1} and \cref{rem:quaderror2}).
Summing over the above estimates for each set of $(k,l)$, the quadrature error in the 4h2p linear term calculation can be estimated as 
\begin{align*}
        \left|
        \mathcal{E}_{\Omega^*}
        \left(
                \sum_{kl} \braket{k\vk_k, l\vk_l | i\vk_i, j\vk_j} t_{klab}(\vk_k, \vk_l, \vk_a), 
                \mathcal{K}
        \right)
        \right|
        & \leqslant 
        C N_\vk^{-\frac13}.
\end{align*}

Similar to the linear term case, all the quadratic terms and their quadrature error estimates can be categorized into four classes according to the nonsmoothness with respect 
to the two intermediate momentum vectors, as listed in \cref{tab:class_quadratic_term}. 

\begin{table}[htbp]
        \caption{Classification of quadratic amplitude terms by the nonsmooth points with respect to the intermediate momentum vectors.
        We use $\vk_1, \vk_2$ to denote the two integration variables after proper change of variables over the intermediate momentum vectors and 
        $\vx_1, \vx_2, \vx_3$ to denote generic points in $\Omega^*$. 
        The 3h3p terms permuted by $\mathcal{P}$ are of the same class as the unpermuted ones and thus not listed. 
        The intermediate block underlying each product indicates the origin of this product in \cref{eqn:amplitude_ccd_tdl}. 
        }\label{tab:class_quadratic_term}
        \begin{tabular}{cc}
                \toprule
                Singular Points 
                & Quadratic terms  \\
                \midrule
                None & $\underbrace{\braket{LK|DC}t_{IL}^{AD}}_{\chi_{IC}^{AK}}t_{KJ}^{CB}$\\
                $\vk_1 = \vx_1$  &$\underbrace{\braket{LK|CD}t_{IL}^{AD}}_{\chi_{IC}^{AK}}t_{KJ}^{CB}$,  $\underbrace{\braket{LK|DC}t_{IL}^{DA}}_{\chi_{IC}^{AK}}t_{KJ}^{CB}$, $\underbrace{\braket{LK|DC}t_{IL}^{AD}}_{\chi_{IC}^{AK}}t_{KJ}^{BC}$
                                $\underbrace{\braket{KL|CD}t_{KL}^{AD}}_{\kappa_C^A}t_{IJ}^{CB}$,     $\underbrace{\braket{KL|CD} t_{IL}^{CD}}_{\kappa_I^K}t_{KJ}^{AB}$\\
                $\vk_1 = \vx_1$, $\vk_2 = \vx_2$& $\underbrace{\braket{LK|CD}t_{IL}^{DA}}_{\chi_{CI}^{AK}}t_{KJ}^{CB}$, $\underbrace{\braket{LK|CD}t_{IL}^{AD}}_{\chi_{IC}^{AK}}t_{KJ}^{BC}$, $\underbrace{\braket{KL|DC}t_{KL}^{AD}}_{\kappa_C^A}t_{IJ}^{CB}$,     $\underbrace{\braket{KL|DC} t_{IL}^{CD}}_{\kappa_I^K}t_{KJ}^{AB}$\\ 
                $\vk_1 = \vx_1$, $\vk_2 = \vx_2$, $\vk_1 \pm \vk_2 = \vx_3$ & $\underbrace{\braket{KL|CD}t_{IJ}^{CD}}_{\chi_{IJ}^{AB}}t_{KL}^{AB}$, $\underbrace{\braket{LK|DC} t_{IL}^{DA}}_{\chi_{IC}^{AK}}t_{KJ}^{BC}$,            $\underbrace{\braket{LK|CD}t_{JL}^{DA}}_{\chi_{CJ}^{AK}}t_{KI}^{BC}$\\
                \bottomrule
        \end{tabular}
\end{table}

For quadratic terms of the first and second class, their quadrature errors can be estimated using \cref{lem:quaderror0} and \cref{lem:quaderror1} in a similar way 
as for the linear terms above. 
In the following, we demonstrate the quadrature error estimate for the third and forth classes of quadratic terms. 

For the third class, we consider the 3h3p quadratic term $\braket{LK|CD}t_{IL}^{AD}t_{KJ}^{BC}$ and denote the integrand 
for each set of $(k,l,c,d)$ as 
\[
       F_\text{3h3p}^{klcd}(\vk_k, \vk_l) = \braket{l\vk_l, k\vk_k | c\vk_c, d\vk_d} t_{ilad}(\vk_i, \vk_l, \vk_a) t_{kjbc}(\vk_k, \vk_j, \vk_b),
\]
where $\vk_c = \vk_k + \vk_j - \vk_b$ and $\vk_d = \vk_l  + \vk_i  - \vk_a$. 
The ERI and the two amplitudes have momentum transfers as $\vk_c - \vk_l = \vk_k - \vk_l + \vk_j-\vk_b$, $\vk_a - \vk_i$, and $\vk_b - \vk_k$, respectively.
To single out the nonsmoothness with respect to $\vk_k$ and $\vk_l$, we introduce two changes of variables $\vk_k \rightarrow \vk_b - \vq_2$ and 
$\vk_l \rightarrow \vk_k + \vk_j - \vk_b + \vq_1$ and this term calculation can be formulated using the integrand periodicity as 
\[
        \dfrac{1}{N_\vk^2}\sum_{\vq_1,\vq_2\in\mathcal{K}_\vq}\sum_{klcd}F_\text{3h3p}^{klcd}(\vq_1, \vq_2) \longrightarrow \dfrac{1}{|\Omega^*|^2}\int_{\Omega^*\times\Omega^*}\ud\vq_1\ud\vq_2 \sum_{klcd}F_\text{3h3p}^{klcd}(\vq_1, \vq_2),
\]
where the integrand is smooth everywhere except at $\vq_1 = \bm{0}$ and $\vq_2 = \bm{0}$.
This explains the classification of this term as the third class listed in \cref{tab:class_quadratic_term}.
Note that the first amplitude in $F_\text{3h3p}^{klcd}$ is smooth with respect to $\vq_1, \vq_2$
and thus $F_\text{3h3p}^{klcd}$ can be written in a product form $f_1(\vq_1, \vq_2)f_2(\vq_1, \vq_2)$ where $f_s(\vq_1, \vq_2)$ with $s=1,2$ is smooth everywhere except at $\vq_s = \bm{0}$
with order $0$. 
\cref{lem:quaderror3} provides the quadrature error estimate for bivariate functions in such a product form, and specifically in this case we have
\[
        \left|
                \mathcal{E}_{\Omega^*\times\Omega^*}\left(
                        \sum_{klcd}F_\text{3h3p}^{klcd}(\vq_1, \vq_2), \mathcal{K}_\vq\times\mathcal{K}_\vq
                \right)
        \right|         
        \leqslant C N_\vk^{-1},
\]
where constant $C$ can be shown independent of $\vk_i,\vk_j,\vk_a$ using the prefactor estimate in \cref{lem:quaderror3} and the 
algebraic singularity characterization of ERIs and amplitudes (see \cref{rem:quaderror3}).

For the forth class, we consider the 4h2p quadratic term $\braket{KL|CD}t_{IJ}^{CD}t_{KL}^{AB}$ and denote the integrand with each 
set of $(k,l,c,d)$ with the change of variable $\vk_k \rightarrow \vk_c + \vq_1$ and 
$\vk_c \rightarrow \vk_i + \vq_2$ as 
\[
        F_\text{4h2p}^{klcd}(\vq_1, \vq_2) = 
        \braket{k\vk_k, l\vk_l | c\vk_c, d\vk_d} t_{ijcd}(\vk_i, \vk_j, \vk_c)t_{klab}(\vk_k, \vk_l, \vk_a),
\]
where the ERI is smooth everywhere except at $\vq_1 = \bm{0}$ with order $0$,
the two amplitudes are smooth everywhere except at $\vq_2 = \bm{0}$ and $\vk_a - \vk_i = \vq_2 - \vq_1$, respectively, with order $0$. 
This explains the classification of this term as the forth class listed in \cref{tab:class_quadratic_term}. 
\cref{lem:quaderror4} provides the quadrature error estimate for bivariate functions in such a product form, and specifically in this case we have
\[
        \left|
                \mathcal{E}_{\Omega^*\times\Omega^*}\left(
                        \sum_{klcd}F_\text{4h2p}^{klcd}(\vq_1, \vq_2), \mathcal{K}_\vq\times\mathcal{K}_\vq
                \right)
        \right|         
        \leqslant C N_\vk^{-1},
\]
where constant $C$ can be shown independent of $\vk_i,\vk_j,\vk_a$ using the prefactor estimate in \cref{lem:quaderror4} and the 
algebraic singularity characterization of ERIs and amplitudes (see \cref{rem:quaderror4}).

Collecting all the above quadrature error estimates for linear and quadratic terms (see \cref{tab:amplitude_error} for a summary), 
we obtain the final error estimate in amplitude calculation
\[
        \left|
        \left[
                \mathcal{M}_{\mathcal{K}}\mathcal{F}_\text{TDL}(t) - \mathcal{F}_{N_\vk}(\mathcal{M}_{\mathcal{K}}t) 
        \right]_{ijab, \vk_i\vk_j\vk_a}
        \right|
        \leqslant 
        C N_\vk^{-\frac13}, 
        \qquad
        \forall i,j,a,b, \forall \vk_i,\vk_j,\vk_a\in\mathcal{K}.
\]
It is worth pointing out that the dominant error comes from the following six linear amplitude terms, 
\[
\sum_{KL}\braket{KL|IJ}t_{KL}^{AB},\quad \sum_{CD}\braket{AB|CD}t_{IJ}^{CD},\quad \mathcal{P}\sum_{KC}\braket{AK|CJ}t_{IK}^{CB},\quad \mathcal{P}\sum_{KC}\braket{AK|CI}t_{KJ}^{BC},        
\]
where the involved ERIs can have matching band indices and thus are nonsmooth at zero momentum transfer points with order $-2$.

\subsection{Proof of \cref{lem:thm1_amplitude_prev}: Error accumulation in the CCD iteration}
Fixing a set of $i,j,a,b$ and $\vk_i,\vk_j,\vk_a$, we focus on one constant, one linear, and one quadratic terms in the entry 
$\left[\mathcal{F}_{N_\vk}(T)\right]_{ijab,\vk_i\vk_j\vk_a}$ detailed as follows
\begin{align}
        & 
        \left[\mathcal{F}_{N_\vk}(T)\right]_{ijab,\vk_i\vk_j\vk_a}
        = 
        \dfrac{1}{\varepsilon_{i\vk_i, j\vk_j}^{a\vk_a, b\vk_b}}\braket{a\vk_a, b\vk_b | i\vk_i, j\vk_j}
        \label{eqn:amplitude_map_expansion_finite}\\
        & 
        \qquad + 
        \dfrac{1}{\varepsilon_{i\vk_i, j\vk_j}^{a\vk_a, b\vk_b}}
        \dfrac{1}{N_\vk}\sum_{\vk_k \in \mathcal{K}}\sum_{kl} \braket{k\vk_k, l\vk_l | i\vk_i, j\vk_j} T_{klab}(\vk_k, \vk_l, \vk_a)
        \nonumber\\
        & 
        \qquad + 
        \dfrac{1}{\varepsilon_{i\vk_i, j\vk_j}^{a\vk_a, b\vk_b}}
        \dfrac{1}{N_\vk^2}\sum_{\vk_k\vk_c \in \mathcal{K}} \sum_{klcd} \braket{k\vk_k, l\vk_l | c\vk_c, d\vk_d} 
        T_{ijcd}(\vk_i, \vk_j, \vk_c)T_{klab}(\vk_k, \vk_l, \vk_a)
        \nonumber \\
        & \qquad + \cdots
        \nonumber,
\end{align}
where the linear and quadratic terms come from the 4h2p term $\chi_{IJ}^{KL}t_{KL}^{AB}$ in \cref{eqn:amplitude_ccd}
and the neglected terms above are the other linear and quadratic terms included in \cref{eqn:amplitude_ccd}.

In the subtraction $\mathcal{F}_{N_\vk}(T) - \mathcal{F}_{N_\vk}(S)$, the constant terms above in these two maps cancel each other.   
The subtraction between the two 4h2p linear terms can be formulated and bounded as 
\begin{align*}
        & 
        \left|
        \dfrac{1}{\varepsilon_{i\vk_i, j\vk_j}^{a\vk_a, b\vk_b}}
        \dfrac{1}{N_\vk}\sum_{\vk_k \in \mathcal{K}}\sum_{kl} \braket{k\vk_k, l\vk_l | i\vk_i, j\vk_j} 
        \left(
        T_{klab}(\vk_k, \vk_l, \vk_a)
        - 
        S_{klab}(\vk_k, \vk_l, \vk_a)
        \right)
        \right|
        \\ 
        \leqslant 
        & 
        C
        \dfrac{1}{N_\vk}\sum_{\vk_k \in \mathcal{K}}\sum_{kl} 
        \left|
        \braket{k\vk_k, l\vk_l | i\vk_i, j\vk_j} 
        \right|
        \| T - S \|_\infty
        \\
        \leqslant 
        & 
        C
        \| T - S \|_\infty
        \int_{\Omega^*}\ud \vk_k
        \sum_{kl}
        \left|
        \braket{k\vk_k, l\vk_l | i\vk_i, j\vk_j} 
        \right|
        \leqslant 
        C
        \| T - S \|_\infty.
\end{align*}
Similar estimate can be obtained for all the other linear terms in the amplitude map $\mathcal{F}_{N_\vk}$. 
The subtraction between the two 4h2p quadratic terms can be formulated and bounded as 
\begin{align*}
        & 
        \left|
        \dfrac{1}{\varepsilon_{i\vk_i, j\vk_j}^{a\vk_a, b\vk_b}}
        \dfrac{1}{N_\vk^2}\sum_{\vk_k\vk_c \in \mathcal{K}} \sum_{klcd} \braket{k\vk_k, l\vk_l | c\vk_c, d\vk_d} 
        \left(
        T_{ijcd}(\vk_i, \vk_j, \vk_c)T_{klab}(\vk_k, \vk_l, \vk_a)
        -
        S_{ijcd}(\vk_i, \vk_j, \vk_c)S_{klab}(\vk_k, \vk_l, \vk_a)
        \right)
        \right|
        \\
        \leqslant & 
        C
        \int_{\Omega^*\times\Omega^*}\ud\vk_k\ud\vk_c
        \sum_{klcd} |\braket{k\vk_k, l\vk_l | c\vk_c, d\vk_d}|
        \left(
                \|T\|_\infty + \|S\|_\infty
        \right)
        \|T - S\|_\infty
        \\
        \leqslant & 
        C
        \left(
                \|T\|_\infty + \|S\|_\infty
        \right)
        \|T - S\|_\infty,
\end{align*}
where the first inequality uses the estimate
\begin{align*}
        & 
        \left|
        T_{ijcd}(\vk_i, \vk_j, \vk_c)T_{klab}(\vk_k, \vk_l, \vk_a)
        -
        S_{ijcd}(\vk_i, \vk_j, \vk_c)S_{klab}(\vk_k, \vk_l, \vk_a)
        \right|
        \\
        \leqslant 
        & 
        \left|
        T_{ijcd}(\vk_i, \vk_j, \vk_c)
        \left(
                T_{klab}(\vk_k, \vk_l, \vk_a)
                - 
                S_{klab}(\vk_k, \vk_l, \vk_a)
        \right)
        \right|
        + 
        \left|
        \left(
                T_{ijcd}(\vk_i, \vk_j, \vk_c)
                -
                S_{ijcd}(\vk_i, \vk_j, \vk_c)
        \right)
                S_{klab}(\vk_k, \vk_l, \vk_a)
        \right|
        \\
        \leqslant 
        & 
        \|T\|_\infty \|T - S\|_\infty + 
        \|S\|_\infty \|T - S\|_\infty.
\end{align*}
Similar estimate can be obtained for all the other quadratic terms in the amplitude map $\mathcal{F}_{N_\vk}$. 
Collecting all these estimates together, we have 
\begin{align*}
        \|\mathcal{F}_{N_\vk}(T) - \mathcal{F}_{N_\vk}(S)\|_\infty
        & = 
        \max_{ijab,\vk_i\vk_j\vk_a\in\mathcal{K}}
        \left|
        \left[\mathcal{F}_{N_\vk}(T) - \mathcal{F}_{N_\vk}(S)\right]_{ijab,\vk_i\vk_j\vk_a}
        \right|
        \\
        & \leqslant
        C \left(1 + \|T\|_\infty + \|S\|_\infty\right)
        \|T - S\|_\infty.
\end{align*}

\section{Proof of Corollary \ref{thm:error_ccd_converge}}

As discussed in \cref{sec:discuss}, it is possible in general that the CCD amplitude equation may have multiple solutions or its fixed point iteration may diverge. 
In these cases, the finite size error in CCD energy calculation can be ill-defined and not connected to CCD($n$) we have analyzed. 
Here, we consider the ideal case where $T = \mathcal{F}_{N_\vk}(T)$ for any sufficiently 
large $N_\vk$ and $t = \mathcal{F}_\text{TDL}(t)$ both have unique solutions, denoted as $T_*^{N_\vk}$ and $t_*$, and the corresponding fixed point iterations converge in the sense of the 
$\|\cdot\|_\infty$-norm, i.e., 
\[
        \lim_{n\rightarrow \infty} \|T_n - T_*^{N_\vk}\|_\infty = 0 \qquad \text{and} \qquad 
        \lim_{n\rightarrow \infty} \|t_n - t_*\|_\infty = 0.
\] 
In general, a common sufficient condition that guarantees the convergence of a fixed point iteration is that the target mapping is contractive 
(to be specified later) in a domain that contains the solution point and the initial guess also lies in this domain.
Following this practice, we make four assumptions:
\begin{itemize}
        \item  $\mathcal{F}_\text{TDL}$ is a contraction map in a domain 
$\mathbb{B}_\text{TDL}\subset \mathbb{T}(\Omega^*)^{n_\text{occ}\times n_\text{occ} \times n_\text{vir} \times n_\text{vir}}$ that contains $t_*$ and the initial guess $\bm{0}$, i.e., 
\begin{align*}
        & \mathcal{F}_\text{TDL}(t) \in \mathbb{B}_\text{TDL},\quad \forall t \in \mathbb{B}_\text{TDL}, 
        \\
        & \|\mathcal{F}_\text{TDL}(t) - \mathcal{F}_\text{TDL}(s)\|_\infty \leqslant L \|t - s\|_\infty,\quad  \forall t,s \in \mathbb{B}_\text{TDL},
\end{align*}
with a constant $L < 1$.
This assumption guarantees that $\{t_n\}$ lies in $\mathbb{B}_\text{TDL}$ and converges to $t_*$. 

        \item 
$\mathcal{F}_{N_\vk}$ with sufficiently large $N_\vk$ is a contraction map in a domain 
$\mathbb{B}_{N_\vk}\subset \mathbb{C}^{n_\text{occ}\times n_\text{occ} \times n_\text{vir} \times n_\text{vir}\times N_\vk\times N_\vk \times N_\vk}$ that 
contains $T_*^{N_\vk}$ and the initial guess $\bm{0}$, i.e., 
\begin{align*}
        & \mathcal{F}_{N_\vk}(T)  \in \mathbb{B}_{N_\vk}, \quad \forall T \in \mathbb{B}_{N_\vk},
        \\
        & \|\mathcal{F}_{N_\vk}(T) - \mathcal{F}_{N_\vk}(S)\|_\infty \leqslant L \|T - S\|_\infty,\quad  \forall T,S \in \mathbb{B}_{N_\vk},
\end{align*}
with a constant $L < 1$. 
This assumption guarantees that $\{T_n^{N_\vk}\}$ lies in $\mathbb{B}_{N_\vk}$ and converges to $T_*^{N_\vk}$. 

        \item For sufficiently large $N_\vk$, the domains in the above two assumptions satisfy that 
        \begin{equation}\label{eqn:assumption_ccd_3}
                \mathcal{M}_\mathcal{K}\mathbb{B}_\text{TDL}
                := 
                \{
                \mathcal{M}_\mathcal{K}t: t\in \mathbb{B}_\text{TDL}         
                \}
                \subset 
                \mathbb{B}_{N_\vk}.
        \end{equation}
        \cref{thm:error_ccd} proves that for each fixed $n$ the finite-size amplitude $T_n^{N_\vk}$ converges to $t_n$ in the sense of
        \[
                \lim_{N_\vk\rightarrow\infty} \|\mathcal{M}_\mathcal{K}t_n - T_n^{N_\vk}\|_\infty  = 0,
        \]
        showing that $\{T_n^{N_\vk}\} \subset \mathbb{B}_{N_\vk}$ converges to $\{t_n\} \subset \mathbb{B}_\text{TDL}$
        with $\mathcal{K} \rightarrow \Omega^*$. 
        Intuitively, this argument suggests certain closeness between $\mathbb{B}_{N_\vk}$ and $\mathcal{M}_\mathcal{K}\mathbb{B}_\text{TDL}$
        which leads to the assumption here. 

        \item Note that the error estimate in \cref{lem:thm1_amplitude_exact} has prefactor $C$ dependent on the amplitude $t$. 
        For the whole set of iterates $\{t_n\}$, we make a stronger assumption that there exists a constant $C$ such that 
        \begin{equation}\label{eqn:assumption_ccd_4}
                \left\|
                        \mathcal{M}_{\mathcal{K}}\mathcal{F}_\text{TDL}(t_{n}) - 
                        \mathcal{F}_{N_\vk}(\mathcal{M}_{\mathcal{K}}t_{n})
                \right\|_\infty
                \leqslant C N_\vk^{-\frac13}, 
                \quad 
                \forall n > 0.
        \end{equation}
\end{itemize}

Under these assumptions, the finite-size error in the CCD energy calculation can be estimated as 
\begin{align}
        \left|
                E_{\text{CCD}}^\text{TDL} - E_{\text{CCD}}^{N_\vk}
        \right|
        & = 
        \left|
                \mathcal{G}_\text{TDL}(t_*) - \mathcal{G}_{N_\vk}(T_*^{N_\vk})
        \right|
        \nonumber\\
        & \leqslant 
        \left|
                \mathcal{G}_{N_\vk}(\mathcal{M}_{\mathcal{K}}t_*) - \mathcal{G}_{N_\vk}(T_*^{N_\vk})
        \right|
        +
        \left|
                \mathcal{G}_\text{TDL}(t_*) - \mathcal{G}_{N_\vk}(\mathcal{M}_{\mathcal{K}}t_*)
        \right|
        \nonumber\\
        & \leqslant 
        C \left\|\mathcal{M}_{\mathcal{K}}t_* - T_*^{N_\vk}\right\|_\infty
        +
        C N_\vk^{-1}
        \label{eqn:error_ccd_converge},
\end{align}
where the last inequality uses the boundedness of linear operator $\mathcal{G}_{N_\vk}$ and \cref{lem:thm1_energy}. 
To estimate the finite-size error in the converged amplitudes, $\left\|\mathcal{M}_{\mathcal{K}}t_* - T_*^{N_\vk}\right\|_\infty$, 
we consider the error splitting \cref{eqn:amplitude_splitting} for the amplitude calculation at the $n$-th fixed point iteration as
\begin{align*}
        \left\|
                \mathcal{M}_{\mathcal{K}}t_{n} - T_{n}^{N_\vk}
        \right\|_\infty
        & \leqslant 
        \left\|
                \mathcal{M}_{\mathcal{K}}\mathcal{F}_\text{TDL}(t_{n-1}) - 
                \mathcal{F}_{N_\vk}(\mathcal{M}_{\mathcal{K}}t_{n-1})
        \right\|_\infty
        +
        \left\|
                \mathcal{F}_{N_\vk}(T_{n-1}^{N_\vk}) - \mathcal{F}_{N_\vk}( \mathcal{M}_{\mathcal{K}}t_{n-1})
        \right\|_\infty
        \\
        & \leqslant 
        C
        N_\vk^{-\frac13}
        + 
        L  
        \left\|
                \mathcal{M}_\mathcal{K}t_{n-1} - T_{n-1}^{N_\vk}
        \right\|_\infty,
\end{align*}
where the last estimate uses the assumption in \cref{eqn:assumption_ccd_4} for the first term and the 
assumptions that $\mathcal{F}_{N_\vk}$ is a contraction map and  $\mathcal{M}_\mathcal{K}t_{n-1}\in \mathbb{B}_{N_\vk}$  for the second term.
Since the initial guesses in the finite and the TDL cases satisfy $\|\mathcal{M}_\mathcal{K}t_0 - T_0^{N_\vk}\|_\infty = 0$, 
we can recursively derive that 
\[
        \left\|
                \mathcal{M}_{\mathcal{K}}t_{n} - T_{n}^{N_\vk}
        \right\|_\infty
        \leqslant 
        C \dfrac{1 - L^n}{1 - L} N_\vk^{-\frac13},
\]
and thus 
\[
        \left\|
                \mathcal{M}_{\mathcal{K}}t_{*} - T_{*}^{N_\vk}
        \right\|_\infty
        = \lim_{n\rightarrow\infty}
        \left\|
                \mathcal{M}_{\mathcal{K}}t_{n} - T_{n}^{N_\vk}
        \right\|_\infty
        \leqslant 
        C N_\vk^{-\frac13}.
\]
Plugging this estimate into \cref{eqn:error_ccd_converge} then finishes the proof.

\section{Quadrature error estimate for periodic function with algebraic singularity}\label{app:quadrature_error}

This section consists of five lemmas that provide the quadrature error estimates for trapezoidal rules over integrands in five different classes 
as listed in \cref{tab:class_integral}. 
All the integrands are either smooth or built by univariate/multivariate functions that have algebraic singularity at one single point.
In addition to the asymptotic scaling of the quadrature errors, our finite-size error analysis also needs quantitative descriptions about the relation 
between prefactors in the estimate  and the smoothness properties of the integrand.

For a univariate function $f(\vx)$ that is smooth everywhere in $V$ except at $\vx = \vx_0$ with algebraic singularity of order $\gamma$, 
we define a constant  
\begin{align}
        \mathcal{H}_{V, \vx_0}^{l}(f) 
& = \min
\left\{
        C: 
        \left|
                \partial^\valpha_\vx f(\vx)
        \right|
        \leqslant 
        C 
        |\vx - \vx_0|^{\gamma - |\valpha|}, 
        \forall |\valpha|\leqslant l, 
        \forall \vx \in V\setminus\{\vx_0\}
\right\}
\nonumber\\
& = 
        \max_{|\valpha|\leqslant l} 
        \left\| 
                \left(\partial^\valpha_\vx f(\vx)\right) / |\vx - \vx_0|^{\gamma - |\valpha|}
        \right\|_{L^{\infty}(V)}.
        \label{eqn:fractional_univariate}
\end{align}
For a multivariate function $f(\vx, \vy)$ that is smooth everywhere in $V_X\times V_Y$ except at $\vx = \vx_0$ with algebraic singularity of order $\gamma$,
we define a constant 
\begin{align}
\mathcal{H}_{V_X\times V_Y, (\vx_0,\cdot)}^{l}(f) 
& = \min
\left\{
        C: 
        \left|
                \partial^\valpha_\vx \partial_\vy^\vbeta f(\vx, \vy)
        \right|
        \leqslant 
        C 
        |\vx - \vx_0|^{\gamma - |\valpha|}, 
        \forall |\valpha|, |\vbeta|\leqslant l, 
        \forall \vx \in V_X\setminus\{\vx_0\}, \vy\in V_Y
\right\}
\nonumber\\
& = 
        \max_{|\valpha|\leqslant l} 
        \left\| 
                \left(\partial^\valpha_\vx \partial_\vy^\vbeta f(\vx, \vy)\right)
                / |\vx - \vx_0|^{\gamma - |\valpha|}
        \right\|_{L^{\infty}(V\times V)},
        \label{eqn:fractional_multivariate}
\end{align}
where ``$\cdot$'' in the subscript ``$(\vx_0, \cdot)$'' is a placeholder to indicate the smooth variable.
Using these two quantities, we have following function estimates that will be extensively used in this section 
\begin{align}
        & 
        \left|
                \partial^\valpha_\vx f(\vx)
        \right|
        \leqslant  
        \mathcal{H}_{V, \vx_0}^{l}(f) |\vx - \vx_0|^{\gamma - |\valpha|}, 
        \qquad
        \forall l \geqslant |\valpha|,\ \forall \vx \in V\setminus\{\vx_0\},
        \label{eqn:fractional_univariate_ineq}
        \\
        & 
        \left|
                \partial^\valpha_\vx \partial^\vbeta_\vy f(\vx, \vy)
        \right|
        \leqslant  
        \mathcal{H}_{V_X\times V_Y, (\vx_0, \cdot)}^{l}(f) |\vx - \vx_0|^{\gamma - |\valpha|}, 
        \qquad
        \forall l \geqslant |\valpha|, |\vbeta|,\ \forall \vx \in V_X\setminus\{\vx_0\},\vy\in V_Y.
        \label{eqn:fractional_multivariate_ineq}
\end{align}

The following lemma is the standard result that the quadrature error of a trapezoidal rule applied to smooth and periodic functions decays super-algebraically. 
For completeness, we include a proof using the Poisson summation formula.
\begin{lem}\label{lem:quaderror0}
        Let $f(\vx)$ be smooth and periodic in $V = [-\frac12, \frac12]^d$. 
        The quadrature error of a trapezoidal rule using an $m^d$-sized uniform mesh $\mathcal{X}$ in $V$ decays super-algebraically as
        \[
                \left|\mathcal{E}_V(f, \mathcal{X})\right| \leqslant C_l m^{-l},\quad \forall l > 0.
        \]
\end{lem}
\begin{proof}
        Denote the Fourier transform of $f(\vx)$ as 
        \begin{equation*}
                f(\vx) = \frac{1}{|V|}\sum_{\vk \in 2\pi\mathbb{Z}^d}\hat{f}(\vk) e^{\I\vk\cdot\vx} \quad \text{with} \quad \hat{f}(\vk) = \int_{V}f(\vx)e^{-\I \vk\cdot\vx}\ud\vx.
        \end{equation*}
        Let $\mathcal{X}_0$ be the $m^d$-sized uniform mesh in $V$ that contains $\vx = \bm{0}$ and let $\mathcal{X} = \mathcal{X}_0 + \vx_0$.
        Denote the unit length $h = 1/m$. 
        Using these notations, we have
        \begin{align*}
                \mathcal{I}_V(f) & = \hat{f}(\bm{0}), \\
                \mathcal{Q}_V(f, \mathcal{X})
                & = \frac{|V|}{m^d}\sum_{\vx \in \mathcal{X}_0} f(\vx + \vx_0)
                = \sum_{\vk\in 2\pi\mathbb{Z}^d}\hat{f}(\vk) e^{\I\vk\cdot\vx_0}\frac{1}{m^d}\sum_{\vx\in\mathcal{X}_0}e^{\I\vk\cdot\vx}
                = \sum_{\vk\in 2\pi \mathbb{Z}^d}\hat{f}\left(\frac{\vk}{h}\right) e^{\I\frac{\vk}{h}\cdot\vx_0},
        \end{align*}
        and thus the quadrature error of the trapezoidal rule using $\mathcal{X}$ can be estimated as 
        \begin{equation}\label{eqn:quaderror0_fourier}
                \left|\mathcal{E}_V(f, \mathcal{X})\right|
                \leqslant 
                \sum_{\vk\in 2\pi \mathbb{Z}^d \setminus \{\bm{0}\}} \left|\hat{f}\left(\frac{\vk}{h}\right)\right|.
        \end{equation}
        Based on the periodicity and smoothness of $f(\vx)$ in $V$, we can use integration by part to estimate the Fourier transform coefficient 
        as 
        \begin{align}
                \left|\hat{f}\left(\dfrac{\vk}{h}\right)\right| 
                & =  
                \left|
                \int_{V}\ud\vx f(\vx) e^{-\I \frac{\vk}{h}\cdot\vx} 
                \right|
                \nonumber
                \\
                & = 
                \dfrac{h^{|\valpha|}}{|\vk|^{\valpha}}
                \left|\int_{V}\ud\vx \dfrac{\partial^\valpha}{\partial \vx^\valpha}f(\vx)e^{-\I \frac{\vk}{h}\cdot\vx}\right|
                \nonumber
                \\
                & \leqslant 
                \dfrac{h^{|\valpha|}}{|\vk|^{\valpha}}
                \int_{V}\ud\vx \left|\dfrac{\partial^\valpha}{\partial \vx^\valpha}f(\vx)\right|
                = C_\valpha \frac{h^{|\valpha|}}{|\vk|^{\valpha}},
                \label{eqn:quaderror0_fourier_est}
        \end{align} 
        with any derivative order $\valpha \geqslant 0$ where $|\valpha| = \sum_i \valpha_i$ and $|\vk|^\valpha = \prod_i |\vk_i|^{\valpha_i}$. 
        Plugging this estimate into \cref{eqn:quaderror0_fourier}, we obtain 
        \begin{equation*}
                \left|\mathcal{E}_V(f, \mathcal{X})\right|
                \leqslant 
                C_\valpha h^{|\valpha|}
                \sum_{\vk\in 2\pi \mathbb{Z}^d \setminus \{\bm{0}\}} 
                \frac{1}{|\vk|^{\valpha}},
        \end{equation*}
        which then proves the lemma by choosing an arbitrary $\valpha$ with $|\valpha| = l > d$.
\end{proof}

\begin{rem}\label{rem:quaderror0}
        If we replace $f(\vx)$ by $f(\vx, \vy)$ defined in $V\times V_Y$ which is smooth and periodic with respect to $\vx$ for each $\vy\in V_Y$ 
        and satisfies $\sup_{\vx\in V, \vy\in V_Y}|\partial_{\vx}^{\valpha}f(\vx, \vy)| < \infty$ for any $\valpha \geqslant 0$, \cref{lem:quaderror0} can be generalized
        as 
        \[
                \left|\mathcal{E}_V(f(\cdot, \vy), \mathcal{X})\right| \leqslant C_l m^{-l},\quad \forall l > 0, \forall \vy \in V_Y,
        \] 
        where constant $C_l$ is independent of $\vy \in V_Y$ based on the prefactor estimate in \cref{eqn:quaderror0_fourier_est}.
\end{rem}

\begin{lem}\label{lem:quaderror1}
        Let $f(\vx)$ be periodic with respect to $V = [-\frac12,\frac12]^d$ and smooth everywhere except at $\vx = \bm{0}$ with order $\gamma \geqslant -d+1$.     
        At $\vx = \bm{0}$, $f(\vx)$ is set to $0$.
        The quadrature error of a trapezoidal rule using an $m^d$-sized uniform mesh $\mathcal{X}$ that contains $\vx = \bm{0}$ can be estimated as 
        \[
                \left|\mathcal{E}_V(f, \mathcal{X})\right| \leqslant C \mathcal{H}_{V, \bm{0}}^{d+\max(1,\gamma)}(f) m^{-(d + \gamma)}.
        \]
        If $f(\bm{0})$ is set to an $\Or(1)$ value in the calculation,
        it introduces additional $\Or(m^{-d})$ quadrature error.

\end{lem}


\begin{proof}
        Define a cutoff function $\psi \in \mathbb{C}_{c}^{\infty}(\mathbb{R}^n)$ satisfying
        \[
                \psi(\mathbf{x})= \begin{cases}1, & |\mathbf{x}| \leqslant \frac{1}{2} \\ 0, & |\mathbf{x}| \geqslant 1\end{cases},
        \]
        and denote its scaling as $\psi_{L}(\vx) = \psi(\vx/L)$ that is compactly supported in $|\vx| \leqslant L$. 
        Let $h = 1/m$ be the unit length of the uniform mesh $\mathcal{X}$. 
        Using the cutoff function, we split $f(\vx)$ as 
        \[
                f(\vx) = f(\vx)\psi_\frac12(\vx) + f(\vx)(1- \psi_\frac12(\vx)),
        \]
        where the first term is compactly supported in $V$ and the second term is smooth in $V$ and satisfies the periodic boundary condition on $\partial V$. 
        Accordingly, the quadrature error can be split into 
        \[
        \mathcal{E}_V(f, \mathcal{X}) = \mathcal{E}_V(f\psi_\frac12, \mathcal{X}) + \mathcal{E}_V(f(1-\psi_\frac12), \mathcal{X}),
        \] 
        where the second error decays super-algebraically with respect to $m$ according to \cref{lem:quaderror0}.
        The problem is reduced to estimating the first quadrature error for a localized integrand $f\psi_\frac12$. 
        In the following discussion, we abuse the notation $f$ to denote the localized function $f\psi_\frac12$ and assume that $f$ is 
        compactly supported in $V$ and smooth everywhere except at $\bm{0}$ with order $\gamma$. 

        Since $f(\bm{0}) = 0$, the trapezoidal rule over $f(\vx)$ using $\mathcal{X}$ satisfies  
        \[
        \mathcal{Q}_V(f, \mathcal{X}) = \mathcal{Q}_V(f(1 - \psi_h), \mathcal{X}),
        \]
        and accordingly its quadrature error can be split as
        \begin{align}
        \mathcal{E}_V(f, \mathcal{X})
        & = 
        \mathcal{I}_V(f\psi_h) + \mathcal{I}_V(f(1 - \psi_h)) - 
        \mathcal{Q}_V(f(1 - \psi_h), \mathcal{X})         
        \nonumber       \\
        & =
        \mathcal{I}_V(f\psi_h) + \mathcal{E}_V(f(1 - \psi_h), \mathcal{X}).
        \label{eqn:error_split}
        \end{align}

        The first part in \cref{eqn:error_split} can be estimated as
        \begin{equation}\label{eqn:01}
                \left|\mathcal{I}_V(f\psi_h)\right| 
                \leqslant C\int_{|\vx|\leqslant h} |f(\vx)|\ud\vx
                \leqslant 
                C
                \mathcal{H}_{V,\bm{0}}^0(f)
                \int_{|\vx|\leqslant h} |\vx|^{\gamma}\ud\vx 
                \leqslant
                C
                \mathcal{H}_{V,\bm{0}}^0(f)
                h^{d+\gamma},
        \end{equation}
        using the algebraic singularity characterization \cref{eqn:fractional_univariate_ineq} for $f(\vx)$ at $\vx = \bm{0}$.
        The second part in \cref{eqn:error_split} can be reformulated using the Poisson summation formula as
        \begin{equation}\label{eqn:error_poisson}
                \mathcal{E}_V\left(f(1 - \psi_h), \mathcal{X}\right) 
                = - \sum_{\vk \in 2\pi\mathbb{Z}^d\setminus \{\bm{0}\}} \hat{f}_{\psi, h}\left(\dfrac{\vk}{h}\right),
        \end{equation}
        where $f_{\psi, h} = f(1 - \psi_h)$ and its Fourier transform can be estimated as 
        \begin{align*}
                \left|\hat{f}_{\psi, h}\left(\dfrac{\vk}{h}\right)\right| 
                & = 
                \dfrac{h^{|\valpha|}}{|\vk|^{\valpha}}
                \left|
                        \int_{\mathbb{R}^d}\ud\vx \dfrac{\partial^\valpha}{\partial \vx^\valpha}f(\vx)(1 - \psi_h(\vx)) e^{- \I\frac{\vk}{h}\cdot\vx}
                \right|
                \\
                & \leqslant
                \dfrac{h^{|\valpha|}}{|\vk|^{\valpha}}
                \int_{\mathbb{R}^d}\ud\vx \left|\dfrac{\partial^\valpha}{\partial \vx^\valpha}f(\vx)(1 - \psi_h(\vx))\right|,
                \nonumber
        \end{align*} 
        with any derivative order $\valpha \geqslant 0$. 
        The derivative in the last integral can be further expanded as
        \begin{align*}
                \dfrac{\partial^\valpha}{\partial \vx^\valpha}f(\vx)(1 - \psi_h(\vx))
                & = 
                \sum_{\valpha_1 + \valpha_2  = \valpha}{ \valpha \choose \valpha_1} \dfrac{\partial^{\valpha_1}}{\partial \vx^{\valpha_1}}(1- \psi_h(\vx))
                \dfrac{\partial^{\valpha_2}}{\partial \vx^{\valpha_2}}f(\vx).
        \end{align*}
        Using the locality of $1 - \psi_h(\vx)$ and $f(\vx)$ and the inequality 
        $|\partial^{\vbeta}_\vx f(\vx)|\leqslant \mathcal{H}_{V, \bm{0}}^{|\valpha|}(f) |\vx|^{\gamma - |\vbeta|}$ with any $\vbeta\leqslant\valpha$, 
        we can estimate this derivative as 
        \begin{equation*}\label{eqn:02}
                \left|\dfrac{\partial^\valpha}{\partial \vx^\valpha}f(\vx)(1 - \psi_h(\vx))\right| 
                \leqslant C 
                \mathcal{H}_{V, \bm{0}}^{|\valpha|}(f)
                \begin{cases}
                        0, & |\mathbf{x}| \leqslant \frac{1}{2}h \\ 
                        h^{\gamma - |\valpha|}, & \frac12 h\leqslant|\mathbf{x}| \leqslant h \\ 
                        |\mathbf{x}|^{\gamma-|\valpha|}, & h \leqslant|\mathbf{x}| \leqslant \frac12 \\ 
                        0, & |\mathbf{x}|> \frac12
                \end{cases}.    
        \end{equation*}
        Using this estimate, the associated integral can be bounded as
        \begin{align*}
                \int_{\mathbb{R}^d}\ud\vx \left|\dfrac{\partial^\valpha}{\partial \vx^\valpha}f(\vx)(1 - \psi_h(\vx))\right|
                & \leqslant
                C
                \mathcal{H}_{V, \bm{0}}^{|\valpha|}(f)
                \left(
                \int_{\frac12 h}^h h^{\gamma-|\valpha|}r^{d-1}\ud r + \int_{h}^{\frac12} r^{\gamma - |\valpha| + d - 1}\ud r
                \right)
                \\
                & \leqslant 
                C
                \mathcal{H}_{V, \bm{0}}^{|\valpha|}(f)
                (h^{\gamma+d-|\valpha|} + 1). 
        \end{align*}
        Plugging this estimate into \cref{eqn:error_poisson}, we obtain 
        \[
                \left|\mathcal{E}_V(f(1 - \psi_h), \mathcal{X})\right|
                \leqslant 
                C
                \mathcal{H}_{V, \bm{0}}^{|\valpha|}(f)
                \sum_{\vk \in 2\pi\mathbb{Z}^d\setminus\{\bm{0}\}}          
                \dfrac{1}{ |\vk|^{\valpha}}\left(h^{\gamma + d} + h^{|\valpha|}\right).
        \]
        Choosing an arbitrary $\valpha$ with $|\valpha| = \max(d+1, d+\gamma)$, we obtain 
        \[
                \left|\mathcal{E}_V(f(1 - \psi_h), \mathcal{X})\right|
                \leqslant 
                C \mathcal{H}_{V, \bm{0}}^{|\valpha|}(f) h^{\gamma + d},
        \]
        which together with \cref{eqn:error_split} and \cref{eqn:01} proves the lemma. 
\end{proof}

\begin{rem}\label{rem:quaderror1}
        If we replace $f(\vx)$ by $f(\vx, \vy)$ defined in $V \times V_Y$ which is smooth everywhere in $V\times V_Y$ except at $\vx = \bm{0}$ with 
        order $\gamma$, \cref{lem:quaderror1} can be generalized to 
        \[
                \left|\mathcal{E}_V(f(\cdot, \vy), \mathcal{X})\right| \leqslant C \mathcal{H}_{V\times V_Y, (\bm{0},\cdot)}^{d+\max(1,\gamma)}(f) m^{-(d + \gamma)}, \quad \forall \vy \in V_Y,
        \] 
        where the prefactor applies uniformly across all  $\vy\in V_Y$. 
        This generalization can be obtained using the prefactor estimate in \cref{lem:quaderror1} and the fact that 
        \[
                \mathcal{H}_{V, \bm{0}}^{l}(f(\cdot, \vy)) \leqslant 
                \mathcal{H}_{V\times V_Y, (\bm{0},\cdot)}^{l}(f), \quad \forall l \geqslant 0, \forall \vy\in V_Y,
        \]
        based on the definitions of the two quantities in \cref{eqn:fractional_univariate} and \cref{eqn:fractional_multivariate}. 
\end{rem}

\begin{lem}\label{lem:quaderror2}
        Let $f(\vx) = f_1(\vx)f_2(\vx)$ where $f_1(\vx)$ and $f_2(\vx)$ are periodic with respect to $V = [-\frac12,\frac12]^d$ and
        \begin{itemize}
                \item $f_1(\vx)$ is smooth everywhere except at $\vx = \vz_1 = \bm{0}$ with order $\gamma\leqslant 0$,
                \item $f_2(\vx)$ is smooth everywhere except at $\vx = \vz_2 \neq \bm{0}$ with order $0$.
        \end{itemize}
        Consider an $m^d$-sized uniform mesh $\mathcal{X}$ in $V$. Assume that $\mathcal{X}$ satisfies that $\vz_1, \vz_2$ are either on the 
        mesh or $\Theta(m^{-1})$ away from any mesh points, and $m$ is sufficiently large that $|\vz_1 - \vz_2| = \Omega(m^{-1})$.
        At $\vx = \vz_1$ and $\vx = \vz_2$, $f(\vx)$ is set to $0$. 
        The trapezoidal rule using $\mathcal{X}$ has quadrature error
        \[
                \left|\mathcal{E}_V(f, \mathcal{X})\right| \leqslant C \mathcal{H}_{V, \vz_1}^{d+1}(f_1)\mathcal{H}_{V, \vz_2}^{d+1}(f_2)m^{-(d + \gamma)}.
        \]
        If $f(\vz_1)$ and $f(\vz_2)$ are set to arbitrary $\Or(1)$ values, it introduces additional $\Or(m^{-d})$ quadrature error. 
\end{lem}
\begin{proof}
        First, we can introduce a proper translation $f(\vx) \rightarrow f(\vx - \vx_0)$ and move the nonsmooth points $\vz_1$ and $\vz_2$ 
        to both lie in the smaller cube $[-\frac14, \frac14]^d$ in $V$.
        Assume such a translation has been applied to $f$ and the mesh $\mathcal{X}$, which does not change the integral value due to the function periodicity.
        Using the cut-off function $\psi_\frac12(\vx)$, we split $f(\vx)$ as 
        \begin{equation}
                f(\vx) = f(\vx)\psi_\frac12^2(\vx) + f(\vx)(1- \psi_\frac12^2(\vx)),
        \end{equation}
        where the quadrature error for the second term decays super-algebraically according to the analysis in the proof of \cref{lem:quaderror1}. 
        The problem is reduced to the quadrature error estimate for the localized integrand $f\psi_\frac12^2$ which can be decomposed as 
        $
                f\psi_\frac12^2 = (f_1\psi_\frac12) (f_2\psi_\frac12).
        $
        In the following, we abuse the notation $f_i$ to denote $f_i\psi_\frac12$ and $f$ to denote $f\psi_\frac12^2$, and then assume that $f_1$ and $f_2$ is 
        are both compactly supported in $V$ and smooth everywhere except at $\vz_1$ and $\vz_2$ with order $\gamma$ and $0$, respectively. 
        Denote $h = 1/m$ be the unit length of $\mathcal{X}$.
        The remaining analysis follows the proof of \cref{lem:quaderror1}. 

        First, the quadrature error of the trapezoidal rule can be reformulated as 
        \begin{equation}
                \label{eqn:error_split2}
                \mathcal{E}_V(f, \mathcal{X})
                = 
                \left(
                        \mathcal{I}_V\left(f\psi_{h,1}\right)+
                        \mathcal{I}_V\left(f\psi_{h,2}\right)
                \right)+
                \mathcal{E}_V\left(f(1 -\psi_{h,1})(1 -\psi_{h,2}), \mathcal{X}\right),
        \end{equation}
        where $\psi_{h,1}(\vx)$ and $\psi_{h,2}(\vx)$ are two scaled and shifted cutoff functions centered at $\vz_1$ and $\vz_2$, respectively, 
        with cutoff radius $C_1h$ as 
        \begin{align*}
                \psi_{h,1}(\vx) = \psi\left(\dfrac{\vx- \vz_1}{C_1h}\right),
                \qquad 
                \psi_{h,2}(\vx) = \psi\left(\dfrac{\vx- \vz_2}{C_1h}\right),
        \end{align*}
        where $C_1$ is a constant such that the two balls $B(\vz_1, C_1h)$ and $B(\vz_2, C_1h)$  
        do not overlap with each other or with any mesh points in $\mathcal{X}$ other than 
        $\vz_1$ and $\vz_2$. 
        Here we use $B(\vx, r)$ to denote a ball centered at $\vx$ with radius $r$. Such $C_1$ exists based on the assumptions over $\mathcal{X}$ in the lemma.  

        The first part in \cref{eqn:error_split2} can be estimated directly as 
        \begin{align*}
                & 
                \left| 
                        \mathcal{I}_V\left(f\psi_{h,1}\right)+
                        \mathcal{I}_V\left(f\psi_{h,2}\right)
                \right|
                \\
                \leqslant&  
                \int_{B(\vz_1, C_1h)}\ud\vx |f(\vx)|
                +\int_{B(\vz_2, C_1h)}\ud\vx |f(\vx)|
                \\
                \leqslant& 
                C
                \mathcal{H}_{V, \vz_1}^0(f_1)
                \mathcal{H}_{V, \vz_2}^0(f_2)
                \left(
                        \int_{B(\vz_1, C_1h)}\ud\vx |\vx - \vz_1|^{\gamma}
                        + 
                        \int_{B(\vz_2, C_1h)}\ud\vx h^\gamma|\vx - \vz_2|^{0}
                \right)
                \\
                \leqslant & 
                C
                \mathcal{H}_{V, \vz_1}^0(f_1)
                \mathcal{H}_{V, \vz_2}^0(f_2)
                h^{d+\gamma},
        \end{align*}
        where the second inequality uses the two estimates from \cref{eqn:fractional_univariate_ineq} as
        \begin{align*}
                |f_1(\vx)| & \leqslant \mathcal{H}_{V, \vz_1}^0(f_1)|\vx - \vz_1|^{\gamma},     
                \\
                |f_2(\vx)| & \leqslant \mathcal{H}_{V, \vz_2}^0(f_2)|\vx - \vz_2|^0.
        \end{align*}
        The second part in \cref{eqn:error_split2} can be estimated using Poisson summation formula with any $\valpha \geqslant \bm{0}$ as
        \begin{align}
                & 
                \left|\mathcal{E}_V\left(
                f(1 -\psi_{h,1})(1 -\psi_{h,2}), \mathcal{X}\right)\right|
                \nonumber\\
                \leqslant&  
                \sum_{\vk \in 2\pi\mathbb{Z}^d\setminus\{\bm{0}\}} \left|\hat{f}_{\psi, h}\left(\dfrac{\vk}{h}\right)\right|
                \nonumber\\
                \leqslant&
                        \sum_{\vk \in 2\pi\mathbb{Z}^d\setminus\{\bm{0}\}}  \dfrac{h^{|\valpha|}}{|\vk|^{\valpha}}
                \int_{\mathbb{R}^d}\ud\vx 
                \left|
                        \dfrac{\partial^\valpha}{\partial \vx^\valpha}
                \left(
                        f(\vx)(1 -\psi_{h,1}(\vx))(1 -\psi_{h,2}(\vx))
                \right)
                \right|
                \nonumber \\
                \leqslant 
                & 
                C \sum_{\vk \in 2\pi\mathbb{Z}^d\setminus\{\bm{0}\}}  \dfrac{h^{|\valpha|}}{|\vk|^{\valpha}}
                \int_{\mathbb{R}^d}\ud\vx       
                \sum_{\valpha_1+\valpha_2+\valpha_3 =\valpha} 
                \left|
                        \partial^{\valpha_1}_{\vx}(1 -\psi_{h,1}(\vx))
                        \partial^{\valpha_2}_{\vx}(1 -\psi_{h,2}(\vx))
                        \partial^{\valpha_3}_{\vx}f(\vx)
                \right|.
                \label{eqn:f_fourier2}
        \end{align}
        To estimate the last integral for each set of $(\valpha_1, \valpha_2, \valpha_3)$, we consider three cases: 
        \begin{itemize}
                \item $\valpha_1 > 0$. 
                Since $1 -\psi_{h,1}(\vx)$ is constant inside $B(\vz_1, \frac12 C_1h)$ and outside $B(\vz_1, C_1h)$, its derivatives are 
                nonzero only in the annulus $B(\vz_1, C_1)\setminus B(\vz_1, \frac12 C_1h)$. 
                The derivative term in \cref{eqn:f_fourier2} in this case can be estimated as 
                \[
                \left|
                        \partial^{\valpha_1}_{\vx}(1 -\psi_{h,1})
                        \partial^{\valpha_2}_{\vx}(1 -\psi_{h,2})
                        \partial^{\valpha_3}_{\vx}f
                \right|
                \leqslant C
                \mathcal{H}_{V, \vz_1}^{|\valpha|}(f_1)
                \mathcal{H}_{V, \vz_2}^{|\valpha|}(f_2)
                \begin{cases}
                        0 & |\vx - \vz_1| < \frac12C_1h\\
                        0 & |\vx - \vz_1| > C_1h\\      
                        h^{\gamma - |\valpha|} & \text{otherwise}                
                \end{cases}.
                \]
                The estimate for $\vx$ in the annulus around $\vz_1$ above uses for any $\vbeta \geqslant \bm{0}$
                \begin{align*}
                        \left|
                                \partial^{\vbeta}_\vx f(\vx) 
                        \right|
                        & \leqslant 
                        C
                        \sum_{\vbeta_1+\vbeta_2 = \vbeta} 
                        \left| 
                                \partial^{\vbeta_1}_\vx f_1(\vx)
                                \partial^{\vbeta_2}_\vx f_2(\vx)
                        \right|
                        \\
                        & \leqslant
                        C
                        \mathcal{H}_{V, \vz_1}^{|\vbeta|}(f_1)
                        \mathcal{H}_{V, \vz_2}^{|\vbeta|}(f_2)
                        \sum_{\vbeta_1+\vbeta_2 = \vbeta} 
                                |\vx - \vz_1|^{\gamma - |\vbeta_1|}
                                |\vx - \vz_2|^{-|\vbeta_2|} 
                        \\
                        & \leqslant
                        C
                        \mathcal{H}_{V, \vz_1}^{|\vbeta|}(f_1)
                        \mathcal{H}_{V, \vz_2}^{|\vbeta|}(f_2)
                        \sum_{\vbeta_1+\vbeta_2 = \vbeta} 
                                h^{\gamma - |\vbeta_1|}
                                h^{-|\vbeta_2|} 
                        \\
                        & \leqslant 
                        C
                        \mathcal{H}_{V, \vz_1}^{|\vbeta|}(f_1)
                        \mathcal{H}_{V, \vz_2}^{|\vbeta|}(f_2)
                        h^{\gamma - |\vbeta|},
                \end{align*}
                where the third inequality uses the fact $|\vx - \vz_1| = \Or(h)$ and $|\vx - \vz_2| = \Omega(h)$. 

                \item $\valpha_2 > 0$. 
                Similar to the first case, we could get an estimate of the derivative as 
                \[
                \left|
                        \partial^{\valpha_1}_{\vx}(1 -\psi_{h,1})
                        \partial^{\valpha_2}_{\vx}(1 -\psi_{h,2})
                        \partial^{\valpha_3}_{\vx}f
                \right|
                \leqslant 
                C
                \mathcal{H}_{V, \vz_1}^{|\valpha|}(f_1)
                \mathcal{H}_{V, \vz_2}^{|\valpha|}(f_2)
                \begin{cases}
                        0 & |\vx - \vz_2| < \frac12 C_1 h\\
                        0 & |\vx - \vz_2| > C_1 h\\     
                        h^{\gamma - |\valpha|} & \text{otherwise}                
                \end{cases}.
                \]
                
                \item $\valpha_1= \valpha_2 = \bm{0}, \valpha_3 = \valpha$. We could use the estimate in the first case for $\partial^{\valpha}_\vx f(\vx)$ to get
                \[
                \left|
                (1 -\psi_{h,1})
                (1 -\psi_{h,2})
                \partial^\valpha_\vx f
                \right|
                \leqslant C
                \mathcal{H}_{V, \vz_1}^{|\valpha|}(f_1)
                \mathcal{H}_{V, \vz_2}^{|\valpha|}(f_2)
                \begin{cases}
                        0 & |\vx - \vz_1| < \frac12C_1h\\
                        0 & |\vx - \vz_2| < \frac12C_1h\\       
                        \sum_{\vbeta_1+\vbeta_2 = \valpha}  
                        |\vx - \vz_1|^{\gamma - |\vbeta_1|}
                        |\vx - \vz_2|^{-|\vbeta_2|} 
                         & \text{otherwise with } |\vx| < 1             
                \end{cases}.
                \]
        \end{itemize}
        Based on the analysis of the three cases above, we get an estimate of the integral as 
        \[
                \int_{\mathbb{R}^d}\ud\vx
                \left|
                        \partial^{\valpha_1}_{\vx}(1 -\psi_{h,1})
                        \partial^{\valpha_2}_{\vx}(1 -\psi_{h,2})
                        \partial^{\valpha_3}_{\vx}f
                \right| 
                \leqslant 
                        C
                        \mathcal{H}_{V, \vz_1}^{|\valpha|}(f_1)
                        \mathcal{H}_{V, \vz_2}^{|\valpha|}(f_2)
                \begin{cases}
                        h^{\gamma - |\valpha| + d} & \valpha_1 > 0 \text{ or } \valpha_2 > 0 \\
                        1 + h^{\gamma - |\valpha| + d} & \valpha_1 = \valpha_2 = \bm{0}\\
                \end{cases}.
        \]
        The estimate for the case with $\valpha_1 = \valpha_2 = \bm{0}$ can be obtained directly when $\vbeta_1= \valpha$ and be
        bounded as follows when $\vbeta_1< \valpha$ 
        \begin{align*}
                & \quad  
                \int_{B(\bm{0}, 2)\setminus (B(\vz_1, \frac12C_1h)\cup B(\vz_2, \frac12C_1h))}\ud \vx
                |\vx - \vz_1|^{\gamma - |\vbeta_1|}
                |\vx - \vz_2|^{|\vbeta_1| - |\valpha|} 
                \\
                & \leqslant
                \left(
                \int_{B(\bm{0}, 2)\setminus B(\vz_1, \frac12C_1h)}\ud \vx
                |\vx - \vz_1|^{p\gamma - p|\vbeta_1|}
                \right)^{\frac{1}{p}}
                \left(
                \int_{B(\bm{0}, 2)\setminus B(\vz_2, \frac12C_1h)}\ud \vx
                |\vx - \vz_2|^{q|\vbeta_1| - q|\valpha|}
                \right)^{\frac{1}{q}}
                \\
                & \leqslant
                C
                \left(1 + h^{p\gamma - p|\vbeta_1| + d}\right)^\frac{1}{p}
                \left(1 + h^{q|\vbeta_1| - q|\valpha| + d}\right)^\frac{1}{q}
                \\
                & \leqslant
                C\left(
                1 + h^{\gamma - |\vbeta_1| + \frac{1}{p}d}
                +
                h^{|\vbeta_1| - |\valpha| + \frac{1}{q}d}
                + h^{\gamma - |\valpha| + d}
                \right)
                \leqslant
                C\left(1 + h^{\gamma - |\valpha| + d}\right),
        \end{align*}
        where the last estimate assumes $|\valpha| \geqslant \gamma + d$ 
        and is obtained by setting the H{\"o}lder inequality exponents as
        \[
                \begin{cases}
                        p = 1, q = \infty & \text{when }\ d + \gamma - |\vbeta_1| \leqslant 0 \\
                        p = \frac{d}{|\vbeta_1| - \gamma}, q = \frac{d}{d + \gamma - |\vbeta_1|} & \text{when }\ d + \gamma - |\vbeta_1| >  0 
                \end{cases}.
        \]

        Plugging this estimate into \cref{eqn:f_fourier2} and choosing an arbitrary $\valpha$ with $|\valpha| = d+1$, we obtain 
        \[
                \left|
                        \mathcal{E}_V\left(f(1 -\psi_{h,1})(1 -\psi_{h,2}), \mathcal{X}\right)
                \right|
                \leqslant 
                C 
                \mathcal{H}_{V, \vz_1}^{d+1}(f_1)
                \mathcal{H}_{V, \vz_2}^{d+1}(f_2)
                h^{d + \gamma},
        \]
        which together with \cref{eqn:error_split2} proves the lemma. 
\end{proof}

\begin{rem}\label{rem:quaderror2}
        If we replace $f_i(\vx)$ with $i=1,2$ by $f_i(\vx, \vy)$ defined in $V\times V_Y$ which is smooth everywhere in $V\times V_Y$ except at $\vx = \vz_i$ with order $\gamma$ and $0$, respectively,
        \cref{lem:quaderror2} can be generalized to 
        \[
                \left|\mathcal{E}_V(f_1(\cdot, \vy)f_2(\cdot, \vy), \mathcal{X})\right| \leqslant C \mathcal{H}_{V\times V_Y, (\bm{0},\cdot)}^{d+1}(f_1) \mathcal{H}_{V\times V_Y, (\bm{0},\cdot)}^{d+1}(f_2)m^{-(d + \gamma)}, \quad \forall \vy \in V_Y,
        \] 
        where the prefactor applies uniformly across all $\vy \in V_Y$. 
        This generalization can be obtained using the prefactor characterization in \cref{lem:quaderror1} and a similar discussion as in \cref{rem:quaderror1}.
\end{rem}

\begin{lem}\label{lem:quaderror3}
        Let $f(\vx_1, \vx_2) = f_1(\vx_1, \vx_2)f_2(\vx_1, \vx_2)$ where  $f_1(\vx_1, \vx_2)$ and $f_2(\vx_1, \vx_2)$ are periodic with respect to $\vx_1, \vx_2 \in V = [-\frac12,\frac12]^d$ and
        \begin{itemize}
                \item $f_1(\vx_1, \vx_2)$ is smooth everywhere except at $\vx_1 = \bm{0}$ with order $\gamma_1$, 
                \item $f_2(\vx_1, \vx_2)$ is smooth everywhere except at $\vx_2 = \bm{0}$ with order $\gamma_2$. 
        \end{itemize}
        Consider an $m^d$-sized uniform mesh $\mathcal{X}$ in $V$ that contains $\vx = \bm{0}$.
        At $\vx_1 = \bm{0}$ or $\vx_2 = \bm{0}$, $f(\vx_1, \vx_2)$ is set to zero. 
        The trapezoidal rule using $\mathcal{X}\times\mathcal{X}$ for $f(\vx_1, \vx_2)$ has quadrature error
        \[
                \left|
                        \mathcal{E}_{V\times V}(f, \mathcal{X}\times \mathcal{X})
                \right|
                \leqslant 
                C
                \mathcal{H}_{V\times V,(\bm{0},\cdot)}^{d+\max(1, \gamma_1,\gamma_2)}(f_1)
                \mathcal{H}_{V\times V,(\cdot,\bm{0})}^{d+\max(1, \gamma_1,\gamma_2)}(f_2)
                m^{-(d+\min_i\gamma_i)},                
        \]
        If $f(\vx_1, \bm{0})$ and $f(\bm{0}, \vx_2)$ are set to arbitrary $\Or(1)$ values, it introduces additional  
        $\Or(m^{-d})$ quadrature error.
\end{lem}

\begin{proof}
The quadrature error for $f$ can be split into two parts
\begin{equation} \label{eqn:error_split3}
        \mathcal{E}_{V\times V}(f, \mathcal{X}\times \mathcal{X}) 
         = 
        \mathcal{E}_{V}\left(\int_V \ud\vx_2 f(\cdot, \vx_2), \mathcal{X}\right) 
        +
        \dfrac{|V|}{m^d}\sum_{\vx_1 \in \mathcal{X}} 
        \mathcal{E}_{V}\left(f(\vx_1, \cdot), \mathcal{X}\right). 
\end{equation}  

For the first quadrature error, $\int_V \ud\vx_2 f(\vx_1, \vx_2)$ as a function of $\vx_1$ is periodic with respect to $V$. 
Using the Leibniz integral rule, we can show that the integrand is smooth everywhere in $V$ except at $\vx_1 = \bm{0}$  and its derivatives can be estimated 
as 
\begin{align*}
        \left|
                \dfrac{\partial^\valpha}{\partial \vx_1^\valpha}
                \int_V \ud\vx_2 f(\vx_1, \vx_2)
        \right|
        & = 
        \left|
                \int_V \ud\vx_2
                \dfrac{\partial^\valpha}{\partial \vx_1^\valpha}
                f(\vx_1, \vx_2)
        \right|
        \\
        & \leqslant 
        C
        \int_V \ud\vx_2
        \sum_{\valpha_1 + \valpha_2 = \valpha}
        \left|
                \partial_{\vx_1}^{\valpha_1}
                f_1(\vx_1, \vx_2)
        \right|
        \left|
                \partial_{\vx_1}^{\valpha_2}
                f_2(\vx_1, \vx_2)
        \right|
        \\
        & 
        \leqslant
        C
        \int_V \ud\vx_2
        \sum_{\valpha_1 + \valpha_2 = \valpha}
                \left(
                        \mathcal{H}_{V\times V,(\bm{0},\cdot)}^{|\valpha|}(f_1)
                        |\vx_1|^{\gamma_1 - |\valpha_1|}
                \right)
                \left(
                        \mathcal{H}_{V\times V,(\cdot, \bm{0})}^{|\valpha|}(f_2)
                        |\vx_2|^{\gamma_2}
                \right)
        \\
        & 
        \leqslant
        C
                        \mathcal{H}_{V\times V,(\bm{0},\cdot)}^{|\valpha|}(f_1)
                        \mathcal{H}_{V\times V,(\cdot, \bm{0})}^{|\valpha|}(f_2)
        |\vx_1|^{\gamma_1 - |\valpha|},
\end{align*}
where the third inequality uses the estimate in \cref{eqn:fractional_multivariate} for multivariate functions with algebraic singularity. 
This derivative estimate suggests that $\int_V \ud\vx_2 f(\vx_1, \vx_2)$ is smooth everywhere in $V$ except at $\vx_1 = \bm{0}$ with order $\gamma_1$, 
and its algebraic singularity characterization satisfies, 
\[
        \mathcal{H}_{V,\bm{0}}^l
        \left(\int_V\ud\vx_2 f(\cdot, \vx_2)\right)
        \leqslant C
        \mathcal{H}_{V\times V,(\bm{0},\cdot)}^{l}(f_1)
        \mathcal{H}_{V\times V,(\cdot, \bm{0})}^{l}(f_2),\quad \forall l \geqslant 0.
\]
\cref{lem:quaderror1} then gives 
\begin{equation}\label{eqn:quaderror3_01}
        \left|
                \mathcal{E}_{V}\left(\int_V \ud\vx_2 f(\cdot, \vx_2), \mathcal{X}\right)
        \right| 
        \leqslant C 
                        \mathcal{H}_{V\times V,(\bm{0},\cdot)}^{d + \max(1,\gamma_1)}(f_1)
                        \mathcal{H}_{V\times V,(\cdot, \bm{0})}^{d + \max(1,\gamma_1)}(f_2)
        m^{-(d + \gamma_1)}.
\end{equation}

For the second quadrature error with each $\vx_1 \in \mathcal{X}$, $f(\vx_1, \vx_2)$ as a function of $\vx_2$ is periodic and smooth everywhere except at $\vx_2 = \bm{0}$
with order $\gamma_2$. 
Fixing $\vx_1$, \cref{lem:quaderror1} gives 
\[
        \left|
                \mathcal{E}_{V}\left(f(\vx_1, \cdot), \mathcal{X}\right)
        \right| 
        \leqslant C \mathcal{H}_{V, \bm{0}}^{d+\max(1,\gamma_2)}(f(\vx_1, \cdot)) 
        m^{-(d + \gamma_2)}, 
\]
where the characterization prefactor with $l = d + \max(1,\gamma_2)$ and any $\vx_1\in V$ can be further estimated by its definition in \cref{eqn:fractional_univariate} as
\begin{align}
        \mathcal{H}_{V, \bm{0}}^{l}(f(\vx_1, \cdot))
        & =  
        \max_{
                |\valpha|\leqslant l 
        }
        \left\|
                \dfrac{\partial^{\valpha}_{\vx_2}f(\vx_1, \vx_2)}{
                |\vx_2|^{\gamma_2 - |\valpha|}
        }
        \right\|_{L^\infty(V)}
        \nonumber\\
        & \leqslant 
        C
        \max_{
                |\valpha|\leqslant l 
        }
        \sum_{\valpha_1 + \valpha_2= \valpha}
        \left\|
                \partial^{\valpha_1}_{\vx_2}f_1(\vx_1, \vx_2)
                \dfrac{\partial^{\valpha_2}_{\vx_2}f_2(\vx_1, \vx_2)}{
                |\vx_2|^{\gamma_2 - |\valpha_2|}
                }
                |\vx_2|^{|\valpha_1|}
        \right\|_{L^\infty(V)}
        \nonumber\\
        & \leqslant 
        C 
        \max_{
                |\valpha|\leqslant l 
        }
        \sum_{\valpha_1 + \valpha_2 = \valpha}
        \left(
                \mathcal{H}_{V\times V,(\bm{0},\cdot)}^{l}(f_1)|\vx_1|^{\gamma_1}
        \right)
        \left(
                \mathcal{H}_{V\times V,(\cdot, \bm{0})}^{l}(f_2)
        \right)
        \left\|
                |\vx_2|^{|\valpha_1|}
        \right\|_{L^\infty(V)}
        \nonumber\\
        & 
        \leqslant
        C
                \mathcal{H}_{V\times V,(\bm{0},\cdot)}^{l}(f_1)
                \mathcal{H}_{V\times V,(\cdot,\bm{0})}^{l}(f_2)
        |\vx_1|^{\gamma_1},
        \label{eqn:quaderror3_02}
\end{align}
where the third inequality uses the estimate in \cref{eqn:fractional_multivariate}. 

Plugging the two estimates \cref{eqn:quaderror3_01} and \cref{eqn:quaderror3_02} above into \cref{eqn:error_split3}, we have 
\begin{align*}
        \left|
                \mathcal{E}_{V\times V}(f, \mathcal{X}\times \mathcal{X})         
        \right|
        & \leqslant 
        C
        \mathcal{H}_{V\times V,(\bm{0},\cdot)}^{l_*}(f_1)
        \mathcal{H}_{V\times V,(\cdot,\bm{0})}^{l_*}(f_2)
        \left(
        m^{-(d+\gamma_1)} + 
        \dfrac{|V|}{m^d}\sum_{\vx_1 \in\mathcal{X}} |\vx_1|^{\gamma_1} m^{-(d+\gamma_2)}
        \right)
        \\
        & \leqslant
        C
        \mathcal{H}_{V\times V,(\bm{0},\cdot)}^{l_*}(f_1)
        \mathcal{H}_{V\times V,(\cdot,\bm{0})}^{l_*}(f_2)
        m^{-(d+\min_i\gamma_i)},
\end{align*}
where $l_* = d + \max(1, \gamma_1, \gamma_2)$. 
This finishes the proof. 
\end{proof}

\begin{rem}\label{rem:quaderror3}
        If we replace $f_i(\vx_1, \vx_2)$ by $f_i(\vx_1, \vx_2, \vy)$ that are smooth and periodic with respect to $\vy\in V_Y$, \cref{lem:quaderror3} can be generalized as 
        \[
                \left|
                        \mathcal{E}_{V\times V}(f(\cdot, \cdot, \vy), \mathcal{X}\times \mathcal{X})
                \right|
                \leqslant 
                C
                \mathcal{H}_{V\times V\times V_Y,(\bm{0},\cdot, \cdot)}^{d+\max(1, \gamma_1,\gamma_2)}(f_1)
                \mathcal{H}_{V\times V\times V_Y,(\cdot,\bm{0}),\cdot}^{d+\max(1, \gamma_1,\gamma_2)}(f_2)
                m^{-(d+\min_i\gamma_i)}, \ \forall \vy \in V_Y,               
        \]
        where the prefactor applies uniformly across all $\vy \in V_Y$. 
\end{rem}

\begin{lem}\label{lem:quaderror4}
        Let $f(\vx_1, \vx_2) = f_1(\vx_1, \vx_2)f_2(\vx_1, \vx_2)f_3(\vx_1, \vx_2-\vx_1)$   satisfy that 
        \begin{itemize}
                \item $f_1(\vx_1, \vx_2)$ and $f_2(\vx_1, \vx_2)$ are the same as in \cref{lem:quaderror3},
                \item $f_3(\vx_1, \vz)$ is periodic with respect to $\vx_1, \vz \in V$ and smooth everywhere except at $\vz = \bm{0}$ with order $0$.
        \end{itemize} 
        Consider a uniform mesh $\mathcal{X}$ in $V$ that is of size $m^d$ and contains $\vx = \bm{0}$.
        At $\vx_1 = \bm{0}$ or $\vx_2 = \bm{0}$, $f(\vx_1, \vx_2)$ is set to $0$. 
        The trapezoidal rule using $\mathcal{X}\times\mathcal{X}$ for $f(\vx_1, \vx_2)$ has quadrature error
        \[
                \left|
                        \mathcal{E}_{V\times V}(f, \mathcal{X}\times \mathcal{X})
                \right|
                \leqslant 
                C 
                \mathcal{H}_{V\times V, (\bm{0}, \cdot)}^{d + \max(1, \gamma_1, \gamma_2)}(f_1)
                \mathcal{H}_{V\times V, (\cdot, \bm{0})}^{d + \max(1, \gamma_1, \gamma_2)}(f_2)
                \mathcal{H}_{V\times V, (\cdot, \bm{0})}^{d + \max(1, \gamma_1, \gamma_2)}(f_3)
                m^{-(d+\min_i\gamma_i)}.
        \]
        If $f(\vx_1, \bm{0})$ and $f(\bm{0}, \vx_2)$ are set to arbitrary $\Or(1)$ values, it introduces additional  
        $\Or(m^{-d})$ quadrature error.
\end{lem}

\begin{proof}
        For this new function, we still split the quadrature error into the two parts as 
        \begin{equation} \label{eqn:error_split4}
                \mathcal{E}_{V\times V}(f, \mathcal{X}\times \mathcal{X}) 
                 = 
                \mathcal{E}_{V}\left(\int_V \ud\vx_2 f(\cdot, \vx_2), \mathcal{X}\right) 
                +
                \dfrac{|V|}{m^d}\sum_{\vx_1 \in \mathcal{X}} 
                \mathcal{E}_{V}\left(f(\vx_1, \cdot), \mathcal{X}\right). 
        \end{equation}  

        In the first part of the quadrature error, $\int_V \ud\vx_2 f(\vx_1, \vx_2)$ as a function of $\vx_1$ is periodic with respect to $V$
        and below we first show it to be smooth everywhere except at $\vx_1 = \bm{0}$ with order $\gamma_1$.

        In order to exploit the existing nonsmoothness analysis in \cref{lem:nonsmooth_integral}, we define an auxiliary function
        \[
                F(\vx_1, \vy) = \int_{V}\ud\vx_2 f_1(\vy, \vx_2) f_2(\vy, \vx_2) f_3(\vy, \vx_2 - \vx_1),
        \]
        which satisfies that $\int_V \ud \vx_2 f(\vx_1, \vx_2) = F(\vx_1, \vx_1)$.
        First fixing $\vy$, the integrand for $F(\vx_1, \vy)$ as a function of $\vx_1, \vx_2$ meets the condition in \cref{lem:nonsmooth_integral} and thus 
        $F(\vx_1, \vy)$ is smooth everywhere with respect to $\vx_1 \in V$ except at $\vx_1 = \bm{0}$ with order $0$. 
        Next, fix $\vx_1$ and consider a point $\vy_0\neq\bm{0}\in V$. 
        It can be verified that there is an open domain containing $\vy_0$ where with any $\vy$ in this domain the integrand is smooth with respect to $\vx_2$ except at $\vx_2 = \bm{0}$ and $\vx_2 = \vx_1$, 
        and the absolute integrand is bounded by $C|\vx_2|^{\gamma_2}$ (from the boundedness of $f_1,f_3$ and the algebraic singularity of $f_2$) which is integrable in $V$. 
        This meets the condition for the Leibniz integral rule around $\vy_0$ and thus $F(\vx_1, \vy)$ is smooth with respect to $\vy\in V$ except at $\vy = \bm{0}$. 
        These two discussions thus show that $F(\vx_1, \vy)$ is smooth at any points with $\vx_1 \neq \bm{0}, \vy\neq \bm{0}$. 

        Furthermore, any partial derivative of $F(\vx_1, \vy)$ over $\vy$ in $V\setminus\{\bm{0}\}$ can be estimated as 
        \begin{align*}
                \left| 
                        \partial^{\valpha}_\vy F(\vx_1, \vy) 
                \right|
                &  
                \leqslant 
                C
                \int_{V}\ud\vx_2 
                \sum_{\valpha_1+\valpha_2+\valpha_3 = \valpha}
                \left|
                        \partial^{\valpha_1}_\vy f_1(\vy, \vx_2) 
                        \partial^{\valpha_2}_\vy f_2(\vy, \vx_2) 
                        \partial^{\valpha_3}_\vy f_3(\vy, \vx_2 - \vx_1)
                \right|
                \\
                & \leqslant
                C
                \int_{V}\ud\vx_2 
                \sum_{\valpha_1+\valpha_2+\valpha_3 = \valpha}
                \mathcal{H}_{V\times V, (\bm{0}, \cdot)}^{|\valpha|}(f_1)|\vy|^{\gamma_1 - |\valpha_1|}
                \mathcal{H}_{V\times V, (\cdot, \bm{0})}^{|\valpha|}(f_2)|\vx_2|^{\gamma_2}
                \mathcal{H}_{V\times V, (\cdot, \bm{0})}^{|\valpha|}(f_3)
                \\
                &\leqslant 
                C
                \mathcal{H}_{V\times V, (\bm{0}, \cdot)}^{|\valpha|}(f_1)
                \mathcal{H}_{V\times V, (\cdot, \bm{0})}^{|\valpha|}(f_2)
                \mathcal{H}_{V\times V, (\cdot, \bm{0})}^{|\valpha|}(f_3)
                |\vy|^{\gamma_1 - |\vbeta|}
                \\
                & = C C_{f_1,f_2,f_3, |\valpha|}
                |\vy|^{\gamma_1 - |\vbeta|},
        \end{align*}
        where $C_{f_1, f_2, f_3, |\valpha|}$ denotes the product of the three algebraic singularity prefactors for brevity.  

        Substituting $\vy$ by $\vx_1$ in the above estimate then shows that $F(\vx_1, \vx_1) = \int_V\ud\vx_2f(\vx_1,\vx_2)$ is smooth everywhere except at $\vx_1 = \bm{0}$ with order $\gamma_1$. 
        \cref{lem:quaderror1} then gives  
        \begin{equation}\label{eqn:quaderror4_01}
                \left|
                        \mathcal{E}_{V}\left(\int_V \ud\vx_2 f(\cdot, \vx_2), \mathcal{X}\right) 
                \right| 
                \leqslant C C_{f_1, f_2, f_3, l_1}m^{-(d + \gamma_1)},
                \quad
                \text{with } l_1 = d+\max(1, \gamma_1).
        \end{equation}

        For the second part of the quadrature error with each $\vx_1 \in \mathcal{X}$, $f(\vx_1, \vx_2)$ as a function of $\vx_2$ is periodic and smooth 
        everywhere except at $\vx_2 = \bm{0}$ and $\vx_2 = \vx_1$ with orders $\gamma_2$ and $0$, respectively.
        Applying \cref{lem:quaderror2} and \cref{rem:quaderror2} with the decomposition $f = (f_1f_2)(f_3)$ then shows that 
        \begin{align}
                \left|
                        \mathcal{E}_{V}\left(f(\vx_1, \cdot), \mathcal{X}\right)
                \right|
                & \leqslant 
                C 
                \mathcal{H}_{V, \bm{0}}^{l_2}((f_1f_2)(\vx_1, \cdot))
                \mathcal{H}_{V\times V, (\cdot, \bm{0})}^{l_2}(f_3)
                m^{-(d + \gamma_2)}
                \nonumber \\
                & \leqslant 
                C 
                \mathcal{H}_{V\times V, (\bm{0}, \cdot)}^{l_2}(f_1)
                \mathcal{H}_{V\times V, (\cdot, \bm{0})}^{l_2}(f_2)
                |\vx_1|^{\gamma_1}
                \mathcal{H}_{V\times V, (\cdot, \bm{0})}^{l_2}(f_3)
                m^{-(d + \gamma_2)}
                \nonumber
                \\
                & \leqslant 
                C
                C_{f_1,f_2,f_3,l_2} 
                |\vx_1|^{\gamma_1}
                m^{-(d + \gamma_2)},
                \label{eqn:quaderror4_02}
        \end{align}
        where $l_2 = d + \max(1, \gamma_2)$ and the second inequality is derived as 
        \begin{align*}
                \mathcal{H}_{V, \bm{0}}^{l_2}((f_1f_2)(\vx_1, \cdot))
                & = 
                \max_{
                        |\valpha|\leqslant l_2 
                }
                \left\|
                        \dfrac{\partial^{\valpha}_{\vx_2} \left(f_1(\vx_1, \vx_2) f_2(\vx_1, \vx_2)\right)}{|\vx_2|^{\gamma_2 - |\valpha|}}
                \right\|_{L^\infty(V)}
                \\
                & \leqslant 
                C
                \max_{
                        |\valpha|\leqslant l_2 
                }
                \sum_{\valpha_1 +\valpha_2 = \valpha}
                \left\|
                        \partial^{\valpha_1}_{\vx_2} f_1(\vx_1, \vx_2)
                        \dfrac{ \partial^{\valpha_2}_{\vx_2}f_2(\vx_1, \vx_2)}{|\vx_2|^{\gamma_2 - |\valpha_2|}} |\vx_2|^{|\valpha_1|}
                \right\|_{L^\infty(V)}
                \\
                & \leqslant 
                C 
                \left(
                \mathcal{H}_{V\times V, (\bm{0}, \cdot)}^{l_2}(f_1)
                |\vx_1|^{\gamma_1}
                \right)
                \mathcal{H}_{V\times V, (\cdot, \bm{0})}^{l_2}(f_2).
        \end{align*}
        
        Plugging the two estimates \cref{eqn:quaderror4_01} and \cref{eqn:quaderror4_02} into \cref{eqn:error_split3}, we have 
        \begin{align*}
                \left|
                        \mathcal{E}_{V\times V}(f, \mathcal{X}\times \mathcal{X})         
                \right|
                & \leqslant 
                C
                C_{f_1,f_2,f_3,l_*}
                \left(
                        m^{-(d+\gamma_1)} + \dfrac{|V|}{m^d}\sum_{\vx_1 \in\mathcal{X}} |\vx_1|^{\gamma_1} m^{-(d+\gamma_2)}
                \right)
                \\
                & \leqslant 
                C
                C_{f_1,f_2,f_3,l_*}
                m^{-(d+\min_i\gamma_i)}
        \end{align*}
        with $l_* = d+ \max(1, \gamma_1, \gamma_2)$. 
        This finishes the proof. 
\end{proof}

\begin{rem}\label{rem:quaderror4}
        If we replace $f_i(\vx_1, \vx_2)$ by $f_i(\vx_1, \vx_2, \vy)$ that are smooth and periodic with respect to $\vy\in V_Y$, \cref{lem:quaderror4} can be generalized as 
        \[
                \left|
                        \mathcal{E}_{V\times V}(f(\cdot, \cdot, \vy), \mathcal{X}\times \mathcal{X})
                \right|
                \leqslant 
                C
                C_{f_1,f_2,f_3}
                m^{-(d+\min_i\gamma_i)}, \ \forall \vy \in V_Y,               
        \]
        where prefactor $C_{f_1,f_2,f_3} = 
                \mathcal{H}_{V\times V\times V_Y,(\bm{0},\cdot, \cdot)}^{d+\max(1, \gamma_1,\gamma_2)}(f_1)
                \mathcal{H}_{V\times V\times V_Y,(\cdot,\bm{0},\cdot)}^{d+\max(1, \gamma_1,\gamma_2)}(f_2)
                \mathcal{H}_{V\times V\times V_Y,(\cdot,\bm{0},\cdot)}^{d+\max(1, \gamma_1,\gamma_2)}(f_3)
                $
                applies uniformly across all $\vy \in V_Y$. 
\end{rem}

\section{Proof of Lemma \ref{lem:nonsmooth_integral}: Singularity structure of functions in integral form}\label{app:nonsmooth_integral}
        In this proof, we assume $\gamma_2 = \min_i \gamma_i \leqslant \gamma_1$. 
        Otherwise if $\gamma_2 > \gamma_1$, the target function can be reformulated as
        \begin{equation*}
                F(\vy, \vz) = \int_{V - \vy}\ud\vx f(\vx + \vy, \vx, \vz)
                       = \int_{V}\ud\vx f(\vx + \vy, \vx, \vz),
        \end{equation*}
        using change of variable $\vx\rightarrow \vx + \vy$ and integrand periodicity
        with respect to $\vx \in V$. 
        With this reformulation, the following proof can be applied to $\wt{f}(\vx_1, \vx_2, \vz) = f(\vx_2, \vx_1,\vz)$ with $F(\vy,\vz) = \int_V \ud \vx \wt{f}(\vx, \vx + \vy, \vz)$,
        where the order for the second variable of $\wt{f}$ equals $\gamma_1$ and is the minimum order.

        First we study the smoothness property of $F(\vy, \vz)$ with respect to $\vy\in V$ while fixing $\vz$. 
        For notation brevity, we denote $F(\vy, \vz)$ and $f(\vx_1, \vx_2, \vz)$ by $F_\vz(\vy)$ and $f_\vz(\vx_1, \vx_2)$, respectively, 
        when $\vz$ is assumed to be a fixed point in $V_Z$.
        Consider an arbitrary open ball domain $B_\sigma$ of radius $\sigma$ in $V$ that does not contain $\bm{0}$.
        We can split $F_\vz(\vy)$ for any $\vy \in B_\sigma$ into two parts, 
        \begin{equation}\label{eqn:split_Fy}
                F_\vz(\vy) = \int_{V \setminus B_\sigma}\ud\vx f_\vz(\vx, \vx-\vy) + \int_{B_\sigma}\ud\vx f_\vz(\vx, \vx-\vy), \quad \forall \vy \in B_\sigma. 
        \end{equation}
        
        For the first term, its integrand $f_\vz(\vx, \vx-\vy)$ is smooth at $(\vx, \vy) \in \left(V \setminus (B_\sigma\cup \{\bm{0}\})\right) \times B_\sigma$ and its partial derivatives 
        over $\vy\in B_\sigma$ are integrable over $\vx \in V \setminus B_\sigma$ based on the nonsmoothness characterization in \cref{eqn:singular_behavior}. 
        We thus can use the Leibniz integral rule to prove that the first term is smooth at $\vy \in B_\sigma$ and 
        \[
        \dfrac{\partial^{\valpha}}{\partial \vy^{\valpha}} \int_{V \setminus B_\sigma}\ud\vx f_\vz(\vx, \vx-\vy)
        = \int_{V \setminus B_\sigma}\ud\vx \dfrac{\partial^{\valpha}}{\partial \vy^{\valpha}} f_\vz(\vx, \vx-\vy), 
        \quad \forall \vy \in B_\sigma.
        \]
        
        For the second term with any $\vy \in B_\sigma$, we introduce a small perturbation $\delta\vy$ to $\vy$ such that $\vy + \delta \vy \in B_\sigma$ and 
        $B_\sigma + \delta\vy$ does not contain $\bm{0}$, 
        and consider the difference between the second term evaluated at $\vy + \delta\vy$ and $\vy$ as
        \begin{align*}
                & \int_{B_\sigma}\ud\vx f_\vz(\vx, \vx- \vy - \delta\vy) -  \int_{B_\sigma}\ud\vx f_\vz(\vx, \vx- \vy)
                \\
                = & 
                \int_{B_\sigma}\ud\vx 
                \Big(
                f_\vz(\vx - \delta\vy, \vx- \vy - \delta\vy) -   f_\vz(\vx, \vx- \vy)
                +
                f_\vz(\vx, \vx- \vy - \delta\vy)  - f_\vz(\vx - \delta\vy, \vx- \vy - \delta\vy)
                \Big)
                \\
                = & 
                \left(
                \int_{B_\sigma -\delta\vy}\ud\vx- \int_{B_\sigma}\ud\vx
                \right)f_\vz(\vx, \vx- \vy)
                +  \int_{B_\sigma - \delta\vy}\ud\vx \left(f_\vz(\vx+\delta\vy, \vx- \vy)  - f_\vz(\vx, \vx- \vy)\right)
                \\
                = & \int_{\partial B_\sigma}\ud S f_\vz(\vx, \vx-\vy) (-\delta\vy) \cdot \vn(\vx) + \Or(|\delta\vy|^2)
                + \int_{B_\sigma - \delta\vy}\ud\vx \delta\vy \cdot \nabla_1f_\vz(\vx, \vx-\vy) + \Or(|\delta\vy|^2)
                \\
                = & \delta\vy \cdot \left(
                - 
                \int_{\partial B_\sigma}\ud S f_\vz(\vx, \vx-\vy) \vn(\vx) 
                +
                \int_{B_\sigma}\ud\vx \nabla_1f_\vz(\vx, \vx-\vy) 
                \right)
                + \Or(|\delta\vy|^2),
        \end{align*}
        where $\nabla_1 f_\vz(\cdot, \cdot)$ denotes the gradient of $f$ over its first variable and similarly for notation $\partial_1^\valpha f(\cdot,\cdot)$ in later use.
        This calculation shows that the second term in \cref{eqn:split_Fy} is continuous at any $\vy \in B_\sigma$ up to first order derivatives and its gradient equals 
        to the term in the parenthesis above. 

        Putting the above analysis for the two terms in \cref{eqn:split_Fy} together, we have
        \begin{equation}\label{eqn:gradient_Fy}
                \nabla F_\vz(\vy)
                = 
                \int_{V \setminus B_\sigma}\ud\vx \nabla_{\vy} f_\vz(\vx, \vx-\vy)
                - 
                \int_{\partial B_\sigma}\ud S f_\vz(\vx, \vx-\vy)  \vn(\vx) 
                +
                \int_{B_\sigma}\ud\vx \nabla_1f_\vz(\vx, \vx-\vy), \quad \vy \in B_\sigma.
        \end{equation}
        It is worth noting that the above smoothness analysis and the gradient calculation in \cref{eqn:gradient_Fy} work for any open domain $B_\sigma \subset V$ that does not contain $\bm{0}$. 
        Thus, the analysis above shows the $F_\vz(\vy)$ is continuous up to first order derivative at any $\vy \in V\setminus\{\bm{0}\}$.

        In \cref{eqn:gradient_Fy}, the integrands of the first two terms are smooth at any $\vy\in B_\sigma$ as $\vx\not\in B_\sigma$. 
        These two integrals are thus smooth with respect to $\vy\in B_\sigma$ according to the Leibniz integral rule. 
        For the third term, each entry in $\nabla_1f_\vz(\vx_1, \vx_2)$ shares similar nonsmooth behavior as $f_\vz(\vx_1, \vx_2)$ described 
        in \cref{eqn:singular_behavior} only with $\gamma_1$ changed to $\gamma_1-1$.  
        Due to this similarity, we can use the same analysis for $\int_{V}\ud\vx f_\vz(\vx, \vx - \vy)$ above to prove the continuity up to first order derivatives for the third term 
        at any $\vy \in V \setminus\{\bm{0}\}$. 
        Recursively applying this analysis, we then prove that $F_\vz(\vy)$ is smooth in $B_\sigma$ and thus in $V\setminus\{\bm{0}\}$.

        Now taking $\vz$ back into account, the above analysis shows that $F(\vy, \vz)$ is smooth with respect to $\vy \in V \setminus\{\bm{0}\}$ for any 
        fixed $\vz \in V_Z$. 
        On the other hand, with any fixed $\vy$, $f(\vx, \vx-\vy, \vz)$ is smooth with respect to $\vz\in V_Z$, has two nonsmooth points $\vx = \bm{0}$ and $\vx = \vy$ with 
        respect to $\vx$, and its partial derivatives over $\vz$ are integrable over $\vx \in V$ according to the assumption in \cref{eqn:singular_behavior}. 
        This meets the condition for the Leibniz integral rule and thus $F(\vy, \vz)$ is smooth with respect to $\vz\in V_Z$ for any fixed $\vy$. 
        Based on these two partial smoothness properties, we have $F(\vy, \vz)$ to be smooth everywhere in $V\times V_Z$ except at $\vy = \bm{0}$.
        
        Lastly, we characterize the algebraic singularity of $F(\vy, \vz)$ at $\vy = \bm{0}$ by proving that there exists constants $\{C_{\valpha, \vbeta}\}$ such 
        that 
        \begin{equation}\label{eqn:lem_nonsmooth_02}
                \left|
                        \partial_\vy^\valpha \partial_\vz^\vbeta F(\vy, \vz)
                \right|
                \leqslant 
                C_{\valpha, \vbeta} |\vy|^{\gamma_2 - |\valpha|}, 
                \qquad \forall \vy \in V\setminus\{\bm{0}\}, \vz \in V_Z, \forall \valpha, \vbeta\geqslant \bm{0}.
        \end{equation}
        Consider any $\vy\in V\setminus\{\bm{0}\}$ and choose the ball domain $B_\sigma$ centered at $\vy$ with radius $\sigma = |\vy|/2$.
        Using the Leibniz integral rule over variable $\vz$, \cref{eqn:gradient_Fy} can be generalized to provide the first-order partial derivatives of $\partial_\vz^\vbeta F(\vy, \vz)$ 
        over $\vy$ as
        \begin{align}
                \partial_\vy^\valpha \partial_\vz^\vbeta F(\vy, \vz)
                & = 
                \int_{V \setminus B_\sigma}\ud\vx \partial_\vy^\valpha \partial_\vz^\vbeta f(\vx, \vx-\vy, \vz)
                - 
                \int_{\partial B_\sigma}\ud S \partial_\vz^\vbeta f(\vx, \vx-\vy, \vz)(\valpha\cdot\vn(\vx)) 
                \nonumber\\
                & \qquad 
                +
                \int_{B_\sigma}\ud\vx \partial_1^\valpha \partial_\vz^\vbeta f(\vx, \vx-\vy, \vz), 
                \qquad \forall \vy \in B_\sigma, \vz \in V_Z, \forall |\valpha| = 1,\vbeta\geqslant \bm{0}.
                \label{eqn:gradient_Fy2}
        \end{align}
        These three terms can be estimated separately based on the assumption in \cref{eqn:singular_behavior} as
        \begin{align*}
                \left|
                \int_{V \setminus B_\sigma}\ud\vx \partial_\vy^\valpha \partial_\vz^\vbeta f(\vx, \vx-\vy, \vz)
                \right|
                & 
                \leqslant C
                \int_{V \setminus B_\sigma}\ud\vx |\vx|^{\gamma_1} |\vx - \vy|^{\gamma_2 - 1}
                \\
                & 
                \leqslant C
                \int_{V \setminus B_\sigma}\ud\vx |\vx|^{\gamma_1} |\sigma|^{\gamma_2 - 1}    
                \leqslant 
                C |\vy|^{\gamma_2 - 1}, 
                \\
                \left|
                \int_{\partial B_\sigma}\ud S \partial_\vz^\vbeta f(\vx, \vx-\vy, \vz)(\valpha\cdot\vn(\vx)) 
                \right|
                & 
                \leqslant 
                C
                \int_{\partial B_\sigma}\ud S |\vx|^{\gamma_1} |\vx - \vy|^{\gamma_2}
                \\
                & 
                \leqslant 
                C |\vy|^{\gamma_1} |\partial B_\sigma| \sigma^{\gamma_2} 
                \leqslant 
                C |\vy|^{\gamma_1 + \gamma_2 + d-1},
                \\
                \left|
                        \int_{B_\sigma}\ud\vx \partial_1^\valpha \partial_\vz^\vbeta f(\vx, \vx-\vy, \vz)
                \right|
                & 
                \leqslant 
                C
                \int_{B_\sigma}\ud\vx |\vx|^{\gamma_1 - 1}|\vx - \vy|^{\gamma_2}
                \\
                & 
                \leqslant 
                C
                |\sigma|^{\gamma_1 - 1}\int_{B_\sigma}|\vx - \vy|^{\gamma_2}
                \leqslant 
                C |\vy|^{\gamma_1 + \gamma_2 + d - 1},
        \end{align*}
        where all constants $C$'s depend on $\valpha, \vbeta$, and the corresponding prefactors in \cref{eqn:singular_behavior}.
        Using the assumption $\gamma_1 \geqslant -d + 1$, these estimates together show that the algebraic singularity
        characterization in \cref{eqn:lem_nonsmooth_02} for $F(\vy, \vz)$ holds true for all $|\valpha| = 1$. 

        Next we prove \cref{eqn:lem_nonsmooth_02} for all $|\valpha| = 2$. 
        For any $|\valpha| = 2$, decompose $\valpha = \valpha_1 + \valpha_2$ with 
        $|\valpha_1| = |\valpha_2| = 1$ and accordingly $\partial_\vy^{\valpha}\partial^\vbeta_\vz F(\vy, \vz)$ into $\partial_\vy^{\valpha_1}\left(\partial_\vy^{\valpha_2}\partial^\vbeta_\vz F(\vy, \vz)\right)$.
        Noting that $\partial_\vy^{\valpha_2}\partial^\vbeta_\vz F(\vy, \vz)$ can be expanded into three terms in \cref{eqn:gradient_Fy2} at any $\vy \neq \bm{0}$, 
        we consider the outer partial derivative $\partial_\vy^{\valpha_1}$ applied to each of the three terms. 
        The partial derivatives $\partial_\vy^{\valpha_1}$ over the first two terms in \cref{eqn:gradient_Fy2} can be estimated directly as 
        \begin{align*}
                \left|
                        \partial_\vy^{\valpha_1}
                        \int_{V \setminus B_\sigma}\ud\vx \partial_\vy^{\valpha_2} \partial_\vz^\vbeta f(\vx, \vx-\vy, \vz)
                \right|
                & 
                = 
                \left|
                        \int_{V \setminus B_\sigma}\ud\vx \partial_\vy^{\valpha_1 + \valpha_2} \partial_\vz^\vbeta f(\vx, \vx-\vy, \vz)
                \right|
                \\
                & \leqslant 
                C
                \int_{V \setminus B_\sigma}\ud\vx |\vx|^{\gamma_1} |\vx - \vy|^{\gamma_2 - 2}
                \leqslant 
                C |\vy|^{\gamma_2 - 2}, 
                \\
                \left|
                        \partial_\vy^{\valpha_1}
                        \int_{\partial B_\sigma}\ud S \partial_\vz^\vbeta f(\vx, \vx-\vy, \vz)(\valpha_2\cdot\vn(\vx)) 
                \right|
                & 
                =
                \left|
                        \int_{\partial B_\sigma}\ud S \partial_\vy^{\valpha_1}\partial_\vz^\vbeta f(\vx, \vx-\vy, \vz)(\valpha_2\cdot\vn(\vx)) 
                \right|
                \\
                &
                \leqslant 
                C
                \int_{\partial B_\sigma}\ud S |\vx|^{\gamma_1} |\vx - \vy|^{\gamma_2 - |\valpha_1|}
                \leqslant 
                C |\vy|^{\gamma_1 + \gamma_2 + d-2}.
        \end{align*}

        For the partial derivative $\partial_\vy^{\valpha_1}$ over the third term in \cref{eqn:gradient_Fy2}, note that $\partial_1^{\valpha_2} f(\vx_1, \vx_2, \vz)$ has similar 
        nonsmooth behavior as $f(\vx_1, \vx_2, \vz)$ described in \cref{eqn:singular_behavior} only with $\gamma_1$ changed to $\gamma_1-1$.  
        We thus can apply the above derivative estimate in \cref{eqn:lem_nonsmooth_02} with $|\valpha| = 1$ for $\partial_\vy^{\valpha} \partial_\vz^{\vbeta}\int_{V}\ud\vx f(\vx, \vx - \vy, \vz)$ 
        to estimate $\partial_\vy^{\valpha_1} \partial_\vz^{\vbeta}\int_{B_\sigma}\ud\vx \partial_1^{\valpha_2}f(\vx, \vx - \vy, \vz)$ and obtain that 
        \[
        \left|
                \partial_\vy^{\valpha_1} \partial_\vz^{\vbeta}\int_{B_\sigma}\ud\vx \partial_1^{\valpha_2}f(\vx, \vx - \vy, \vz)
        \right|
        \leqslant C |\vy|^{\gamma_1  + \gamma_2 + d - 2}.
        \]

        Combining the above estimates of the partial derivatives $\partial_\vy^{\valpha_1}$ over the three terms in \cref{eqn:gradient_Fy2} together then proves the 
        algebraic singularity characterization in \cref{eqn:lem_nonsmooth_02} with $|\valpha| = 2$. 
        Recursively applying the above estimates, we can validate \cref{eqn:lem_nonsmooth_02} for all derivative orders $\valpha \geqslant \bm{0}$
        and thus prove that $F(\vy, \vz)$ is smooth everywhere in $V\times V_Z$ except at $\vy = \bm{0}$ with order $\gamma_2$. 
        This finishes the proof.
\end{appendices}

\end{document}